\pdfoutput=1
\documentclass[oneside, hidelinks, 12pt, a4paper,linktoc=page,colorlinks,reqno,authoryear]{amsart}
\usepackage[toc,page]{appendix}
\usepackage{amsthm,amssymb,mathtools,bbm,tensor}
\usepackage{titling}
\usepackage[T1]{fontenc}
\usepackage[bitstream-charter]{mathdesign}
\usepackage{XCharter}
\let\mathcal\undefined
\DeclareMathAlphabet{\mathcal}{U}{dutchcal}{m}{n}
\usepackage[utf8]{inputenc}
\usepackage{enumerate,comment,braket,xspace,tikz-cd,csquotes}
\usepackage[driver=dvips,centering,vscale=0.7,hscale=0.8]{geometry}

\usepackage{microtype}
\usepackage[greek.polutoniko,english]{babel}
\usepackage[backend=biber,sorting=nyt,citestyle=ieee-alphabetic, bibstyle=alphabetic,texencoding=ascii,bibencoding=utf8,maxnames=5,minnames=3,maxbibnames=99]{biblatex}
\usepackage{setspace}
\usepackage{enumitem}
\usepackage[cal=dutchcal, scr=boondoxo]{mathalfa}
\usepackage{listings}
\usepackage{colordvi}
\usepackage{tikz,pgf}
\usetikzlibrary{arrows,decorations.pathmorphing,backgrounds,positioning,fit,matrix,calc}
\tikzset{curve/.style={settings={#1},to path={(\tikztostart)
			.. controls ($(\tikztostart)!\pv{pos}!(\tikztotarget)!\pv{height}!270:(\tikztotarget)$)
			and ($(\tikztostart)!1-\pv{pos}!(\tikztotarget)!\pv{height}!270:(\tikztotarget)$)
			.. (\tikztotarget)\tikztonodes}},
	settings/.code={\tikzset{quiver/.cd,#1}
		\def\pv##1{\pgfkeysvalueof{/tikz/quiver/##1}}},
	quiver/.cd,pos/.initial=0.35,height/.initial=0}
\tikzset{tail reversed/.code={\pgfsetarrowsstart{tikzcd to}}}
\tikzset{2tail/.code={\pgfsetarrowsstart{Implies[reversed]}}}
\tikzset{2tail reversed/.code={\pgfsetarrowsstart{Implies}}}
\usepackage{epsfig}
\usepackage{epstopdf}
\usepackage{tikz-cd}
\usepackage{fancyhdr}
\usepackage{graphicx}
\usepackage{fullpage}
\usepackage{hyperref}
\hypersetup{
	colorlinks=true,
	linkcolor=red,
	anchorcolor=black,
	filecolor=blue,
	citecolor = blue,   
	menucolor=.,   
	urlcolor=cyan,
	breaklinks=true,
}
\usepackage[capitalize]{cleveref}
\usepackage{nameref}
\usepackage{enumitem}
\crefalias{enumi}{lemman}
\usepackage{comment}
\usepackage{color}
\usepackage[totoc]{idxlayout}
\usepackage[lastexercise]{exercise}
\usepackage{xcolor}
\usepackage{breakcites}
\usetikzlibrary{positioning}

\makeatletter
\newcommand*{\doublerightarrow}[2]{\mathrel{
		\settowidth{\@tempdima}{$\scriptstyle#1$}
		\settowidth{\@tempdimb}{$\scriptstyle#2$}
		\ifdim\@tempdimb>\@tempdima \@tempdima=\@tempdimb\fi
		\mathop{\vcenter{
				\offinterlineskip\ialign{\hbox to\dimexpr\@tempdima+1em{##}\cr
					\rightarrowfill\cr\noalign{\kern.5ex}
					\rightarrowfill\cr}}}\limits^{\!#1}_{\!#2}}}
\newcommand*{\triplerightarrow}[1]{\mathrel{
		\settowidth{\@tempdima}{$\scriptstyle#1$}
		\mathop{\vcenter{
				\offinterlineskip\ialign{\hbox to\dimexpr\@tempdima+1em{##}\cr
					\rightarrowfill\cr\noalign{\kern.5ex}
					\rightarrowfill\cr\noalign{\kern.5ex}
					\rightarrowfill\cr}}}\limits^{\!#1}}}
\makeatother

\makeatletter
\def\@tocline#1#2#3#4#5#6#7{\relax
	\ifnum #1>\c@tocdepth 
	\else
	\par \addpenalty\@secpenalty\addvspace{#2}%
	\begingroup \hyphenpenalty\@M
	\@ifempty{#4}{%
		\@tempdima\csname r@tocindent\number#1\endcsname\relax
	}{%
		\@tempdima#4\relax
	}%
	\parindent\z@ \leftskip#3\relax \advance\leftskip\@tempdima\relax
	\rightskip\@pnumwidth plus4em \parfillskip-\@pnumwidth
	#5\leavevmode\hskip-\@tempdima
	\ifcase #1
	\or\or \hskip 1em \or \hskip 2em \else \hskip 3em \fi%
	#6\nobreak\relax
	\dotfill\hbox to\@pnumwidth{\@tocpagenum{#7}}\par
	\nobreak
	\endgroup
	\fi}
\makeatother

\newcommand{\infinity}{\mbox{\footnotesize $\infty$}}
\newcommand{\LL}{\mathbb{L}}

\newcommand{\Prim}{\operatorname{Prim}}
\newcommand{\triv}{\operatorname{triv}}
\newcommand{\oblv}{\operatorname{oblv}}
\newcommand{\fil}{\operatorname{fil}}

\newcommand{\Mor}{{\operatorname{Mor}}}
\newcommand{\Mormixgr}{\Mor^{\mixgr}}

\newcommand{\CoCommmixgr}{{\varepsilon\text{-}\operatorname{cCcAlg}^{\operatorname{gr}}_{\Bbbk}}}

\newcommand{\Einf}{\mathbb{E}_{\infty}}
\newcommand{\Linf}{\mathbb{L}_{\infty}}
\newcommand{\Lie}{{\operatorname{Lie}}}
\newcommand{\dgLie}{{\operatorname{dgLie}}}
\newcommand{\CEcAlgmixgr}{{\varepsilon\text{-}\operatorname{CEcAlg}^{\operatorname{gr}}_{\Bbbk}}}

\newcommand{\Lieperf}{\Lie^{\operatorname{perf}}}

\newcommand{\Rep}{{\operatorname{Rep}}}
\newcommand{\dgRep}{{\operatorname{dgRep}}}

\newcommand{\GrpChev}{{{\operatorname{Grp}}{\lp\operatorname{Chev}^{\operatorname{enh}}\rp}}}
\newcommand{\LMod}{{\operatorname{LMod}}}
\newcommand{\RMod}{{\operatorname{RMod}}}

\newcommand{\GrpLie}{{{\operatorname{Grp}}{\lp\Lie_{\Bbbk}\rp}}}

\newcommand{\op}{\operatorname{op}}

\newcommand{\End}{{\operatorname{End}}}

\newcommand{\Liemixgr}{\varepsilon\text{-}\Lie^{\gr}}
\newcommand{\Perf}{{\operatorname{Perf}}}

\newcommand{{\Cn}}{\operatorname{Cn}}
\newcommand{\Cnmixgr}{\Cn^{\gr}_{\varepsilon}}

\newcommand{\BGa}{\mathsf{B}\Ga}

\newcommand{\RR}{\mathbb{R}}
\newcommand{\colim}{\operatorname{colim}}

\newcommand{\Mapin}{{\underline{\smash{\Map}}}}

\newcommand{\scrC}{\mathscr{C}}

\newcommand{\scrD}{\mathscr{D}}
\newcommand{\Hom}{{\operatorname{Hom}}}
\newcommand{\hsp}{\hspace{0.1cm}}
\newcommand{\Tor}{{\operatorname{Tor}}}
\newcommand{\Ext}{{\operatorname{Ext}}}

\newcommand{\fib}{{\operatorname{fib}}}

\newcommand{\Spec}{{\operatorname{Spec}}}

\newcommand{\Qcoh}{{\operatorname{QCoh}}}
\newcommand{\Coh}{{\operatorname{Coh}}}
\newcommand{\IndCoh}{{\operatorname{IndCoh}}}

\newcommand{\lpd}{(\!(}
\newcommand{\rpd}{)\!)}
\newcommand{\lqd}{[\!\![}
\newcommand{\rqd}{]\!\!]}

\newcommand{\Mod}{{\operatorname{Mod}}}

\newcommand{\lp}{\left(}
\newcommand{\rp}{\right)}
\newcommand{\Catinfty}{\operatorname{Cat}_{\infty}}
\newcommand{\Ind}{{\operatorname{Ind}}}

\newcommand{\ModCEmixgr}{{{\Mod_{\CEgmixgr}}{\lp\Modmixgr_{\Bbbk}\rp}}}

\newcommand{\ModCEmixgrconst}{{{\Mod^{\operatorname{const}}_{\CEgmixgr}}{\lp\Modmixgr_{\Bbbk}\rp}}}
\newcommand*{\longhookrightarrow}{\ensuremath{\lhook\joinrel\relbar\joinrel\rightarrow}}

\makeatletter
\usepackage{etoolbox}
   \makeatletter

\appto\maketitle{%
	\let\@makefnmark\relax  \let\@thefnmark\relax
	\ifx\@empty\addresses\else\@footnotetext{%
		\vskip-\bigskipamount\@setaddresses}
}
\def\enddoc@text{}
\makeatother
\newcounter{savedchapter}
\preto\appendix{\setcounter{savedchapter}{\arabic{chapter}}}
\newcommand\resumechapters{
	\setcounter{chapter}{\arabic{savedchapter}}
	\setcounter{section}{0}
	\gdef\@chapapp{\chaptername}
	\gdef\thechapter{\@arabic\c@chapter}
}
\makeatother

\newcommand{\dgMod}{{\operatorname{dgMod}}}

\newcommand{\Gr}{{\operatorname{Gr}}}

\newcommand{\Ebb}{\mathbb{E}}

\newcommand{\Ga}{\mathbb{G}_{\operatorname{a},\Bbbk}}

\newcommand{\ZZ}{\mathbb{Z}}

\newcommand{\Map}{{\operatorname{Map}}}

\newcommand{\Sym}{{\operatorname{Sym}}}
\newcommand{\CoCommmixgraug}{{\varepsilon\text{-}\operatorname{cCcAlg}^{\operatorname{gr}}_{\Bbbk//\Bbbk}}}

\newcommand{\CoCommgraug}{\operatorname{cCcAlg}_{\Bbbk//\Bbbk}^{\gr}}

\newcommand{\Symgr}{\operatorname{Sym}^{\gr}}

\newcommand{\gfrak}{\mathfrak{g}}
\newcommand{\hfrak}{\mathfrak{h}}
\newcommand{\Lfrak}{\mathfrak{L}}

\newcommand{\Ug}{{\operatorname{U}}{(\gfrak)}}
\newcommand{\Umixgr}{{\operatorname{U}^{\gr}_{\varepsilon}}}

\newcommand{\UCngmixgr}{{\Umixgr}{\lp\Cnmixgr(\gfrak)\rp}}

\newcommand{\CE}{{\operatorname{CE}}}

\newcommand{\CEgmixgr}{{\CE}^{\varepsilon}{\lp\gfrak\rp}}

\newcommand{\Fun}{{\operatorname{Fun}}}

\newcommand{\CAlg}{{\operatorname{CAlg}}}
\newcommand{\dgCAlg}{{\operatorname{dgCAlg}}}

\newcommand{\Alg}{{\operatorname{Alg}}}
\newcommand{\dgAlg}{{\operatorname{dgAlg}}}

\newcommand{\Modgr}{\Mod^{{\operatorname{gr}}}}
\newcommand{\gr}{{\operatorname{gr}}}
\newcommand{\Modmixgr}{\varepsilon\operatorname{-}\Mod^{{\operatorname{gr}}}}
\newcommand{\Modmixgrcn}{\varepsilon\operatorname{-}\Mod^{{\operatorname{gr},\geqslant0}}}
\newcommand{\Perfmixgrcn}{\varepsilon\operatorname{-}\Perf^{{\operatorname{gr},\geqslant0}}}
\newcommand{\Modmixgrccn}{\varepsilon\operatorname{-}\Mod^{{\operatorname{gr},\leqslant0}}}
\newcommand{\Perfmixgrccn}{\varepsilon\operatorname{-}\Perf^{{\operatorname{gr},\leqslant0}}}
\newcommand{\mixgr}{\varepsilon\operatorname{-gr}}

\newcommand{\Algmixgr}{\varepsilon\operatorname{-}\Alg^{{\operatorname{gr}}}}
\newcommand{\CAlgmixgr}{\varepsilon\operatorname{-}\CAlg^{{\operatorname{gr}}}}
\newcommand{\CAlgmixgraug}{\varepsilon\operatorname{-}\CAlg_{\Bbbk//\Bbbk}^{{\operatorname{gr}}}}

\newcommand{\CAlggraug}{\CAlg_{\Bbbk//\Bbbk}^{{\operatorname{gr}}}}
\newcommand{\LModmixgr}{\varepsilon\operatorname{-}\LMod^{{\operatorname{gr}}}}
\newcommand{\RModmixgr}{\varepsilon\operatorname{-}\RMod^{{\operatorname{gr}}}}

\newcommand{\Modfil}{\Mod^{{\operatorname{fil}}}}

\newcommand{\otimesmixgr}{\otimes^{\varepsilon\text{-}\gr}}

\newcommand{\Mapmixgr}{{\Mapin}^{\mixgr}}
\numberwithin{equation}{subsection}
\theoremstyle{plain}

\newtheorem{lemman}[equation]{Lemma}
\newtheorem{propositionn}[equation]{Proposition}
\newtheorem{corollaryn}[equation]{Corollary}
\newtheorem{conjn}[equation]{Conjecture}
\newtheorem*{theoremn}{Theorem}

\theoremstyle{definition}
\newtheorem{defn}[equation]{Definition}
\newtheorem{parag}[equation]{}

\newtheorem{remark}[equation]{Remark}
\newtheorem{construction}[equation]{Construction}
\newtheorem{porism}[equation]{Porism}
\newtheorem{warning}[equation]{Warning}
\newtheorem{notation}[equation]{Notation}

\newtheorem{exmp}[equation]{Example}

\newcommand\restr[2]{{
		\left.\kern-\nulldelimiterspace
		#1
		\vphantom{\|}
		\right|_{#2} 
}}
\usepackage{lipsum}
\usepackage{fancyhdr}
\usepackage{authblk}
\providecommand{\abstract}{}
\usepackage{abstract}
\DeclareGraphicsExtensions{.png,.jpg,.pdf,.mps}
\usepackage{fancyhdr}
\usepackage{lastpage}
\usepackage{hyperref}
\hypersetup{
	colorlinks=true,
	linkcolor=red,
	anchorcolor=black,
	filecolor=blue,
	citecolor = blue,   
	menucolor=.,   
	urlcolor=cyan,
	breaklinks=true,
}
\providecommand{\keyword}[1]{\textbf{\textit{Key words and phrases---}} #1}
\begin{document}

 \title{Mixed graded structure on Chevalley-Eilenberg functors}
\author{Emanuele Pavia}
\address{Università degli Studi di Milano}
\email{emanuele.pavia@unimi.it}

	\maketitle
	\begin{abstract}
In this paper, we shall provide a purely \infinity-categorical construction of the mixed graded structure (in the sense of \cite{PTVV,CPTVV}) over Chevalley-Eilenberg complexes computing homology and cohomology of Lie algebras defined over a field $\Bbbk$ of characteristic $0$. While this additional piece of structure on Chevalley-Eilenberg complexes is expected and described in \cite{CG} in terms of explicit models, there is not a formal and model independent description of the mixed graded Chevalley-Eilenberg \infinity-functors in available literature. After constructing in all details the Chevalley-Eilenberg \infinity-functors and studying their main formal properties, we present some further conjectures on their behavior.
	\end{abstract}
\keyword{Mixed graded modules, algebras and coalgebras, Chevalley-Eilenberg complexes, Lie algebras, homology and cohomology.}
\tableofcontents
\section*{Introduction}
\addtocontents{toc}{\protect\setcounter{tocdepth}{0}}
\subsection*{Motivations}
Given a classical Lie algebra $\gfrak$ defined over a base field of characteristic $0$ together with some fixed representation $M$, one can consider its homology $\operatorname{H}_{\bullet}(\gfrak;\hsp M)$ and cohomology $\operatorname{H}^{\bullet}(\gfrak;\hsp M)$ with values in $M$. It is well known that these gadgets can be seen under many different, but equivalent, perspectives.
\begin{enumerate}
\item The $n$-th homology $\Bbbk$-vector space $\operatorname{H}_{n}(\gfrak;\hsp M)$ is the $n$-th left derived functor of the functor sending a representation $M$ of $\gfrak$ to its invariant sub-$\Bbbk$-vector space $M^{\gfrak}$. Analogously, the $n$-th cohomology $\Bbbk$-vector space $\operatorname{H}^n(\gfrak;\hsp M)$ is the $n$-th right derived functor of the functor sending $M$ to the $\Bbbk$-vector space of coinvariants $M/\gfrak\cdot M$.
\item The homology $\operatorname{H}_{\bullet}(\gfrak;\hsp M)$ is also computed by $\Tor^{\Ug}_{\bullet}{\lp \Bbbk,\hsp M\rp}$, where $\Ug$ is the universal enveloping algebra of $\gfrak$ and $\Bbbk$ is the trivial representation of $\gfrak$. Analogously, the cohomology $\operatorname{H}^{\bullet}(\gfrak;\hsp M)$ is computed by $\Ext^{\bullet}_{\Ug}{\lp\Bbbk,\hsp M\rp}$. 
\end{enumerate}
The latter point of view is particularly interesting, since $\Bbbk$ admits a well known projective resolution as a $\Ug$-module (see \cite[Theorem $7.7.2$]{weibel}) given by the graded $\Bbbk$-vector space $$\operatorname{V}_{\bullet}(\gfrak)\coloneqq\Ug\otimes_{\Bbbk}\bigwedge\nolimits^{\!\bullet}\gfrak$$endowed with differential\begin{align*}
u\otimes\lp g_1\wedge\ldots\wedge g_n\rp&\mapsto\sum_{1\leqslant i\leqslant n}(-1)^{i+1}u\cdot g_i\otimes\lp g_1\wedge\ldots\wedge\widehat{g}_i\wedge\ldots\wedge g_n\rp\\&+\sum_{1\leqslant i<j\leqslant n}(-1)^{i+j}u\otimes \lp g_1\wedge\ldots\wedge\widehat{g}_i\wedge\ldots\wedge g_{j-1}\wedge[g_i,g_j]\wedge\ldots\wedge g_n\rp.
\end{align*}In particular, one can compute the homology and the cohomology of the Lie algebra $\gfrak$ with coefficients in a representation $M$ by considering respectively the honest (underived) tensor product $\CE_{\bullet}(\gfrak;\hsp M)\coloneqq M\otimes_{\Ug}\operatorname{V}_{\bullet}(\gfrak)$ and the hom-vector space $\CE^{\bullet}(\gfrak;\hsp M)\coloneqq \Hom_{\Ug}{\lp \operatorname{V}_{\bullet}(\gfrak),\hsp M\rp}$, and taking their homology. When $M$ is the trivial representation $\Bbbk$, then the complexes $\CE_{\bullet}(\gfrak;\hsp \Bbbk)$ and $\CE^{\bullet}(\gfrak;\hsp \Bbbk)$ are respectively a coaugmented cocommutative differential graded coalgebra and an augmented commutative differential graded algebra. Moreover, for any other representation $M$ of $\gfrak$, the chain complex $\CE^{\bullet}(\gfrak;\hsp M)$ is naturally a differential graded $\CE^{\bullet}(\gfrak;\hsp \Bbbk)$-module. Actually, there is more: the Chevalley-Eilenberg coalgebra functor determines an equivalence of categories between the category of Lie algebras over $\Bbbk$ and the differential graded cocommutative coalgebras over $\Bbbk$ which are semi-free over generators sitting in homological degree $1$ (i.e., which are isomorphic to $\bigwedge^{\bullet}V$ for some $\Bbbk$-vector space $V$ as graded coalgebras). The analogous statement holds also for the Chevalley-Eilenberg algebra functor, if one considers \textit{finitely generated} Lie algebras.\\

In the derived setting, where one considers chain complexes only up to quasi-isomorphism, this assertions cannot hold: already in the discrete case, there are plenty of Lie algebras which are not isomorphic one to the other, having the same homology or cohomology. Indeed, the classical statement takes into consideration not only the homology, but also the \textit{filtration} on the homology; we actually need to consider the Chevalley-Eilenberg algebras and coalgebras as objects in the derived filtered category of Beilinson of \cite{beilinsonderived} (or, better, in its stable enhancement presented in \cite{BMS}).\\

Recently, the research in derived deformation theory and derived differential geometry has been working extensively with \textit{mixed graded modules}, which are equivalent to modules with an action of the circle $S^1\coloneqq\mathsf{B}\ZZ$ when working over a base ring of characteristic $0$. The theory of mixed graded modules in characteristic $0$, which was developed in \cite{PTVV} and \cite{CPTVV}, has the concrete computational advantage that it strongly resembles a theory of \textit{double complexes}, with an internal homological degree presenting some homotopical and purely derived datum and an external differential collecting more geometrical meaning. Indeed, while the treatment of mixed graded modules can be carried out completely internally to the theory of \infinity-categories, the definition of the \infinity-category of mixed graded modules of \cite{PTVV,CPTVV} is given in terms of a Dwyer-Kan localization of a model category of \textit{mixed graded chain complexes}, which are very close to \textit{double chain complexes}. at the class of weight-wise weak equivalences: this is the approach that we follow in this work, see \cref{def:mixedgraded}.\\

Recently, the theory of mixed graded modules has also been linked to the usual derived filtered category of Beilinson in \cite{UHKR}, \cite{toen2020algebraic, derivedfoliations2}, \cite{calaque2021lie}, and \cite{pavia1}. Because of this, in the last years mixed graded modules have been employed also in the homotopy theory of Lie algebras and Lie algebroids, and it is expected that mixed graded modules can provide a natural setting where to work with formal geometry and deformation theory. In particular, it has been known for a while that given a differential graded Lie algebra $\gfrak_{\bullet}$, its Chevalley-Eilenberg algebra has a richer structure of \textit{mixed graded commutative algebra}. In some sense, the differential of the classical Chevalley-Eilenberg algebra of a differential graded Lie algebra $\gfrak_{\bullet}$ contains a homotopical part, coming from the differential of the Lie algebra and which preserves the symmetric power, and a "more geometric" component that keeps track of the Lie bracket of $\gfrak_{\bullet}$: a feasible example of this philosophy is explicitly provided in \cref{example:differential}. This folklore result, expected by many researchers, is still not available in existing literature. Works such as \cite{CG} and \cite{nuiten19} define the mixed graded structure on the Chevalley-Eilenberg algebra in terms of mixed differential on explicit graded chain complexes models; while being very clear when dealing with differential graded Lie algebras, such construction is not obviously functorial and well defined when one deals with $\Linf$-algebras whose brackets have to respect higher and more complicated homotopy coherences. The aim of this work is to fill this gap in the existing literature, carrying out a completely \infinity-categorical construction of the mixed graded structure on the Chevalley-Eilenberg algebra.
\subsection*{Outline of the paper}
In \cref{chapter:mixedgradedmodules} we collect the main definitions and notations regarding mixed graded modules over an arbitrary base ring of characteristic $0$, following the conventions of \cite{pavia1} (hence, for the most part, the ones of \cite{PTVV} and \cite{CPTVV}). The main addition to existing literature is Section \ref{sec:algcoalg}, concerning algebras and coalgebras in the \infinity-category of mixed graded $\Bbbk$-modules for arbitrary (unital and augmented) operads and cooperads. In \cref{sec:dglie}, we recall the main definitions and notations concerning derived Lie algebras in characteristic $0$ (including the classical theory of Chevalley-Eilenberg complexes of discrete Lie algebras), following \cite{dagx}. The core of the article is \cref{chapter:liealgebras}, where we construct with purely \infinity-categorical techniques both homological (\cref{sec:homologicalCE}) and cohomological (\cref{sec:cohomCE}) mixed graded versions of the Chevalley-Eilenberg complexes, which provide a mixed graded cocommutative coalgebra and a mixed graded commutative algebra respectively. This has to be interpreted as a model independent construction of the mixed graded structure on the Chevalley-Eilenberg cohomological complex which has been provided, in terms of explicit mixed differential for a graded chain complex, in \cite[Appendix B]{CG}. Moreover, in \cref{sec:CEmod} we constructed a mixed graded version of the Chevalley-Eilenberg modules of a homotopy Lie algebra with coefficients in some representation. While this description of the Chevalley-Eilenberg complexes is known, to our knowledge there is no description of this functor in a model independent way in existing literature. Moreover, our construction highlights some important features of the Chevalley-Eilenberg functors, such as the preservation of homotopy limits and/or colimits. We can sum up the results of the second section as follows.
\begin{theoremn}[Propositions \ref{prop:promotionCEcocommutativealgebra}, \ref{prop:CEcohomologicalalgebra}, \ref{prop:CEmod}]\
	\begin{enumerate}
		\item There exists a covariant \textit{homological Chevalley-Eilenberg mixed graded coalgebra} \infinity-functor$$\CE_{\varepsilon}\colon\Lie_{\Bbbk}\longrightarrow\CoCommmixgraug$$from the \infinity-category of Lie algebras to the \infinity-category of coaugmented mixed graded cocommutative coalgebras, such that for any Lie algebra $\gfrak$ the underlying graded cocommutative coalgebra of $\CE_{\varepsilon}(\gfrak)$ is equivalent to the graded symmetric coalgebra on $\gfrak[-1]$, with $\gfrak[-1]$ sitting in weight $1$. Moreover, the totalization of $\CE_{\varepsilon}(\gfrak)$ agrees with the classical Chevalley-Eilenberg homology of $\gfrak$.
		\item There exists a contravariant \textit{cohomological Chevalley-Eilenberg mixed graded algebra} \infinity-functor$$\CE^{\varepsilon}\colon\Lie_{\Bbbk}^{\op}\longrightarrow\CAlgmixgraug$$from the \infinity-category of Lie algebras to the \infinity-category of augmented mixed graded commutative algebras, such that for any Lie algebra $\gfrak$ the underlying graded commutative algebra of $\CE^{\varepsilon}(\gfrak)$ is equivalent to the graded symmetric algebra on $\gfrak^{\vee}[1]$, with $\gfrak^{\vee}[1]$ sitting in weight $-1$. Moreover, the totalization of $\CE^{\varepsilon}(\gfrak)$ agrees with the classical Chevalley-Eilenberg cohomology of $\gfrak$.
		\item Fixing a Lie algebra $\gfrak$ which is perfect as a $\Bbbk$-module, there exists a covariant \textit{Chevalley-Eilenberg cohomology} \infinity-functor$$\CE^{\varepsilon}(\gfrak;-)\colon\LMod_{\Ug}\longrightarrow\ModCEmixgr$$from the \infinity-category of representations of $\gfrak$ to the \infinity-category of mixed graded $\Bbbk$-modules which are modules for the commutative algebra object $\CE^{\varepsilon}(\gfrak)$, such that for any representation $M$ of $\gfrak$ the underlying graded module of $\CE^{\varepsilon}(\gfrak;\hsp M)$ is equivalent to the (graded) tensor product of the symmetric algebra over $\gfrak^{\vee}[1]$ with $M$, with naturally induced grading. 
	\end{enumerate}
\end{theoremn}
Finally, in \cref{chapter:conj}, we describe some conjectures regarding the fully faithfulness of the Chevalley-Eilenberg \infinity-functors (Conjectures \ref{conj:main}, \ref{conj:main2}, \ref{conj:modules}) and their essential image, hinting at possible strategies to tackle the proof. 
\subsection*{Possible future developments of the theory}
The work of this paper started by investigating possible frameworks where to study formal derived algebraic geometry. In particular, we expect that our constructions can be carried out, with only minor modifications, in the more general setting of Chevalley-Eilenberg complexes of Lie algebroids. Moreover, the proof of one of our conjectures (namely, \cref{conj:modules}) should provide an insight to the problem of studying $\scrD$-modules for possibly non-smooth derived schemes. In particular, in \cite{pantev2021moduli} the authors show that given the de Rham algebra $\operatorname{dR}(X)$ of some smooth, affine $\Bbbk$-scheme $X\coloneqq\Spec(A)$, with its natural mixed structure given by the de Rham differential, then there exists a fully faithful embedding$$\Mod_{\scrD_X}\longhookrightarrow\Mod_{\operatorname{dR}(X)}{\lp\Modmixgr_{\Bbbk}\rp}$$whose essential image is spanned by objects $M_{\bullet}$ which are equivalent, at the level of the underlying graded object, to $M\otimes_{A}\Sym_{\Bbbk}{\lp\LL_{X/\Bbbk}[1](-1)\rp}$ for some $A$-module $M$. This leads naturally to seek for an analogous result also for (possibly non-smooth) derived affine schemes, using the computation of the left adjoint to the inclusion of constant mixed graded modules over the Chevalley-Eilenberg algebra into all mixed graded modules described in \cref{sec:conjrep}.
\subsection*{Acknowledgments}I am very thankful to my PhD advisor, M. Porta, who introduced me to the theory of mixed graded complexes. I would also like to thank F. Battistoni, I. Di Liberti, A. Gagna, and G. Nocera for fruitful discussions and suggestions which helped shaping the content of this paper. Finally I would like D. Calaque and M. Robalo for precious comments and their careful refereeing.  
\subsection*{Notations, conventions and main references}\
\begin{itemize}
	\item Throughout all this paper, we employ freely the language of derived algebraic geometry, \infinity-categories, and homotopical algebra provided by \cite{htt} and \cite{ha}, from which we borrow the formalism and most notations. Our language is \textit{innerly} derived: every definition and construction has to be interpreted, without further indication suggesting the contrary, in the context of higher algebra. In particular, by \textit{module} over a discrete commutative ring $\Bbbk$ we mean an object of the stable derived \infinity-category of $\Bbbk$-modules, by \textit{limits and colimits} we mean homotopy limits and colimits, by \textit{tensor product} we mean derived tensor product, and so forth.
	\item Our main references for the homotopy theory of mixed graded complexes are provided by \cite{PTVV} and \cite{CPTVV}.
	\item Our main references for the homotopy theory of Lie algebras in characteristic $0$, and its relationship with the derived deformation theory, are provided by \cite{dagx} and \cite{studyindag2}.
	\item When dealing with explicit models provided by chain complexes of $\Bbbk$-modules, we use a homological notation. 
	\item Our standing assumption is that we work over a fixed base field $\Bbbk$ which is assumed to be of characteristic $0$.
	\item Throughout this paper, we shall often work with closed symmetric monoidal \infinity-categories $\scrC^{\otimes}$ enriched over $\Bbbk$-modules. In particular, such \infinity-categories are endowed with both internal mapping objects, obtained as a right adjoint to $\otimes$, and mapping $\Bbbk$-modules providing the enrichment over $\Mod_{\Bbbk}$. In order to avoid confusion, we shall denote the former with $\underline{\smash{\Map}}$ (or, occasionally, with $\Mor$) and the latter with $\Map$.
\end{itemize}
\addtocontents{toc}{\protect\setcounter{tocdepth}{2}}
\section{Mixed graded modules}
\label{chapter:mixedgradedmodules}
In this first section we shall review the main definitions and notations concerning mixed graded $\Bbbk$-modules defined over a base field of characteristic $0$. Everything presented here is already available in existing literature: see for example \cite{PTVV,CPTVV,pavia1}. The only partially original material is \cref{sec:algcoalg}, which studies the monoidal structure on many relevant \infinity-functors defined over mixed graded $\Bbbk$-modules, and describes which \infinity-functors can be defined at the level of algebras or coalgebras. In any case, this subsection contains little profound mathematics; the only non-trivial statement is probably \cref{prop:gradedsymfullyfaithful}, which states that the \textit{graded} symmetric coalgebra \infinity-functor is fully faithful.
\subsection{Main definitions and notations}
\label{sec:basicmixedgraded}
This subsection serves mainly as a quick recollection on definitions, notations and conventions adopted throughout the paper. For additional discussions and proofs, see \cite{PTVV,CPTVV} and especially \cite[Section $1.1$]{pavia1}, of which this subsection is actually a brief summary.
\begin{parag}
Let $\varepsilon\text{-}\dgMod_{\Bbbk}^{\gr}$ be the category of chain complexes of $\Bbbk$-modules $M_{\bullet}$ with a \textit{weight decomposition} of sub-complexes $\left\{ \lp M_{\bullet}\rp_p\right\}_{p\in\ZZ}$ and with a \textit{mixed differential}, i.e., of chain complexes $$\varepsilon_p\colon \lp M_{\bullet}\rp_p\longrightarrow \lp M_{\bullet}\rp_{p-1}[-1]$$such that $\varepsilon_{p-1}[-1]\circ \varepsilon_p = 0$ for all $p\in\ZZ$, with obvious morphisms between them. This category is endowed with a cofibrantly generated symmetric monoidal model category, where the class $\mathscr{W}$ of weak equivalences and the class $\mathscr{F}$ of fibrations are respectively given by the classes of quasi-isomorphisms and surjections in each weight (\cite{PTVV}).
\end{parag}
\begin{defn}
\label{def:mixedgraded}
The \textit{$\infinity$-category of mixed graded $\Bbbk$-modules} $\Modmixgr_{\Bbbk}$ is the simplicial nerve of the Dwyer-Kan localization along the class of weak equivalences ${\operatorname{N}_{\Delta}}{\lp\operatorname{L}^{\operatorname{H}}\lp\varepsilon\text{-}\dgMod_{\Bbbk}^{\gr},\mathscr{W}\rp\rp}$.
\end{defn}
\begin{parag}
	\label{parag:monoidalstructureonmodmixgr}
	The $\infinity$-category of mixed graded $\Bbbk$-modules admits all limits and colimits and they are computed weight-wise (\cite[Lemma $1.1.25$]{pavia1}), and in particular it is stable. Moreover, it is endowed with a symmetric closed monoidal structure which we can describe as follows (see also \cite[Section $1.1$]{CPTVV}).\begin{enumerate}
		\item Given two mixed graded $\Bbbk$-modules $M_{\bullet}$ and $N_{\bullet}$, the tensor product $M_{\bullet}\otimesmixgr_{\Bbbk}N_{\bullet}$ is the mixed graded $\Bbbk$-module whose $p$-th weight component is given by the formula$$\lp M_{\bullet}\otimesmixgr_{\Bbbk} N_{\bullet}\rp_p\coloneqq\bigoplus_{i+j=p} M_i\otimes_{\Bbbk}N_j$$with mixed differential defined on every summand by the formula$$\varepsilon_M\otimes\operatorname{id}_N+\operatorname{id}_M\otimes\varepsilon_N\colon M_i\otimes N_j\longrightarrow\lp M_{i-1}\otimes N_j\rp \bigoplus \lp M_i\otimes N_{j-1}\rp[-1].$$
		\item The unit for $\otimesmixgr_{\Bbbk}$ is the mixed graded $\Bbbk$-module $\Bbbk(0)$, consisting of $\Bbbk$ sitting in pure weight $0$ with trivial mixed structure.
		\item Given two mixed graded $\Bbbk$-modules $M_{\bullet}$ and $N_{\bullet}$, the internal mapping space ${\Mapmixgr_{\Bbbk}}{\lp M_{\bullet},\hsp N_{\bullet}\rp}$ is the mixed graded $\Bbbk$-module whose $p$-th weight component is given by the formula$$\lp{\Mapmixgr_{\Bbbk}}{\lp M_{\bullet},\hsp N_{\bullet}\rp}\rp_p\coloneqq\prod_{q\in\ZZ}{{\Map_{\Mod_{\Bbbk}}}{\lp M_{q},\hsp N_{q+p}\rp}}$$with mixed differential$$\varepsilon_p\colon {\Mapmixgr_{\Bbbk}}{\lp M_{\bullet},\hsp N_{\bullet}\rp}_p\longrightarrow {\Mapmixgr_{\Bbbk}}{\lp M_{\bullet},\hsp N_{\bullet}\rp}_{p-1}[-1]$$given by the sum of the post-composition of the mixed differential for $N_{\bullet}$ and the pre-composition of the mixed differential for $M_{\bullet}$.
	\end{enumerate}
	The enrichment of $\Modmixgr_{\Bbbk}$ over $\Mod_{\Bbbk}$ is then given by$$\Map_{\Modmixgr_{\Bbbk}}{\lp M_{\bullet},\hsp N_{\bullet}\rp}\coloneqq\fib\lp\Mapmixgr_{\Bbbk}{\lp M_{\bullet},\hsp N_{\bullet}\rp}_0\overset{\varepsilon_0}{\longrightarrow}\Mapmixgr_{\Bbbk}{\lp M_{\bullet},\hsp N_{\bullet}\rp}_{-1}[-1]\rp.$$
\end{parag}
\begin{parag}
	\label{parag:modeladjunction}
For any integer $p$, we can consider the \infinity-functor $(-)(p)\colon\Mod_{\Bbbk}\to\Modmixgr_{\Bbbk}$ sending a $\Bbbk$-module $M$ to the mixed graded $\Bbbk$-module $M(p)$ consisting of $M$ sitting in pure weight $p$ and with trivial mixed structure. For $p=0$, we have an adjunction
	\begin{align}
	\label{adjunction:realizationadjunction}
	\begin{tikzpicture}[scale=0.75,baseline=-0.5ex]
	\node at (-3.1,0){$(-)(0)\colon$};
	\node (a) at (-1.5,0){$\Mod_{\Bbbk}$};
	\node (b) at (1.5,0){$\Modmixgr_{\Bbbk}$};
	\node at (3.1,-0){$: \left|-\right|$};
	\draw[->] ([yshift=3.5pt]a.east) -- ([yshift=3.5pt]b.west);
	\draw[->] ([yshift=-3.5pt]b.west) -- ([yshift=-3.5pt]a.east);
	\end{tikzpicture}
	\end{align}
	where the left adjoint simply sends a $\Bbbk$-module $M$ to the mixed graded $\Bbbk$-module consisting of $M$ concentrated in weight $0$, and the right adjoint $\left|-\right|$ is the $\infinity$-functor that sends a mixed graded ${\Bbbk}$-module $M_{\bullet}$ to the mapping $\Bbbk$-module$$\left|M_{\bullet}\right|\coloneqq\Map_{\Modmixgr_{\Bbbk}}{\lp {\Bbbk}(0),\hsp M_{\bullet}\rp}.$$This right adjoint is called the \textit{realization $\infinity$-functor}: a strict model of $\left|M_{\bullet}\right|$ is provided by the chain complex of ${\Bbbk}$-modules $$\prod_{p\geqslant 0}M_{-p}[-2p]$$endowed with the total differential, sum of the usual differential of chain complexes and the mixed differential (\cite[Proposition $1.5.1$]{CPTVV}).
\end{parag}
	For our purposes, it will be convenient to introduce another realization $\infinity$-functor that can keep track of the ${\Bbbk}$-modules in positive weights of a mixed graded module $M_{\bullet}$. Let us recall (\cite[Section $1.5$]{CPTVV}) that for all $p\in\ZZ$ we have that
	\begin{align}
	\label{map}
	{\Map_{\Modmixgr_{\Bbbk}}}{\lp {\Bbbk}(i)[-2i],\hsp {\Bbbk}(i-1)[-2(i-1)]\rp}\simeq {\Bbbk}.
	\end{align}
	So we have a pro-object in $\Modmixgr_{\Bbbk}$, defined by
	\begin{align}
	\label{taterealizationexplicit}
	{\Bbbk}(\infinity) \coloneqq\left\{\ldots\to {\Bbbk}(i)[-2i]\to {\Bbbk}(i-1)[-2(i-1)]\to\ldots \to {\Bbbk}(1)[-2]\to {\Bbbk}(0)\right\}
	\end{align}
	where the morphism ${\Bbbk}(i)[-2i]\to {\Bbbk}(i-1)[-2(i-1)]$ is the unique morphism corresponding to the unit $1$ of ${\Bbbk}$ under the equivalence \ref{map}.
\begin{defn}[{\cite[Definition $1.5.2$]{CPTVV}}]
	\label{def:taterealization}
	The \textit{Tate} or \textit{stabilized realization} $\infinity$-functor is defined as
	\begin{displaymath}
	\left|-\right|^{\operatorname{t}}\coloneqq{\Map_{\Modmixgr_{\Bbbk}}}{\lp {\Bbbk}(\infinity),\hsp -\rp}\colon \Modmixgr_{\Bbbk}\longrightarrow \Ind\lp\Mod_{\Bbbk}\rp\overset{\colim}{\longrightarrow}\Mod_{\Bbbk}.
	\end{displaymath}
\end{defn}
Working with explicit models given by graded chain complexes and mixed differentials, the $\infinity$-functor of Definition \ref{def:taterealization} sends a mixed graded ${\Bbbk}$-module $M_{\bullet}=\{M_p\}_p$ to the ${\Bbbk}$-module \begin{align}
			\label{parag:tateexplicit}
			\left|	M_{\bullet}\right|^{\operatorname{t}}\coloneqq\underset{i\leqslant 0}{\colim}\prod_{p\geqslant i}M_{-p}[-2p]
				\end{align}
again endowed with the total differential.
\begin{parag}
	\label{parag:freeleftadjointmixedgraded}
	For all $p\in\ZZ$ we have a \textit{$p$ weight part $\infinity$-functor}
	\begin{displaymath}
	(-)_p\colon \Modmixgr_{\Bbbk}\longrightarrow\Mod_{\Bbbk}
	\end{displaymath}
	which selects the component in weight $p$ of a mixed graded $\Bbbk$-module. For $p=0$, the $\infinity$-functor $(-)_0$ has a left adjoint \begin{align}
	\label{functor:0weightadjoint}
	\operatorname{Free}_{\varepsilon}\colon \Mod_{\Bbbk}\longrightarrow\Modmixgr_{\Bbbk}
	\end{align} given by the free mixed graded ${\Bbbk}$-module construction, see \cite[$1.4.1$]{CPTVV}. By pre-composing $(-)_0$ with the weight shift by $q$ on the left $\infinity$-endofunctor
	\begin{displaymath}
	(-)\lpd q\rpd \coloneqq-\otimesmixgr_{\Bbbk} {\Bbbk}(q)\colon \Modmixgr_{\Bbbk}\longrightarrow\Modmixgr_{\Bbbk}
	\end{displaymath}
	which informally sends a mixed graded ${\Bbbk}$-module $M_{\bullet}=\left\{M_p\right\}_{\!p\in\ZZ}$ to the mixed graded ${\Bbbk}$-module $M\lpd q\rpd_{\bullet}\coloneqq\left\{M_{p-q}\right\}_{\!p\in\ZZ}$, and by post-composing $\operatorname{Free}_{\varepsilon}$ with the shift $\infinity$-functor $\lpd q\rpd$ we obtain for all $q\in\ZZ$ another adjunction 
	\begin{align}
	\label{adjunction:weightadjunction}
	\begin{tikzpicture}[scale=0.75,baseline=0.5ex]
	\node at (-3.5,0){$\operatorname{Free}_{\varepsilon}\lpd q\rpd \colon$};
	\node (a) at (-1.5,0){$\Mod_{\Bbbk}$};
	\node (b) at (1.5,0){$\Modmixgr_{\Bbbk}$};
	\node at (3.4,-0){$\colon(-)_q$};
	\draw[->] ([yshift=3pt]a.east) -- ([yshift=3pt]b.west);
	\draw[->] ([yshift=-3pt]b.west) -- ([yshift=-3pt]a.east);
	\end{tikzpicture}
	\end{align}
	which generalizes the adjunction \ref{functor:0weightadjoint} to all weights. The description of this left adjoint is  very explicit: indeed, the proof of \cite[Proposition $1.3.8$]{CPTVV} together with the observations made in \cite[Section $1.4.1$]{CPTVV} shows that such left adjoint simply sends a $\Bbbk$-module $M$ to the mixed graded $\Bbbk$-module $\operatorname{Free}_{\varepsilon}(M)\lpd q\rpd$ consisting of $M$ in weight $q$, $M[1]$ in weight $q-1$, and with mixed structure given by the natural equivalence $M\simeq M[1][-1]$.
\end{parag}
\begin{remark}
	\label{porism:rightadjointforgetful}
	We can actually describe the free mixed graded construction of \ref{parag:freeleftadjointmixedgraded} in a more general way. Since both $\Modmixgr_{\Bbbk}$ and $\Modgr_{\Bbbk}$ are presentable $\infinity$-categories, the forgetful $\infinity$-functor $\oblv_{\varepsilon}\colon\Modmixgr_{\Bbbk}\to\Modgr_{\Bbbk}$ which forgets the mixed structure is both a left and right adjoint in virtue of the Adjoint Functor Theorem (\cite[Corollary $5.5.2.9$]{htt}). Its adjoints $\operatorname{L}_{\varepsilon}$ and $\operatorname{R}_{\varepsilon}$ can be described with explicit models in the following way.
	\begin{itemize}
		\item The left adjoint $\operatorname{L}_{\varepsilon}$ sends a graded $\Bbbk$-module $M_{\bullet}$ to the mixed graded $\Bbbk$-module $\operatorname{L}_{\varepsilon}M_{\bullet}$ defined in weight $p$ by the formula$$\lp\operatorname{L}_{\varepsilon}M_{\bullet}\rp_p\simeq M_p\oplus M_{p+1}[1],$$and whose $\BGa$-action is described by the morphism $$\varepsilon_p\colon M_p\oplus M_{p+1}[1]\longrightarrow\lp M_{p-1}\oplus M_{p-1+1} [1]\rp [-1]\simeq M_{p-1}[-1]\oplus M_p$$given by the canonical equivalence $M_p\simeq M_p[1][-1]$ and the zero map on $M_{p+1}[1]$.
		\item The right adjoint $\operatorname{R}_{\varepsilon}$ sends a graded $\Bbbk$-module $M_{\bullet}$ to the mixed graded $\Bbbk$-module $\operatorname{R}_{\varepsilon} M_{\bullet}$ defined in weight $p$ by the formula $\lp\operatorname{R}_{\varepsilon}M_{\bullet}\rp_p\simeq M_p\oplus M_{p-1}[-1]$, and whose $\BGa$-action is described by the morphism
		\begin{displaymath}
		\varepsilon_p\colon M_p\oplus M_{p-1}[-1]\longrightarrow \lp M_{p-1}\oplus M_{p-2}[-1]\rp [-1]\simeq M_{p-1}[-1]\oplus M_{p-2}[-2]
		\end{displaymath}
		given by the canonical equivalence $M_{p-1}[-1]\simeq M_{p-1}[-1]$ and the zero map on $M_p$.
	\end{itemize}
	Since for any integer $q$ one has an ambidextrous adjunction$$
	\begin{tikzpicture}[scale=0.75,baseline=0.5ex]
	\node at (-3.2,0){$(-)(q)\colon$};
	\node (a) at (-1.5,0){$\Mod_{\Bbbk}$};
	\node (b) at (1.5,0){$\Modgr_{\Bbbk}$};
	\node at (3.1,-0){$\colon(-)_q$};
	\draw[->] ([yshift=3pt]a.east) -- ([yshift=3pt]b.west);
	\draw[->] ([yshift=-3pt]b.west) -- ([yshift=-3pt]a.east);
	\end{tikzpicture}$$
	which is an ambidextrous adjunction, one can see that the  $\infinity$-functor $\operatorname{Free}_{\varepsilon}\lpd q\rpd$ of \ref{adjunction:weightadjunction} is canonical equivalent to the composition $$\operatorname{L}_{\varepsilon}\circ\hsp(-)(q)\colon\Mod_{\Bbbk}\longrightarrow\Modmixgr_{\Bbbk}.$$
\end{remark}
\begin{parag}
	\label{construction:trivialmixedstructure}
Given a graded $\Bbbk$-module $M_{\bullet}$, one can equip it with the trivial mixed structure described by the zero mixed differential. This construction yields an \infinity-functor$$\triv_{\varepsilon}\colon\Modgr_{\Bbbk}\longrightarrow\Modmixgr_{\Bbbk},$$which is a section of the forgetful $\infinity$-functor $\oblv_{\varepsilon}\colon\Modmixgr_{\Bbbk}\to\Modgr_{\Bbbk}$, i.e., 
	\begin{align}
\label{equivalence:oblvtriv}
\operatorname{id}_{\Modgr_{\Bbbk}}\simeq\oblv_{\varepsilon}\circ\triv_{\varepsilon}.
\end{align}Let us remark that $\triv_{\varepsilon}$, even if it is obviously fully faithful at the level of model categories, \textit{is not} fully faithful at the level of \infinity-categories. However, the \infinity-functor \begin{align}
\label{functor:insertioninweightq}
(-)(q)\coloneqq \triv_{\varepsilon}\circ\hsp (-)(q)\colon\Mod_{\Bbbk}\longrightarrow\Modgr_{\Bbbk}
\longrightarrow\Modmixgr_{\Bbbk}
\end{align}is fully faithful for any integer $q$. Both these assertions are proved in \cite[Warning $1.1.29$]{pavia1}
\end{parag}
\subsection{Algebras and coalgebras in the mixed graded setting}
\label{sec:algcoalg}
In virtue of \ref{parag:monoidalstructureonmodmixgr}, the $\infinity$-category $\Modmixgr_{\Bbbk}$ is a symmetric monoidal $\infinity$-category (which is also presentable and stable, thanks to the presentation via a Dwyer-Kan localization along weak equivalences of a stable model category provided in \cite[Definition $1.2.1$]{CPTVV}), and analogous results hold also for the $\infinity$-category of graded $\Bbbk$-modules $\Modgr_{\Bbbk}$. This allows us to consider algebras and coalgebras for various operads and cooperads in the graded and mixed graded setting: the aim of this subsection is just to fix notations and establish the lax, oplax, or even strong monoidal structures of the \infinity-functors we introduced earlier, in order to check which \infinity-functors can be defined also at the level of mixed graded algebras and/or coalgebras. 
This subsection serves mainly as a reference for \cref{chapter:liealgebras}, where we shall consider Lie algebras, (augmented) commutative algebras, (coaugmented) cocommutative coalgebras, and associative algebras in $\Modgr_{\Bbbk}$ and $\Modmixgr_{\Bbbk}$; all the results collected here are quite straight-forward and easy to prove. The only exception is probably \cref{prop:gradedsymfullyfaithful}, which concerns the fully faithfulness of the \textit{graded} symmetric coalgebra \infinity-functor, and requires some auxiliary technical results. 
\begin{parag}
	We shall begin by introducing the \infinity-categories of \textit{\infinity-operads} and \textit{\infinity-cooperads.} Even if \cite[Chapters $2$ and $3$]{ha} studies carefully the properties of the former, we shall also be interested in the formalism developed in \cite[Chapter $6$, Section $1$]{studyindag2} and \cite{chiralkoszul}, especially when dealing with coalgebras. Let $$\operatorname{Mod}_{\Bbbk}^{\Sigma}\coloneqq\prod_{n\geqslant 1}\Rep_{\Sigma_n}$$be the \infinity-category of symmetric sequences in the \infinity-category of $\Bbbk$-modules. This \infinity-category has a natural (non-symmetric) monoidal structure, called the \textit{composition monoidal structure}, which makes the \infinity-functor\begin{equation*}
	\begin{aligned}
	\Mod_{\Bbbk}^{\Sigma}&\longrightarrow{\Fun}{\lp\Mod_{\Bbbk},\hsp\Mod_{\Bbbk}\rp}\\
	\left\{M_n\right\}_{n\geqslant1}&\mapsto\left\{N\mapsto \bigoplus_{n\geqslant 1}\lp M_n\otimes_{\Bbbk} N^{\otimes n}\rp_{\Sigma_n}\right\}
	\end{aligned}
	\end{equation*}strongly monoidal. Let us remark that this is just a "convenient" way to think about this monoidal structure: see \cite[Section $4$]{compositionproduct} for a more technical  (and purely \infinity-categorical) construction of the composition product on symmetric sequences.
\end{parag}
\begin{defn}\
	\label{def:operadcooperad}
	\begin{enumerate}[label=(\arabic{enumi}), ref=\thelemman.\arabic{enumi}]
		\item An \textit{augmented \infinity-operad in $\Mod_{\Bbbk}$} is an augmented associative algebra object in $\Mod_{\Bbbk}^{\Sigma}$ for the above monoidal structure.
		\item A \textit{coaugmented \infinity-cooperad in $\Mod_{\Bbbk}$} is coaugmented coassociative coalgebra object in $\Mod_{\Bbbk}^{\Sigma}$ for the above monoidal structure.
	\end{enumerate}
\end{defn}
\begin{parag}
	Given an augmented $\infinity$-operad $\mathscr{O}$ in $\Mod_{\Bbbk}$ as in \cref{def:operadcooperad}, we can consider the \textit{$\infinity$-category of $\mathscr{O}$-algebra objects} in any stable symmetric monoidal $\infinity$-category $\scrC$ suitably tensored over $\Mod_{\Bbbk}$ (such as $\Modgr_{\Bbbk}$ or $\Modmixgr_{\Bbbk}$) via the usual formalism provided in \cite[Chapter $3$]{ha}. This agrees with taking modules for the monad  over $\scrC$ defined by\begin{equation}
	\begin{aligned}
	\label{eq:indnilpotentdp}
	\Mod_{\Bbbk}^{\Sigma}&\longrightarrow{\Fun}{\lp\scrC,\hsp\scrC\rp}\\
	\left\{M_n\right\}_{n\geqslant1}&\mapsto\left\{C\mapsto \bigoplus_{n\geqslant 1}\lp M_n\otimes_{\Bbbk} C^{\otimes n}\rp_{\Sigma_n}\right\}.
	\end{aligned}
	\end{equation}
	Under these assumptions, we shall denote with $\Alg_{\mathscr{O}}(\scrC)$ the \infinity-category of $\mathscr{O}$-algebras in $\scrC$.
\end{parag}
\begin{warning}[{\cite[Section $3.5$]{chiralkoszul}}]
	\label{warning:coalgebras}
	Given an augmented \infinity-cooperad $\mathscr{Q}$ in $\Mod_{\Bbbk}$, we would be tempted to define the \infinity-category of $\mathscr{Q}$-algebra objects in $\scrC$ as the comodules for the comonad on $\scrC$ induced by \ref{eq:indnilpotentdp}. However, this turns out to be the \textit{wrong} \infinity-category: indeed, there are \textit{four} different and \textit{inequivalent} types of coalgebras in a stable symmetric monoidal \infinity-category $\scrC$ for a given coaugmented cooperad $\mathscr{Q}$ in $\Mod_{\Bbbk}$. 
	\begin{enumerate}
		\item Comodules for the comonad $\scrC\to\scrC$ defined by the action \ref{eq:indnilpotentdp} of $\Mod_{\Bbbk}^{\Sigma}$ over $\scrC$ yield \textit{ind-nilpotent $\mathscr{Q}$-coalgebras with divided powers}, which we shall denote by $\operatorname{cAlg}^{\operatorname{ind-nil}}_{\mathscr{Q}}(\scrC)_{\operatorname{dp}}$.
		\item Comodules for the right lax action
		\begin{equation}
		\begin{aligned}
		\label{eq:dp}
		\Mod_{\Bbbk}^{\Sigma}&\longrightarrow{\Fun}{\lp\scrC,\hsp\scrC\rp}\\
		\left\{M_n\right\}_{n\geqslant1}&\mapsto\left\{C\mapsto \prod_{n\geqslant 1}\lp M_n\otimes_{\Bbbk} C^{\otimes n}\rp_{\Sigma_n}\right\}
		\end{aligned}
		\end{equation}
		yield \textit{coalgebras with divided powers}, which we shall denote with $\operatorname{cAlg}_{\mathscr{Q}}(\scrC)_{\operatorname{dp}}$. This is another different \infinity-category of coalgebras and the endofunctor $\scrC\to\scrC$ defined by this right lax action is not even comonadic. Because of this, the forgetful \infinity-functor $\oblv^{\operatorname{dp}}_{\mathscr{Q}}\colon \operatorname{cAlg}_{\mathscr{Q}}(\scrC)_{\operatorname{dp}}\to\scrC$ does not even admit a right adjoint in general. 
		\item Comodules for the action
		\begin{equation}
		\begin{aligned}
		\label{eq:indnil}
		\Mod_{\Bbbk}^{\Sigma}&\longrightarrow{\Fun}{\lp\scrC,\hsp\scrC\rp}\\
		\left\{M_n\right\}_{n\geqslant1}&\mapsto\left\{C\mapsto \bigoplus_{n\geqslant 1}\lp M_n\otimes_{\Bbbk} C^{\otimes n}\rp^{\Sigma_n}\right\}
		\end{aligned}
		\end{equation}
		yield \textit{ind-nilpotent coalgebras}, which we shall denote by $\operatorname{cAlg}_{\mathscr{Q}}^{\operatorname{ind-nil}}(\scrC)$.
		\item Finally, comodules for the right lax action
		\begin{equation}
		\begin{aligned}
		\label{eq:coalgebras}
		\Mod_{\Bbbk}^{\Sigma}&\longrightarrow{\Fun}{\lp\scrC,\hsp\scrC\rp}\\
		\left\{M_n\right\}_{n\geqslant1}&\mapsto\left\{C\mapsto \prod_{n\geqslant 1}\lp M_n\otimes_{\Bbbk} C^{\otimes n}\rp^{\Sigma_n}\right\}
		\end{aligned}
		\end{equation}
		yield the \infinity-category of \textit{all} coalgebras, which we shall denote simply by $\operatorname{cAlg}_{\mathscr{Q}}(\scrC)$. The \infinity-functor $\oblv_{\mathscr{Q}}\colon\operatorname{cAlg}_{\mathscr{Q}}(\scrC)\to\scrC$ is conservative, preserves all colimits, and preserves totalizations of $\oblv_{\mathscr{Q}}$-split co-simplicial objects, hence is comonadic and admits a right adjoint. However, such right adjoint it is not easy to describe because the endofunctor on $\scrC$ defined by this comonad \textit{does not }agree with the one induced by the right lax action of $\mathscr{Q}$. However, this is the correct \infinity-category to consider, at least in our two main cases of interest (namely, $\operatorname{CoComm}^{\operatorname{aug}}$ and $\operatorname{CoAssoc}^{\operatorname{aug}}$).
	\end{enumerate}
	The trace map $(-)_{\Sigma_n}\to(-)^{\Sigma_n}$ yields natural maps of right lax actions $\ref{eq:indnilpotentdp}\to\ref{eq:indnil}$ and $\ref{eq:dp}\to\ref{eq:coalgebras}$. In characteristic $0$, both maps are homotopy equivalences by the usual symmetrization argument: i.e., every coalgebra has naturally a divided power structure. Since this is our standing assumption, we shall omit all references to divided powers in our notations from now on. Still, the two actions $\ref{eq:indnilpotentdp}\simeq\ref{eq:indnil}$ and $\ref{eq:dp}\simeq\ref{eq:coalgebras}$ are \textit{not} equivalent: there is a natural map \begin{align}
	\label{eq:indniltosimple}
	\operatorname{res}\colon\operatorname{cAlg}^{\operatorname{ind-nil}}_{\mathscr{Q}}(\scrC)\longrightarrow\operatorname{cAlg}_{\mathscr{Q}}(\scrC)
	\end{align}and in \cite[Chapter $6$, Conjecture $2.8.4$]{studyindag2} it is conjectured that this map is fully faithful. It is known, however, that there is always an equivalence of \infinity-functors from $\operatorname{cAlg}^{\operatorname{ind-nil}}_{\mathscr{Q}}(\scrC)$ to $\scrC$ 
	\begin{align}
	\label{eq:forgetfulcoalgebras}
	\operatorname{oblv}_{\mathscr{Q}}\circ\operatorname{res}\simeq \oblv_{\mathscr{Q}}^{\operatorname{ind-nil}}
	\end{align}(\cite[Section $2.8.5$]{studyindag2}). Since $\oblv_{\mathscr{Q}}$ and $\oblv_{\mathscr{Q}}^{\operatorname{ind-nil}}$ both commute with colimits and are conservative, it follows that $\operatorname{res}$ preserves all colimits as well, therefore it admits a right adjoint.
\end{warning}
\begin{notation}
	\label{notation:gradedandmixedgradedalgebras}
	If $\Alg_{\mathscr{O}}$ is the $\infinity$-category of $\mathscr{O}$-algebra objects in $\Mod_{\Bbbk}$, we shall denote by $\Alg^{\gr}_{\mathscr{O}}$ and $\Algmixgr_{\mathscr{O}}$ the $\infinity$-categories of $\mathscr{O}$-algebras in $\Modgr_{\Bbbk}$ and $\Modmixgr_{\Bbbk}$, respectively.\\
	Dually, if $\operatorname{cAlg}_{\mathscr{Q}}$ is the $\infinity$-category of $\mathscr{Q}$-coalgebra objects in $\Mod_{\Bbbk}$, we shall denote by $\operatorname{cAlg}_{\mathscr{Q}}^{\gr}$ and $\varepsilon\text{-}\operatorname{cAlg}_{\mathscr{Q}}^{\gr}$ the $\infinity$-categories of $\mathscr{Q}$-coalgebras in $\Modgr_{\Bbbk}$ and $\Modmixgr_{\Bbbk}$. For example, the $\infinity$-category of commutative mixed graded $\Bbbk$-algebras will be denoted by $\CAlgmixgr_{\Bbbk}\coloneqq\CAlg{\lp\Modmixgr_{\Bbbk}\rp}$. Analogous notations will be used, without further explanation or definitions, for other algebra structures in (mixed) graded $\Bbbk$-modules in the remainder of the work.
\end{notation}
\begin{parag}
	\label{parag:forgettingmixedstructureisok}
	Let $\mathscr{O}$ be an $\infinity$-operad. As already observed in \cref{sec:basicmixedgraded}, the $\infinity$-functors$$\oblv_{\varepsilon}\colon\Modmixgr_{\Bbbk}\longrightarrow\Modgr_{\Bbbk}$$and$$\triv_{\varepsilon}\colon\Modgr_{\Bbbk}\longrightarrow\Modmixgr_{\Bbbk}$$are strongly symmetric monoidal and accessible $\infinity$-functors which preserve all limits and colimits, hence their right adjoints (which exist because of the Adjoint Functor Theorem) are lax symmetric monoidal. In particular, \textit{all} these $\infinity$-functors induce $\infinity$-functors at the level of algebras and coalgebras: this means that forgetting  a mixed structure, or setting a trivial one, does not alter the type of algebras considered. Moreover, forgetting the mixed structures preserves (and actually, creates) all limits and colimits.
\end{parag}
\begin{parag}
	\label{parag:gradedandmixedgradedcoalgebras}
	Considering graded $\mathscr{O}$-algebras and mixed graded $\mathscr{O}$-algebras, we have two obvious forgetful $\infinity$-functors$$\oblv^{\mixgr}_{\mathscr{O}}\colon \varepsilon\text{-}\Alg^{\gr}_{\mathscr{O}}\longrightarrow\Modmixgr_{\Bbbk}$$and$$\oblv^{\gr}_{\mathscr{O}}\colon\Alg^{\gr}_{\mathscr{O}}\longrightarrow\Modgr_{\Bbbk}.$$Both $\infinity$-functors satisfy the hypothesis of \cite[Section $3.1$]{ha}, hence admit left adjoints$$\operatorname{Free}^{\mixgr}_{\mathscr{O}}\colon\Modmixgr_{\Bbbk}\longrightarrow\varepsilon\text{-}\Alg^{\gr}_{\mathscr{O}}$$and$$\operatorname{Free}^{\gr}_{\mathscr{O}}\colon \Modgr_{\Bbbk}\longrightarrow\Alg^{\gr}_{\mathscr{O}}.$$Moreover, the square of $\infinity$-functors$$	\begin{tikzpicture}[scale=0.75,baseline=0.5ex]
	\node (a) at (-3,0){$\Modmixgr_{\Bbbk}$};
	\node (b) at (3,0){$\Modgr_{\Bbbk}$};
	\node (a1) at (-3,3){$\varepsilon\text{-}\Alg^{\gr}_{\mathscr{O}}$};
	\node (b1) at (3,3){$\Alg^{\gr}_{\mathscr{O}}$};
	\draw[->,font=\scriptsize] ([yshift=3pt]a.east) to node[above]{$\oblv_{\varepsilon}$} ([yshift=3pt]b.west);
	\draw[->,font=\scriptsize] ([yshift=-3pt]b.west) to node[below]{$\triv_{\varepsilon}$} ([yshift=-3pt]a.east);
	\draw[->,font=\scriptsize] ([yshift=3pt]a1.east) to node[above]{$\oblv_{\varepsilon}$}  ([yshift=3pt]b1.west);
	\draw[->,font=\scriptsize] ([yshift=-3pt]b1.west) to node[below]{$\triv_{\varepsilon}$} ([yshift=-3pt]a1.east);
	\draw[->,font=\scriptsize] ([xshift=3pt]a1.south) to node[right]{$\oblv^{\mixgr}_{\mathscr{O}}$} ([xshift=3pt]a.north);
	\draw[->,font=\scriptsize] ([xshift=-3pt]a.north) to node[left]{$\operatorname{Free}^{\mixgr}_{\mathscr{O}}$} ([xshift=-3pt]a1.south);
	\draw[->,font=\scriptsize] ([xshift=3pt]b1.south) to node[right]{$\oblv^{\gr}_{\mathscr{O}}$}  ([xshift=3pt]b.north);
	\draw[->,font=\scriptsize] ([xshift=-3pt]b.north) to node[left]{$\operatorname{Free}^{\gr}_{\mathscr{O}}$} ([xshift=-3pt]b1.south);
	\end{tikzpicture}$$commutes straight-forwardly in every direction.
\end{parag}
\begin{parag}As foretold in \cref{warning:coalgebras}, one needs to be careful when dealing with coalgebras. For sure, replacing $\mathscr{O}$ with an $\infinity$-cooperad $\mathscr{Q}$ and considering the $\infinity$-categories $\varepsilon\text{-}\operatorname{cAlg}^{\gr,\operatorname{ind-nil}}_{\mathscr{Q}}$ and $\operatorname{cAlg}^{\gr,\operatorname{ind-nil}}_{\mathscr{Q}}$ of (respectively) mixed graded and graded ind-nilpotent $\mathscr{Q}$-coalgebras, we have another commutative square of $\infinity$-functors\begin{equation}
	\label{squarecoalgebras}
	\begin{aligned}\begin{tikzpicture}[scale=0.75,baseline=0.5ex]
	\node (a) at (-3,0){$\Modmixgr_{\Bbbk}$};
	\node (b) at (3,0){$\Modgr_{\Bbbk}$};
	\node (a1) at (-3,3){$\varepsilon\text{-}\operatorname{cAlg}^{\gr,\operatorname{ind-nil}}_{\mathscr{Q}}$};
	\node (b1) at (3,3){$\operatorname{cAlg}^{\gr,\operatorname{ind-nil}}_{\mathscr{Q}}$};
	\draw[->,font=\scriptsize] ([yshift=3pt]a.east) to node[above]{$\oblv_{\varepsilon}$} ([yshift=3pt]b.west);
	\draw[->,font=\scriptsize] ([yshift=-3pt]b.west) to node[below]{$\triv_{\varepsilon}$} ([yshift=-3pt]a.east);
	\draw[->,font=\scriptsize] ([yshift=3pt]a1.east) to node[above]{$\oblv_{\varepsilon}$}  ([yshift=3pt]b1.west);
	\draw[->,font=\scriptsize] ([yshift=-3pt]b1.west) to node[below]{$\triv_{\varepsilon}$} ([yshift=-3pt]a1.east);
	\draw[->,font=\scriptsize] ([xshift=-3pt]a1.south) to node[left]{$\oblv^{\mixgr,\operatorname{ind-nil}}_{\mathscr{Q}}$} ([xshift=-3pt]a.north);
	\draw[->,font=\scriptsize] ([xshift=3pt]a.north) to node[right]{$\operatorname{coFree}^{\mixgr,\operatorname{ind-nil}}_{\mathscr{Q}}$} ([xshift=3pt]a1.south);
	\draw[->,font=\scriptsize] ([xshift=-3pt]b1.south) to node[left]{$\oblv^{\gr,\operatorname{ind-nil}}_{\mathscr{Q}}$}  ([xshift=-3pt]b.north);
	\draw[->,font=\scriptsize] ([xshift=3pt]b.north) to node[right]{$\operatorname{coFree}^{\gr,\operatorname{ind-nil}}_{\mathscr{Q}}$} ([xshift=3pt]b1.south);
	\end{tikzpicture}
	\end{aligned}
	\end{equation}where now $\operatorname{coFree}_{\mathscr{Q}}^{\gr,\operatorname{ind-nil}}$ and $\operatorname{coFree}_{\mathscr{Q}}^{\mixgr,\operatorname{ind-nil}}$ are \textit{right} adjoints to the forgetful $\infinity$-functors $\oblv_{\mathscr{Q}}^{\gr,\operatorname{ind-nil}}$ and $\oblv_{\mathscr{Q}}^{\mixgr,\operatorname{ind-nil}}$ respectively. Again, since equivalences in all these $\infinity$-categories are checked at the level of the underlying graded $\Bbbk$-module, this square commutes in every direction as well.
\end{parag}
\begin{remark}
	Let $\scrC$ be a symmetric monoidal \infinity-category $\Bbbk$-linear over a base ring $\Bbbk$ of characteristic $0$. Even if we do not know whether \ref{eq:indniltosimple} is a fully faithful \infinity-functor, thanks to \cite[Chapter $6$, Theorem $2.9.4$]{studyindag2} we know that\begin{align*}
	\Map_{\scrC}{\lp X,\hsp Y\rp}\simeq \Map_{\scrC}{\lp \oblv_{\mathscr{Q}}{\lp\triv_{\mathscr{Q}}^{\operatorname{ind-nil}}(X)\rp},\hsp Y\rp}&\simeq \Map_{\operatorname{cAlg}^{\operatorname{ind-nil}}_{\mathscr{Q}}(\scrC)}{\lp \triv_{\mathscr{Q}}^{\operatorname{ind-nil}}(X),\hsp\operatorname{coFree}^{\operatorname{ind-nil}}_{\mathscr{Q}}(Y)\rp}\\&\simeq \Map_{\operatorname{cAlg}_{\mathscr{Q}}(\scrC)}{\lp \triv_{\mathscr{Q}}(X),\hsp\operatorname{res}{\lp\operatorname{coFree}^{\operatorname{ind-nil}}_{\mathscr{Q}}(Y)\rp}\rp}
	\end{align*}at least when $\mathscr{Q}$ is either $\operatorname{coAssoc}^{\operatorname{aug}}$ or $\operatorname{coComm}^{\operatorname{aug}}$, which are the two main cooperads we shall be interested in. In particular, denoting by $\operatorname{coFree}_{\mathscr{Q}}$ the composition of \infinity-functors $ \operatorname{res}\circ\operatorname{coFree}_{\mathscr{Q}}^{\operatorname{ind-nil}}$, in these cases one can obtain an analogous square to \ref{squarecoalgebras} where now the \infinity-categories of ind-nilpotent coalgebras are replaced by the \infinity-categories of all $\mathscr{Q}$-coalgebras, thanks to the equivalence \ref{eq:forgetfulcoalgebras}. The vertical arrows - while not being adjoints anymore - interact well one with the other just like in the ind-nilpotent case.
\end{remark}
\begin{parag}
	\label{parag:functorsthatpreservealgebras}
	Many of the important $\infinity$-functors relating usual $\Bbbk$-modules and (mixed) graded $\Bbbk$-modules can be defined at the level of $\mathscr{O}$-algebras and $\mathscr{Q}$-coalgebras.
	\begin{itemize}
		\item  From the description of the tensor product of mixed graded $\Bbbk$-modules provided in \ref{parag:monoidalstructureonmodmixgr}, it follows that the $\infinity$-functor $(-)(0)\colon\Mod_{\Bbbk}\to\Modmixgr_{\Bbbk}$ is obviously strongly monoidal, hence it defines an $\infinity$-functor at the level of algebras and coalgebras objects $$(-)(0)\colon\Alg_{\mathscr{O}}\longrightarrow\varepsilon\text{-}\Alg_{\mathscr{O}}^{\gr}$$ and $$(-)(0)\colon\operatorname{cAlg}_{\mathscr{Q}}\longrightarrow\varepsilon\text{-}\operatorname{cAlg}_{\mathscr{Q}}^{\gr}.$$As a consequence, its right adjoint $|-|$ is lax symmetric monoidal, and thus yields an $\infinity$-functor$$|-|\colon\Alg_{\mathscr{O}}\longrightarrow\varepsilon\text{-}\Alg_{\mathscr{O}}^{\gr}.$$
		\item Moreover, since $\Bbbk(\infinity)$ is a cocommutative and counital coalgebra object (\cite[Section $1.5$]{CPTVV}), the Tate realization acquires a lax symmetric monoidal structure as well. Again, it induces an $\infinity$-functor$$|-|^{\operatorname{t}}\colon\varepsilon\text{-}\Alg_{\mathscr{O}}^{\gr}\longrightarrow\Alg_{\mathscr{O}}.$$Actually, if we restrict ourselves to the $\infinity$-category $\Modmixgrcn_{\Bbbk}$ spanned by mixed graded $\Bbbk$-modules trivial in negative weights, such monoidal structure is \textit{strict} (\cite[Porism $2.3.20$]{pavia1}), hence it also provides an \infinity-functor$$\left|-\right|^{\operatorname{t}}\colon\operatorname{cCcAlg}{\lp\Modmixgrcn_{\Bbbk}\rp}\longrightarrow\operatorname{cCcAlg}_{\Bbbk}.$$
		\item Finally, consider the $\infinity$-functor $(-)_0\colon\Modmixgr_{\Bbbk}\to\Mod_{\Bbbk}$. In general, this is both lax and oplax monoidal, but not strongly so: in fact, given two mixed graded $\Bbbk$-modules $M_{\bullet}$ and $N_{\bullet}$, their tensor product $M_{\bullet}\otimesmixgr_{\Bbbk}N_{\bullet}$ exhibits $M_0\otimes_{\Bbbk}N_0$ as a direct summand of $\lp M_{\bullet}\otimesmixgr_{\Bbbk}N_{\bullet}\rp_0$, and the natural inclusions and projections yield the desired lax and oplax monoidal structures. In particular, we have $\infinity$-functors defined at the level of algebras and coalgebras objects$$(-)_0\colon\varepsilon\text{-}\Alg_{\mathscr{O}}^{\gr}\longrightarrow\Alg_{\mathscr{O}}$$ and $$(-)_0\colon\varepsilon\text{-}\operatorname{cAlg}_{\mathscr{Q}}^{\gr}\longrightarrow\operatorname{cAlg}_{\mathscr{Q}}.$$Let us remark that this lax and oplax monoidal structure is actually strong when restricted to the full sub-$\infinity$-categories $\Modmixgrcn_{\Bbbk}$ and $\Modmixgrccn_{\Bbbk}$: this follows trivially because, in such cases, the only summand in $\lp M_{\bullet}\otimesmixgr_{\Bbbk}N_{\bullet}\rp_0$ is $M_0\otimes_{\Bbbk} N_0$ itself.
	\end{itemize} 
\end{parag}
\begin{remark}
	\label{remark:cofreecoalgebra}
	As stated in \cite[Chapter $6$, Section $4.2$]{GR}, if one takes as $\mathscr{Q}$ the $\infinity$-cooperad $\operatorname{CoComm}^{\operatorname{aug}}$ of coaugmented cocommutative coalgebras, then the composition \begin{align}
	\label{functor:symgr}
	\operatorname{res}\circ\operatorname{coFree}^{\operatorname{ind-nil}}\colon \Mod_{\Bbbk}\longrightarrow\operatorname{cCcAlg}^{\operatorname{ind-nil}}_{\Bbbk//\Bbbk}\xrightarrow{\ref{eq:indniltosimple}}\operatorname{cCcAlg}_{\Bbbk//\Bbbk}
	\end{align}coincides with the usual symmetric algebra $\infinity$-functor with its cocommutative coalgebra structure given by its natural Hopf algebra structure.
\end{remark}
It is known (\cite[Chapter $6$, Conjecture $2.9.3$ and Theorem $2.9.4$]{GR}) that it is possible to extract primitive elements from the symmetric (co)algebra, i.e., there exists an $\infinity$-functor$$\Prim\colon\operatorname{cCcAlg}_{\Bbbk//\Bbbk}\longrightarrow\Mod_{\Bbbk}$$and a natural equivalence $M\simeq\Prim{\lp\Sym_{\Bbbk}{\lp M\rp}\rp}$ for any $\Bbbk$-module $M$. We want to show that, in the graded setting, the extraction of primitive elements is particularly simple, at least in the ind-nilpotent case.
\begin{parag}
	Let $\CoCommgraug$ be the $\infinity$-category ${\operatorname{cAlg}_{\operatorname{CoComm}^{\operatorname{aug}}}{\lp\Modgr_{\Bbbk}\rp}}$ of coaugmented cocommutative coalgebras in mixed graded $\Bbbk$-modules. Thanks to the discussion provided in \ref{parag:functorsthatpreservealgebras}, we know that the $\infinity$-functor $(-)(0)\colon\Mod_{\Bbbk}\to\Modgr_{\Bbbk}$, is strongly monoidal, and its left/right adjoint $(-)_0\colon\Modgr_{\Bbbk}\to\Mod_{\Bbbk}$ hence is both lax and oplax monoidal. So, consider the $\infinity$-functor\begin{align}
	\Symgr_{\Bbbk}\colon\Modgr_{\Bbbk}\longrightarrow\CoCommgraug
	\end{align}
	be the (graded version of the) \infinity-functor \ref{functor:symgr}, given by the composition of the cofree ind-nilpotent coalgebra \infinity-functor for the operad $\operatorname{CoComm}^{\operatorname{aug}}$, right adjoint to the forgetful $\infinity$-functor $$\oblv^{\operatorname{ind-nil}}_{\operatorname{cCcAlg}}\colon\operatorname{cCcAlg}_{\Bbbk//\Bbbk}^{\operatorname{gr,ind-nil}}\longrightarrow\Modgr_{\Bbbk}$$and the \infinity-functor $\operatorname{res}$ of \ref{eq:indniltosimple}. We want to prove the following result.
\end{parag}
\begin{propositionn}
	\label{prop:gradedsymfullyfaithful}
	The \infinity-functor 	\begin{align}
	\label{functor:gradedsym}
	\Symgr_{\Bbbk}{\lp(-)[-1](1)\rp}\colon\Mod_{\Bbbk}\overset{[-1]}{\simeq}\Mod_{\Bbbk}\overset{(-)(1)}{\longhookrightarrow}\Modgr_{\Bbbk}\xrightarrow{\Symgr_{\Bbbk}}\operatorname{cCcAlg}_{\Bbbk//\Bbbk}^{\operatorname{gr}}
	\end{align}is fully faithful.
\end{propositionn}
We shall deduce \cref{prop:gradedsymfullyfaithful} from the following Lemma.
\begin{lemman}
	\label{prop:fakegradedsymfullyfaithful}
	The $\infinity$-functor 
	\begin{align}
	\label{functor:propedeutico}
	\Mod_{\Bbbk}\overset{[-1]}{\simeq}\Mod_{\Bbbk}\overset{(-)(1)}{\longhookrightarrow}\Modgr_{\Bbbk}\xrightarrow{\operatorname{coFree}^{\operatorname{ind-nil}}_{\operatorname{cCcAlg}}}\operatorname{cCcAlg}_{\Bbbk//\Bbbk}^{\operatorname{gr,ind-nil}}
	\end{align}
	is fully faithful.
\end{lemman}
\begin{proof}
	By abstract nonsense, the composition \ref{functor:propedeutico} admits a left adjoint
	$$\operatorname{cCcAlg}_{\Bbbk//\Bbbk}^{\operatorname{gr,ind-nil}}\xrightarrow{\oblv^{\operatorname{ind-nil}}_{\operatorname{cCcAlg}}}\Modgr_{\Bbbk}\xrightarrow{(-)_1}\Mod_{\Bbbk}\overset{[1]}{\simeq}\Mod_{\Bbbk}$$given by composing all left adjoints of each $\infinity$-functor which makes up \ref{functor:propedeutico}. So it will suffice to show that the counit$$M\longrightarrow \oblv^{\operatorname{ind-nil}}_{\operatorname{cCcAlg}}\lp\operatorname{coFree}^{\operatorname{ind-nil}}_{\operatorname{cCcAlg}}\lp M[-1](1)\rp\rp_1[1]$$ of such adjunction is an equivalence. Thanks to the equivalence \ref{eq:forgetfulcoalgebras}, we know that $$\operatorname{oblv}^{\operatorname{ind-nil}}_{\operatorname{cCcAlg}}\simeq \operatorname{oblv}_{\operatorname{cCcAlg}}\circ \operatorname{res}$$and \cref{remark:cofreecoalgebra} assures us that the composition $\operatorname{coFree}_{\operatorname{cCcAlg}}\coloneqq\operatorname{res}\circ \operatorname{coFree}^{\operatorname{ind-nil}}_{\operatorname{cCcAlg}}$ is just the symmetric coalgebra \infinity-functor $\Symgr_{\Bbbk}$, with its natural grading. So, we can rewrite the counit above as
	\begin{align*}
	M\longrightarrow& \oblv^{\operatorname{ind-nil}}_{\operatorname{cCcAlg}}\lp\operatorname{coFree}^{\operatorname{ind-nil}}_{\operatorname{cCcAlg}}\lp M[-1](1)\rp\rp_1[1]\\
	\simeq \hspace{0.1cm} & \oblv_{\operatorname{cCcAlg}}\lp\operatorname{res}\lp\operatorname{coFree}^{\operatorname{ind-nil}}_{\operatorname{cCcAlg}}\lp M[-1](1)\rp\rp\rp_1[1]\\
	\simeq\hspace{0.1cm}  & \Symgr_{\Bbbk}\lp M[-1](1)\rp_1[1]\simeq M
	\end{align*}
	which proves our assertion.
\end{proof}
\begin{proof}[Proof of \cref{prop:gradedsymfullyfaithful}]
The \infinity-functor \ref{eq:indniltosimple} induces an equivalence of mapping spaces$$\Map_{\operatorname{cAlg}_{\mathscr{Q}}^{\operatorname{ind-nil}}(\scrC)}{\lp A,\hsp B\rp}\simeq \Map_{\operatorname{cAlg}_{\mathscr{Q}}(\scrC)}{\lp \operatorname{res}(A),\hsp\operatorname{res}(B)\rp}$$whenever $A$ lies in the essential image of the Koszul duality \infinity-functor$$\operatorname{coPrim}^{\operatorname{enh,ind-nil}}_{\mathscr{O}}\colon\operatorname{Alg}_{\mathscr{O}}{\lp\scrC\rp}\longrightarrow\operatorname{cAlg}^{\operatorname{ind-nil}}_{\mathscr{Q}}(\scrC),$$where $\mathscr{O}\coloneqq\mathscr{Q}^{\vee}$ is the Koszul dual operad to $\mathscr{Q}$: this is the content of \cite[Chapter $6$, Corollary $2.10.7$]{studyindag2}. Moreover, in virtue of \cite[Chapter $6$, Section $2.4.3$]{studyindag2}, it is always true that
$$\operatorname{coPrim}^{\operatorname{enh,ind-nil}}_{\mathscr{Q}^{\vee}}\circ \triv_{\mathscr{Q}^{\vee}}\simeq \operatorname{coFree}^{\operatorname{ind-nil}}_{\mathscr{Q}},$$and since we defined $\Symgr_{\Bbbk}$ as the composition $\operatorname{res}\circ \operatorname{coFree}^{\operatorname{ind-nil}}_{\operatorname{cCcAlg}}$ we have a chain of equivalences of mapping spaces\begin{align*}
\Map_{\Mod_{\Bbbk}}{\lp M,\hsp N\rp}\overset{\ref{prop:fakegradedsymfullyfaithful}}{\simeq}\hspace{0.1cm}& \Map_{\operatorname{cCcAlg}_{\Bbbk//\Bbbk}^{\operatorname{gr,ind-nil}}}{\lp\operatorname{coFree}^{\operatorname{ind-nil}}_{\operatorname{cCcAlg}}\lp M[-1](1)\rp,\hsp \operatorname{coFree}^{\operatorname{ind-nil}}_{\operatorname{cCcAlg}}\lp N[-1](1)\rp\rp}\\\simeq\hspace{0.3cm} &\Map_{\CoCommgraug}{\lp\Symgr_{\Bbbk}{\lp M[-1](1)\rp},\hsp\Symgr_{\Bbbk}{\lp N[-1](1)\rp}\rp}
\end{align*}for any $\Bbbk$-modules $M$ and $N$.
\end{proof}
\begin{remark}
	\label{remark:gradedsymalgebrafullyfaithful}
	The analogous claim of \cref{prop:gradedsymfullyfaithful} holds if we consider the graded symmetric algebra $\infinity$-functor as the \textit{free augmented commutative algebra} $\infinity$-functor from $\Mod_{\Bbbk}$ to $\CAlggraug$, where now the latter denotes the $\infinity$-category of augmented commutative algebras in graded $\Bbbk$-modules. The proof is completely analogous to the one of \cref{prop:fakegradedsymfullyfaithful}, without needing the further technical auxiliary result of \cite{studyindag2}.
\end{remark}
\section{Lie algebras in characteristic $0$}
\label{sec:dglie}
In this section, we shall recollect the main notions and fix the notations we need in order to study the homotopy theory of Lie algebras in characteristic $0$. For the sake of clarity, we shall gather here also elementary definitions and constructions, so as to give some motivations and to make explicit how we generalized the classical theory. Our main reference is \cite{dagx}. 
\subsection{Main definitions and notations}
It is well known (see, for instance, \cite[Propositions $7.1.4.6$ and $7.1.4.11$]{ha}) that the homotopy theory of $\Ebb_1$-ring spectra (respectively $\Einf$-ring spectra) and differential graded algebras (respectively commutative differential graded algebras) are equivalent, when working over a base provided by a discrete commutative ring $\Bbbk$ of characteristic $0$. Following this philosophy, we can define Lie algebras in the derived setting in two different ways.
\begin{defn}[Differential graded Lie algebras]
	\label{def:dgla}
	A \textit{differential graded Lie algebra} defined over a commutative ring $\Bbbk$ is a differential graded $\Bbbk$-module $\gfrak_{\bullet}$, endowed with a bracket $$[-,\hsp-]\colon\gfrak_p\otimes_{\Bbbk}\gfrak_q\longrightarrow\gfrak_{p+q}$$satisfying the following axioms.
	\begin{itemize}
		\item For $x\in\gfrak_p$ and $y\in\gfrak_q$, we have that $$[x,\hsp y]+(-1)^{pq}[y,\hsp x]=0.$$
		\item For $x\in\gfrak_p$, $y\in\gfrak_q$ and $z\in\gfrak_r$, we have that $$(-1)^{pr}[x,\hsp [y,\hsp z]]+(-1)^{pq}[y,\hsp [z,\hsp x]]+(-1)^{qr}[z,\hsp[x,\hsp y]]=0.$$
		\item The differential $\operatorname{d}\colon\gfrak\to\gfrak[1]$ is a derivation with respect to the Lie bracket; i.e. for $x\in\gfrak_p$ and $y\in\gfrak_q$ we have that $$\operatorname{d}[x,\hsp y]=[\operatorname{d}(x),\hsp y]+(-1)^{pq}[x,\hsp\operatorname{d}(y)].$$
	\end{itemize}
	A morphism of differential graded Lie algebras $f\colon\gfrak\to\mathfrak{h}$ is a morphism of the underlying chain complexes preserving the Lie bracket, i.e. such that ${f}{\lp[x,\hsp y]\rp}=\left[f(x),\hsp f(y)\right]$ for all $x$, $y$ in $\gfrak$.
\end{defn}
\begin{parag}
	With the descriptions provided in Definition \ref{def:dgla}, it follows that differential graded Lie algebras are naturally gathered in a category, $\dgLie_{\Bbbk}$. This category is moreover endowed with a model structure where weak equivalences and fibrations are detected by the forgetful functor $$\oblv_{\Lie}\colon\dgLie_{\Bbbk}\longrightarrow\dgMod_{\Bbbk},$$ see for example \cite[Proposition $2.1.10$]{dagx}. In particular, using the formalism of (\cite{dwyerkan}), we can consider the Dwyer-Kan localization of $\dgLie_{\Bbbk}$ at the class $\mathscr{W}$ of weak equivalences (i.e., quasi-isomorphisms) and consider its simplicial nerve to get an $\infinity$-category ${\operatorname{N}_{\Delta}}{\lp\operatorname{L}^{\operatorname{H}}\lp\dgLie_{\Bbbk},\mathscr{W}\rp\rp}$, that we simply denote by $\Lie_{\Bbbk}$.
\end{parag}
\begin{defn}
	The $\infinity$-category $\Lie_{\Bbbk}$ is the \textit{$\infinity$-category of Lie algebras over $\Bbbk$}.
\end{defn}
Alternatively, we can define Lie algebras via the operadic approach in the following way. Let us denote by $\Lie$ the ordinary Lie ($1$-)operad. 
\begin{defn}
	\label{def:linftyalgebras}
	A \textit{Lie algebra over a commutative ring $\Bbbk$} is an algebra for the Lie operad in the $\infinity$-category $\Mod_{\Bbbk}$ of $\Bbbk$-modules. The $\infinity$-category ${\Alg_{\Lie}}{\lp\Mod_{\Bbbk}\rp}$ is the \textit{$\infinity$-category of Lie algebras over $\Bbbk$}, and shall be denoted by $\Lie_{\Bbbk}$.
\end{defn}
\begin{remark}
	Definitions \ref{def:dgla} and \ref{def:linftyalgebras} agree in the characteristic $0$ setting, hence the same notation. Let us briefly recall the main steps that provide the equivalence of these two homotopy theories: recall that homotopy algebras for the Lie operad are the same as $\LL_{\infty}$-algebras in chain complexes, where $\LL_{\infty}$ is any cofibrant replacement of the operad $\Lie$ in the model category of operads. There exists an obvious inclusion $$\operatorname{dgLie}_{\Bbbk}\longhookrightarrow{\operatorname{Alg}_{\LL_{\infty}}}{\lp\dgMod_{\Bbbk}\rp}.$$There exist two different categories related to $\dgLie_{\Bbbk}$ and ${\operatorname{Alg}_{\LL_{\infty}}}{\lp\dgMod_{\Bbbk}\rp}.$
	\begin{enumerate}
		\item The category ${{\operatorname{Pro}}{\lp\operatorname{dgArt}_{\Bbbk}\rp}}^{\op}$, which is the opposite category of the category of pro-objects in local differential graded Artinian commutative $\Bbbk$-algebras. This category is endowed with a model structure, whose fibrant objects are precisely $\LL_{\infty}$-algebras (this is showed in the proof of \cite[Proposition $4.42$]{pridham}). 
		\item The category $\operatorname{dgcCcAlg}^{\operatorname{un}}_{\Bbbk}$, which is the category of unital differential graded cocommutative coalgebras (in the sense of \cite[Definition $2.1.1$]{hinichcoalgebras}). This category is endowed with a model structure as well, and in virtue of \cite[Theorem $3.2$]{hinichcoalgebras} there exists a Quillen equivalence$$\operatorname{dgcCcAlg}^{\operatorname{un}}_{\Bbbk}\overset{\simeq}{\longrightarrow}\dgLie_{\Bbbk}.$$
	\end{enumerate}
	Moreover, there exists an equivalence of categories$${{\operatorname{Pro}}{\lp\operatorname{dgArt}_{\Bbbk}\rp}}^{\op}\overset{\simeq}{\longrightarrow}\operatorname{dgcCcAlg}^{\operatorname{un}}_{\Bbbk}$$which identifies the two model structures (this is the content of \cite[Corollary $4.56$]{pridham}). In particular, the inclusion of fibrant differential graded Lie algebras into the category of fibrant $\LL_{\infty}$-algebras extends to a right Quillen functor, whose left adjoint agrees with the Quillen equivalence $\operatorname{dgcCcAlg}^{\operatorname{un}}_{\Bbbk}\simeq\dgLie_{\Bbbk}.$ This induces the desired equivalence at the level of associated \infinity-categories.
\end{remark}
\begin{parag}Let $\Alg_{\Bbbk}$ denote the $\infinity$-category of associative $\Bbbk$-algebras. We have an $\infinity$-functor$$\operatorname{L}\colon\Alg_{\Bbbk}\longrightarrow\Lie_{\Bbbk}$$ given by the \textit{underlying commutator Lie algebra}, which can be understood as the $\infinity$-functor  induced by the functor $$\dgAlg_{\Bbbk}\coloneqq{\Alg}^{\otimes}{\lp\dgMod_{\Bbbk}\rp}\longrightarrow\dgLie_{\Bbbk}$$sending a differential graded $\Bbbk$-algebra to the differential graded Lie algebra given by the same underlying chain complex and endowed with the commutator Lie bracket. This $\infinity$-functor is actually the right adjoint in the adjoint couple
	\begin{align}
	\label{adjunction:envelopingalgebra/lie}
	\begin{tikzpicture}[scale=0.75,baseline=-0.5ex]
	\node (a)at (-1.5,0){$\operatorname{U}\colon\Lie_{\Bbbk}$};
	\node(b) at (1.5,-0){$\Alg_{\Bbbk}:\operatorname{L}.$};
	\draw[->] ([yshift=4.5pt]a.east) -- ([yshift=4.5pt]b.west);
	\draw[->] ([yshift=-1.5pt]b.west) -- ([yshift=-1.5pt]a.east);
	\end{tikzpicture}
	\end{align}
\end{parag}
\begin{defn}
	The left adjoint in the adjunction \ref{adjunction:envelopingalgebra/lie} is the \textit{universal enveloping algebra $\infinity$-functor}.
\end{defn}
\begin{remark}
	\label{remark:universalenvelopingalgebragaitsgory}
	The universal enveloping algebra $\infinity$-functor has been observed from many perspectives, and a construction in the setting of \textit{$\Linf$-algebras in chain complexes} (i.e., algebras over the operad $\Linf$, which provides a cofibrant replacement in the model category of operads for the operad $\Lie$) is provided for instance in \cite{baranovsky}. Another interesting (and more general) approach to $\mathbb{E}_n$-enveloping algebras for arbitrary $n$ can be found in \cite[Section $2.7$]{linearbatalin}. In \cite[Remark $2.1.7$]{dagx}, it is provided an explicit model in the context of differential graded Lie algebras.\\
	In the remainder of this paper, however, we shall need a universal enveloping algebra $\infinity$-functor also in the context of Lie algebra objects in less standard symmetric monoidal stable $\infinity$-categories than the usual $\infinity$-category $\Mod_{\Bbbk}$; in particular, we shall be interested in the case of ${\Alg_{\Lie}}{\lp\scrC\rp}$ where $\scrC$ is the $\infinity$-category of mixed graded $\Bbbk$-modules $\Modmixgr_{\Bbbk}$. It appears that a universal enveloping algebra construction can be carried out also in this more general setting since we have a natural map of operads$$\Lie\longrightarrow\operatorname{Assoc}^{\operatorname{aug}}$$which induces a restriction $\infinity$-functor between associative augmented algebra objects in a symmetric monoidal $\infinity$-category $\scrC$ to Lie algebra objects. This $\infinity$-functor admits a left adjoint which is the universal enveloping algebra $\infinity$-functor we need: see \cite[Chapter $6$, Section $5.1$]{studyindag2}.
\end{remark}
\begin{construction}
	\label{construction:hopfalgebrastructureonUg}
	Let $\gfrak$ be a classical (i.e., discrete) Lie algebra defined over a field $\Bbbk$. Its enveloping algebra $\Ug$ is naturally endowed with a cocommutative Hopf structure over $\Bbbk$. The morphims of $\Bbbk$-algebras $\epsilon\colon\Ug\to \Bbbk$, $\mu\colon\Ug\to\Ug\otimes_{\Bbbk}\Ug$ and $\iota\colon\Ug\to\Ug$ which provide respectively the counit, the comultiplication and the coinversion for such Hopf structure arise in the following way.
	\begin{itemize}
		\item[$1.$] The counit $\epsilon\colon\Ug\to\operatorname{U}(0)\simeq\Bbbk$ is induced by applying the functor $\operatorname{U}$ to the trivial map of Lie algebras $\gfrak\to0$.
		\item[$2.$]The comultiplication $\mu\colon\Ug\to\operatorname{U}{\lp\gfrak\times\gfrak\rp}\simeq\Ug\otimes_{\Bbbk}\Ug$ is induced by applying the functor $\operatorname{U}$ to the diagonal map $\Delta\colon\gfrak\hookrightarrow\gfrak\times\gfrak$.
		\item[$3.$]The coinversion $\iota\colon\Ug\to\Ug$ is induced by applying the functor $\operatorname{U}$ to the inversion map  $-\operatorname{id}_{\gfrak}\colon\gfrak\to\gfrak$. 
	\end{itemize}
	We can adapt this idea to the context of \infinity-categories, and make the universal enveloping algebra of a Lie algebra in $\scrC$ into a Hopf algebra in $\scrC$ (see \cite[Appendix B]{rationalhomotopytheory} for the case of differential graded Lie algebras or, for a more modern and general perspective, \cite[Chapter $6$, Section $5.1.4$]{studyindag2}). Given a symmetric monoidal stable \infinity-category $\scrC$, we can consider algebras for the associative operad $\Ebb_1$ (i.e., associative algebras) and algebras for the Lie operad. Both these \infinity-categories of algebras are symmetric monoidal, with the monoidal structure given by the Cartesian monoidal structure on $\Alg_{\Lie}(\scrC)$ and by the usual tensor product of algebras in $\Alg_{\Ebb_1}(\scrC)$. The universal enveloping algebra \infinity-functor $\operatorname{U}\colon\Alg_{\Lie}(\scrC)\to\Alg_{\Ebb_1}(\scrC)$ is then a \textit{strongly monoidal \infinity-functor} for these symmetric monoidal structures (\cite[Chapter $6$, Lemma $5.2.8$]{studyindag2}), hence it induces an \infinity-functor at the level of cocommutative coalgebras$$\operatorname{U}\colon\operatorname{cCcAlg}{\lp\Alg_{\Lie}(\scrC)\rp}\longrightarrow\operatorname{cCcAlg}{\lp\Alg_{\Ebb_1}(\scrC)\rp}\eqqcolon \operatorname{cCBAlg}(\scrC)$$where the target is the \infinity-category of cocommutative coalgebras in associative algebras, i.e., \textit{cocommutative bialgebras}. Since the monoidal structure on $\Alg_{\Lie}(\scrC)$ is Cartesian, there is an equivalence of $\infinity$-categories
	\begin{align}
	\label{equiv:everyliealgebracocommcoalg}
	\Alg_{\Lie}(\scrC)\simeq \operatorname{cCcAlg}{\lp\Alg_{\Lie}(\scrC)\rp}
	\end{align}
	in virtue of \cite[Proposition $2.4.3.9$]{ha}: essentially, the diagonal morphism $\Delta\colon\gfrak\to\gfrak\times\gfrak$ turns every Lie algebra into a cocommutative coalgebra object in Lie algebras. By precomposing with the equivalence \ref{equiv:everyliealgebracocommcoalg}, we get our universal enveloping algebra \infinity-functor with its bialgebra structure
	\begin{align}
	\label{functor:universalenvelopingbialgebra}
	\operatorname{U}\colon\Alg_{\Lie}(\scrC)\simeq\operatorname{cCcAlg}{\lp\Alg_{\Lie}(\scrC)\rp}\longrightarrow\operatorname{cCBAlg}(\scrC).
	\end{align}The fact that the \infinity-functor $\operatorname{U}$ factors through the \infinity-category of Hopf algebras follows from the fact that$$\operatorname{cCBAlg}(\scrC)\coloneqq\operatorname{cCcAlg}{\lp\Alg_{\Ebb_1}(\scrC)\rp}\simeq \Alg_{\Ebb_1}{\lp\operatorname{cCcAlg}(\scrC)\rp}\eqqcolon \operatorname{Mon}{\lp\operatorname{cCcAlg}(\scrC)\rp},$$since the monoidal structure on $\operatorname{cCcAlg}(\scrC)$ is given by the tensor product in $\scrC$, which coincides with Cartesian monoidal structure for coalgebras, see\cite[Chapter $6$, Appendix C$.1.1.$]{studyindag2}. So the universal enveloping algebra \infinity-functor preserves products of Lie algebras, hence gives rise to a \infinity-functor at the level of group objects$$\operatorname{U}\colon\operatorname{Grp}{\lp\Alg_{\Lie}(\scrC)\rp}\longrightarrow\operatorname{Grp}{\lp\operatorname{cCcAlg}(\scrC)\rp}.$$Since looping and delooping provides an equivalence $\Alg_{\Lie}(\scrC)\simeq\operatorname{Grp}{\lp\Alg_{\Lie}(\scrC)\rp}$ for any symmetric monoidal stable \infinity-category $\scrC$ (\cite[Proposition $1.6.4$]{studyindag2}), and that a Hopf algebra is precisely a group object in cocommutative coalgebras, we obtain the Hopf structured we asked for.
\end{construction}
\subsection{Representations of Lie algebras}
\label{sec:replie}
In this subsection, we briefly review the definition and properties of representations of a Lie algebra. The most convenient and straight-forward definition for a representation  of a Lie algebra in the derived setting is the following.
\begin{defn}
	\label{def:repg}
	A \textit{representation of a Lie algebra $\gfrak$ over $\Bbbk$} is a left $\Ug$-module.
	The \textit{$\infinity$-category of representations of a Lie algebra $\gfrak$} is then the $\infinity$-category $\LMod_{\Ug}$ of left $\Ug$-modules, and shall be denoted by $\Rep_{\gfrak}$.
\end{defn}
\begin{remark}
	\label{remark:leftUgmodulesandgrepresentations}
	The fact that this definition, motivated by the equivalence of left modules on the enveloping algebra of a classical (discrete) Lie algebra $\gfrak$ and classical representations of $\gfrak$, is still the right notion for derived representations of derived Lie algebras is explained by the following argument. Consider a differential graded Lie algebra $\gfrak_{\bullet}$ over $\Bbbk$; then one can consider, in analogy to the classical setting, a \textit{differential graded representation of $\gfrak_{\bullet}$} to be a chain complex of $\Bbbk$-modules $V_{\bullet}$ endowed with a left action$$\gfrak_{\bullet}\otimes_{\Bbbk}V_{\bullet}\longrightarrow V_{\bullet}$$ such that $$[x,y]\cdot v=x\cdot(y\cdot v)+(-1)^{pq}y\cdot(x\cdot v)$$ for any $x\in\gfrak_p$ and $y\in\gfrak_q$.\ 
	Differential graded representations of a differential graded Lie algebra $\gfrak_{\bullet}$ are gathered in a category, which we shall denote with $\dgRep_{\gfrak_{\bullet}}$. Similarly to the case of differential graded Lie algebras, also $\dgRep_{\gfrak_{\bullet}}$ is endowed with a model structure whose weak equivalences and fibrations are detected by the forgetful functor $\dgRep_{\gfrak_{\bullet}}\longrightarrow\dgMod_{\Bbbk}$ (\cite[Proposition $2.4.5$]{dagx}). By Dwyer-Kan localization with respect to the class $\mathscr{W}$ of weak equivalences yields an $\infinity$-category ${\operatorname{N}_{\Delta}}{\lp\operatorname{L}^{\operatorname{H}}\lp\dgRep_{\gfrak_{\bullet}},\mathscr{W}\rp\rp}$ which we simply denote by $\Rep_{\gfrak}$.\\
	It is known (\cite[Theorem $5.4$]{stronglyhomotopyliealgebras}) that giving a differential graded representation $V_{\bullet}$ of a differential graded Lie algebra $\gfrak_{\bullet}$ over $\Bbbk$ is equivalent to giving a morphism of differential graded Lie algebras$$\gfrak_{\bullet}\longrightarrow{\operatorname{L}}{\lp{\End_{\Bbbk}}{\lp V_{\bullet}\rp}\rp}$$where ${\operatorname{L}}{\lp{\End_{\Bbbk}}{\lp V_{\bullet}\rp}\rp}$ is the associative differential graded $\Bbbk$-algebra of endomorphisms of $V_{\bullet}$ thought as a differential graded Lie algebra via the right adjoint \ref{adjunction:envelopingalgebra/lie}. In particular, because of that adjunction, it follows that giving a morphism of associative differential graded $\Bbbk$-algebras $\gfrak_{\bullet}\to{\operatorname{L}}{\lp{{\End_{\Bbbk}}{\lp V_{\bullet}\rp}}\rp}$ is equivalent to giving a differential graded $\Bbbk$-algebra morphism ${\operatorname{U}}{\lp\gfrak_{\bullet}\rp}\to{\End_{\Bbbk}}{\lp V_{\bullet}\rp}$, which is equivalent (see for instance \cite[Section $4.7.1$]{ha}) to giving a left ${\operatorname{U}}{\lp\gfrak_{\bullet}\rp}$-module structure to the chain complex of $\Bbbk$-modules $V_{\bullet}$ in the category $\dgMod_{\Bbbk}$. In particular, it follows that the $\infinity$-category $\Rep_{\gfrak}$ and the $\infinity$-category $\LMod_{\Ug}$ of left $\Ug$-modules are equivalent one to the other as stable $\Bbbk$-linear $\infinity$-categories, for any Lie algebra $\gfrak$ over $\Bbbk$, hence the motivation for Definition \ref{def:repg}.
\end{remark}
\begin{remark}
	\label{remark:leftorright}
	Given a Lie algebra $\gfrak$, then $\Ug$ is equivalent as an associative algebra to its opposite $\Ug^{\operatorname{rev}}$ (in the sense of \cite[Remark $4.1.1.7$]{ha}). Indeed, the antipode involution $-\operatorname{id}\colon\gfrak\to\gfrak$ induces an equivalence $\Ug\overset{\simeq}{\to}\Ug^{\operatorname{rev}}$. In particular, pulling back along such equivalence induces a natural equivalence of \infinity-categories $$\LMod_{\Ug}\simeq\RMod_{\Ug^{\operatorname{rev}}}\overset{\simeq}{\longrightarrow}\RMod_{\Ug}.$$This implies that any representation of a Lie algebra $\gfrak$ can be equivalently interpreted as a left or right $\Ug$-module, depending on our need.
\end{remark}
\begin{parag}
	Classically, given a Hopf algebra $H$ over a base commutative ring $\Bbbk$ (such as the universal enveloping algebra of a Lie algebra $\gfrak$ over a base field of characteristic $0$) one can define a monoidal structure on the abelian $1$-category of left $H$-modules such that the forgetful functor$$\oblv_H\colon\LMod_H \longrightarrow\Mod_{\Bbbk}$$is strongly monoidal. The left $H$-action on the tensor product $M\otimes_{\Bbbk}N$ is induced by the comultiplication $\mu\colon H\to H\otimes_{\Bbbk}H.$ Moreover, such monoidal structure on $\LMod_H$ is \textit{symmetric} monoidal if $H$ is a \textit{cocommutative} Hopf algebra. The above discussion is actually part of a larger framework, namely \textit{Tannaka duality}, which affirms that any symmetric monoidal category with a strongly monoidal fiber functor to the abelian category of $\Bbbk$-vector spaces is equivalent to the category of left modules on some cocommutative bialgebra (a nice account of this theory is provided in \cite{tannakaduality}). While we do not need the whole Tannaka duality picture for Hopf algebras, we shall need the symmetric monoidal structure on the \infinity-category of $\LMod_{A}$, for $A$ a cocommutative $\Bbbk$-bialgebra. This has been spelled out in great generality in \cite[Section $3.1$]{bialgebrasmonoidalstructure}, even if we need only the following special case.
\end{parag}
\begin{construction}[{\cite[Corollary $3.19$]{bialgebrasmonoidalstructure}}]
	\label{construction:monoidalstructureonLModUG}
	Let $\gfrak$ be a Lie algebra object in a symmetric monoidal stable $\infinity$-category $\scrC$, which is $\Bbbk$-linear over a field $\Bbbk$ of characteristic $0$. Then the \infinity-category $\LMod_{\Ug}(\scrC)$ of left $\Ug$-modules in $\scrC$ is the \infinity-category of left modules over a Hopf algebra, hence of a cocommutative bialgebra: in particular, it is endowed with a symmetric monoidal structure and with a strongly monoidal forgetful \infinity-functor$$\oblv_{\Ug}\colon\LMod_{\Ug}(\scrC)\longrightarrow\scrC.$$In particular, the action of $\Ug$ on the underlying object $M\otimes_{\scrC}N$ of $M\otimes_{\gfrak}N$ is given by$$\Ug\otimes_{\scrC}M\otimes_{\scrC}N\overset{\mu^*}{\longrightarrow}\Ug\otimes_{\scrC}\Ug \otimes_{\scrC}M\otimes_{\scrC}N\simeq \Ug\otimes_{\scrC}M\otimes_{\scrC}\Ug\otimes_{\scrC}N\xrightarrow{\alpha_M\otimes\alpha_N}M\otimes_{\scrC}N$$where $\mu\colon\Ug\to\Ug\otimes_{\scrC}\Ug$ is the Hopf algebra comultiplication and $\alpha_M$, $\alpha_N$ are the two natural left $\Ug$-actions (\cite[Remark $3.20$]{bialgebrasmonoidalstructure}). In the following, we shall refer to such monoidal structure as the \textit{tensor product of left $\Ug$-modules}. Abusing notations, fixed a left $\Ug$-module $M$, we shall refer to the \infinity-functor$$-\otimes_{\gfrak}M\colon \LMod_{\Ug}(\scrC)\longrightarrow\LMod_{\Ug}(\scrC)$$as the \textit{tensor product with $M$.}
\end{construction}
\begin{remark}
	\label{remark:ugmapleftmodules}
	In the situation of \cref{construction:monoidalstructureonLModUG}, given a Lie algebra object $\gfrak$ in a symmetric monoidal \infinity-category $\scrC$ which is $\Bbbk$-linear over a field $\Bbbk$ of characteristic $0$, the comultiplication map $\mu\colon\Ug\to\Ug\otimes_{\scrC}\Ug$ is \textit{always} a morphism of left $\Ug$-modules, i.e., it is actually a map $$\mu\colon\Ug\longrightarrow\Ug\otimes_{\gfrak}\Ug.$$ Indeed, the fact that $\mu$ is a map of left $\Ug$-modules is equivalent to saying that it is compatible with the left action of $\Ug$ on itself, i.e., with the multiplication of $\Ug$. But now \cite[Proposition C.$1.3$]{studyindag2} assures us that the \infinity-category $\operatorname{cBAlg}(\scrC)$ of cocommutative bialgebras of $\scrC$ can be described equivalently either as $\operatorname{Alg}{\lp\operatorname{cCcAlg}(\scrC)\rp}$ or as $\operatorname{cCcAlg}{\lp\Alg(\scrC)\rp},$ hence we have the desired compatibility.
\end{remark}
\subsection{Classical Chevalley-Eilenberg complexes}
\label{sec:classicalCE}Classically, homology and cohomology of a Lie algebra $\gfrak$ can be interpreted respectively in terms of $\Tor$ and $\Ext$ groups, but in concrete computations is often more convenient to deal with standard, explicit presentations at the level of chains - the \textit{homological} and \textit{cohomological Chevalley-Eilenberg complexes}, respectively. In this subsection, we shall briefly review the classical construction of the Chevalley-Eilenberg complexes, providing motivations and ideas to generalize this construction at the \infinity-categorical level.
\begin{parag}
	\label{parag:projectiveres}
	Let $\gfrak$ be a classical finite dimensional Lie algebra defined over $\Bbbk$. We can consider the graded $\Bbbk$-module$$\operatorname{V}_{\bullet}(\gfrak)\coloneqq\Ug\otimes_{\Bbbk}\bigwedge\nolimits^{\!\bullet}\gfrak$$endowed with differential\begin{align*}
	u\otimes\lp g_1\wedge\ldots\wedge g_n\rp&\mapsto\sum_{1\leqslant i\leqslant n}(-1)^{i+1}u\cdot g_i\otimes\lp g_1\wedge\ldots\wedge\widehat{g}_i\wedge\ldots\wedge g_n\rp\\&+\sum_{1\leqslant i<j\leqslant n}(-1)^{i+j}u\otimes \lp g_1\wedge\ldots\wedge\widehat{g}_i\wedge\ldots\wedge g_{j-1}\wedge[g_i,g_j]\wedge\ldots\wedge g_n\rp.
	\end{align*} 
\end{parag}
\begin{defn}Let $\rho\colon \gfrak\to \operatorname{M}_n(M)$ be a representation of $\gfrak$ in some free classical $\Bbbk$-module $M$ that, in virtue of \cref{remark:leftorright}, we can view both as a left or right $\Ug$-module.
	\label{def:CE} 
	\begin{enumerate}[label=(\arabic{enumi}), ref=\thelemman.\arabic{enumi}]
		\item \label{def:CEhom} The \textit{homological Chevalley-Eilenberg complex of $\gfrak$ with coefficients in $M$} is the chain complex $$\CE_{\bullet}{\lp\gfrak;\hsp M\rp}\coloneqq  M\otimes_{\Ug}\operatorname{V}_{\bullet}(\gfrak).$$The \textit{homology of $\gfrak$ with coefficients in $M$} is the graded vector space $\operatorname{H}_{\bullet}{\lp\CE_{\bullet}\lp\gfrak;\hsp M\rp\rp}$, and we will simply denote it as $\operatorname{H}_{\bullet}\lp\gfrak;\hsp M\rp.$
		\item \label{def:CEcohom} The \textit{cohomological Chevalley-Eilenberg complex of $\gfrak$ with coefficients in $M$} is the chain complex $$\CE^{\bullet}{\lp\gfrak;\hsp M\rp}\coloneqq\Hom_{\Ug}{\lp \operatorname{V}_{\bullet}(\gfrak),\hsp M\rp}.$$The \textit{cohomology of $\gfrak$ with coefficients in $M$} is the graded vector space $\operatorname{H}^{\bullet}{\lp\CE^{\bullet}\lp\gfrak;\hsp M\rp\rp}$, and we sill simply denote it as $\operatorname{H}^{\bullet}\lp\gfrak;\hsp M\rp$
	\end{enumerate}
\end{defn}
\begin{notation}
	If $M=\Bbbk$, we shall refer to $\Bbbk\otimes_{\Ug}\operatorname{V}_{\bullet}(\gfrak)$ (respectively, $\Hom_{\Ug}{\lp \operatorname{V}_{\bullet}(\gfrak),\hsp \Bbbk\rp}$) simply as the \textit{homological} (respectively, the \textit{cohomological}) \textit{Chevalley-Eilenberg complex of $\gfrak$}, and denote it simply as $\CE_{\bullet}(\gfrak)$ (respectively, $\CE^{\bullet}(\gfrak)$).
\end{notation}
\begin{remark}
	\label{remark:CEalgebras}
	It is well-known (see \cite[Theorem $7.7.2$]{weibel}) that the natural augmentation $\operatorname{V}_{\bullet}(\gfrak)\to\Bbbk$ induced by the zero morphism $\gfrak \to\Bbbk$ exhibits $\operatorname{V}_{\bullet}(\gfrak)$ as a projective resolution of the trivial representation $\Bbbk$. This means that $\operatorname{CE}_{\bullet}(\gfrak;\hsp M)$ and $\operatorname{CE}^{\bullet}(\gfrak;\hsp M)$ provide explicit models for the derived tensor product $\Bbbk\otimes^{\mathbb{L}}_{\Ug}\Bbbk$ and the derived hom-space $\RR\Hom_{\Ug}{\lp\Bbbk,\hsp\Bbbk\rp}$, respectively. In particular,  the homology and cohomology of Lie algebras can be expressed using the language of derived functors as$$\operatorname{H}_{\bullet}(\gfrak;\hsp M)\coloneqq\Tor^{\Ug}_{\bullet}(\Bbbk,\hsp M)$$and$$\operatorname{H}^{\bullet}(\gfrak;\hsp M)\coloneqq\Ext^{\bullet}_{\Ug}{\lp\gfrak,\hsp M\rp}.$$Moreover, in the case $M=\Bbbk$, unraveling all the definitions one can recover the usual isomorphisms of graded $\Bbbk$-modules $$\CE_{\bullet}(\gfrak)=\Bbbk\otimes_{\Ug}\Ug\otimes_{\Bbbk}\bigwedge\nolimits^{\!\bullet}\gfrak\cong \bigwedge\nolimits^{\!\bullet}\gfrak$$and$$\CE^{\bullet}(\gfrak)=\Hom_{\Ug}{\lp\Ug\otimes_{\Bbbk}\bigwedge\nolimits^{\!\bullet}\gfrak,\hsp\Bbbk\rp}\cong\Hom_{\Bbbk}{\lp\bigwedge\nolimits^{\!\bullet}\gfrak,\hsp\Bbbk\rp}\cong\bigwedge\nolimits^{\!-\bullet}\gfrak^{\vee}.$$In general these isomorphisms cannot be promoted to isomorphisms of chain complexes, unless the Lie algebra is abelian.
\end{remark}
\begin{parag}\label{parag:koszulduality}
	It is however true that the usual comultiplication and multiplication of $\bigwedge\nolimits^{\!\bullet}\gfrak$ and $\bigwedge\nolimits^{\!-\bullet}\gfrak^{\vee}$ are compatible with the differential, turning $\CE_{\bullet}(\gfrak)$ and $\CE^{\bullet}(\gfrak)$ into a differential graded cocommutative coalgebra and into a differential graded commutative algebra, respectively. It is moreover true that the set of Lie algebra structures on $\gfrak$ is in bijective correspondence with the set of differential graded cocommutative coalgebra structures of $\bigwedge\nolimits^{\!\bullet}\gfrak$: in fact, a morphism $$\operatorname{d}_{\bullet}\colon\bigwedge\nolimits^{\!\bullet}\gfrak\longrightarrow\bigwedge\nolimits^{\!\bullet-1}\gfrak$$squares to zero if and only if $[-,-]\coloneqq\operatorname{d}_2\colon\gfrak\otimes\gfrak\to\gfrak$ is anti-symmetric, and it respects the coalgebra structure precisely if it satisfies the Leibniz rule. This statement can be rephrased as follows.
	\begin{propositionn}
		\label{prop:vogliamogeneralizzare}
		Let $\Lie_{\Bbbk}^{\operatorname{disc}}$ be the $1$-category of discrete Lie algebras. Then the Chevalley-Eilenberg functor can be a promoted to a functor with values in differential graded cocommutative coalgebras$$\operatorname{CE}_{\bullet}\colon\Lie_{\Bbbk}^{\operatorname{disc}}\longrightarrow\operatorname{dgcCcAlg}_{\Bbbk}.$$Such functor is fully faithful, and yields an equivalence of categories between the category of discrete Lie algebras and the full subcategory of differential graded cocommutative coalgebras which are semi-free (i.e., isomorphic as graded cocommutative coalgebras to the Grassmann algebra over some $\Bbbk$-vector space).
	\end{propositionn}
	An analogous statement can be made taking into account the commutative algebra structures of $\bigwedge\nolimits^{\!-\bullet}\gfrak$, once we add the assumption for $\gfrak$ to be finitely generated as a $\Bbbk$-module. In this case, the functor
	$$\operatorname{CE}^{\bullet}\colon\Lie_{\Bbbk}^{\operatorname{disc}}\longrightarrow\dgCAlg_{\Bbbk}$$is valued in differential graded commutative algebras.
\end{parag}
\begin{parag}
	The previous discussion can be made, almost \textit{verbatim}, also in the context of differential graded Lie algebras, using \cite[Section $2.2$]{dagx}. Using the fact that for any discrete vector space $M$ there is an isomorphism of non-negatively graded vector spaces$$\bigwedge\nolimits^{\!\bullet}M\cong\Sym^{\bullet}_{\Bbbk}(M[1])$$it is possible to define the \textit{homological Chevalley-Eilenberg complex of a differential graded Lie algebra $\gfrak_{\bullet}$ with values in a differential graded representation $M_{\bullet}$} as the chain complex$$\CE_{\bullet}{\lp\gfrak_{\bullet};\hsp M_{\bullet}\rp}\coloneqq M_{\bullet}\otimes_{\Bbbk}\Sym_{\Bbbk}^{\bullet}{\lp\gfrak_{\bullet}[1]\rp}$$with differential given by the formula
	\begin{equation}
	\label{eq:CEdifferential}
	\begin{aligned}
	m\otimes\lp g_1\ldots g_n\rp&\mapsto\sum_{1\leqslant i\leqslant n}(-1)^{\left|x_1\right|+\ldots+\left|x_{i-1}\right|}m\otimes\lp g_1\ldots\operatorname{d}_{\gfrak}(g_i)\ldots g_n\rp\\
	&+(-1)^{\left|x_i\right|\lp\left|x_{i+1}\right|+\ldots+\left|x_{n}\right|\rp}\operatorname{d}_{M}(m)\otimes \lp g_1\ldots g_n\rp\\
	&+\sum_{1\leqslant i\leqslant n}(-1)^{\left|x_i\right|\lp\left|x_{i+1}\right|+\ldots+\left|x_{n}\right|\rp}m\cdot g_i\otimes\lp g_1\ldots\widehat{g}_i\ldots\wedge g_n\rp\\
	&+\sum_{1\leqslant i<j\leqslant n}(-1)^{\left|x_i\right|\lp\left|x_{i+1}\right|+\ldots+\left|x_{j-1}\right|\rp}m\otimes \lp g_1\ldots\widehat{g}_i\ldots g_{j-1}\cdot [g_i,g_j]\wedge\ldots g_n\rp.
	\end{aligned}
	\end{equation}
	In the same way, the \textit{cohomological Chevalley-Eilenberg complex of a differential graded Lie algebra $\gfrak_{\bullet}$ with coefficients in a differential graded representation $M_{\bullet}$} is the chain complex$$\CE^{\bullet}{\lp\gfrak_{\bullet};\hsp M_{\bullet}\rp}\coloneqq \Hom_{\Bbbk}{\lp\Sym_{\Bbbk}{\lp\gfrak_{\bullet}[1]\rp},\hsp M_{\bullet}\rp}.$$Again, homological and cohomological Chevalley-Eilenberg complexes with coefficients in the trivial $\operatorname{U}(\gfrak_{\bullet})$-module $\Bbbk$ inherit a cocommutative coalgebra and a commutative algebra structure, respectively, from the commutative and cocommutative bialgebra structure of the symmetric algebra. Moreover, they both preserve quasi-isomorphisms of differential graded Lie algebras, hence they provide \infinity-functors$$\CE_{\bullet}\colon\Lie_{\Bbbk}\longrightarrow\operatorname{cCcAlg}_{\Bbbk}$$and$$\CE^{\bullet}\colon\Lie_{\Bbbk}\longrightarrow\CAlg_{\Bbbk}.$$
\end{parag}
\begin{warning}
	There is \textit{no hope} that these \infinity-functors are again fully faithful: passing to the \infinity-categorical setting, we identify strict commutative differential graded algebras with isomorphic homology $\Bbbk$-vector spaces. Thus, for example, the Lie algebras $\Bbbk$ with abelian Lie structure and $\Bbbk\oplus\Bbbk$ with its essentially unique non-abelian Lie bracket are both mapped, via the cohomological Chevalley-Eilenberg \infinity-functor $\CE^{\bullet}$, onto the square-zero extension $\Bbbk\oplus\Bbbk[-1]$, even if they are obviously not equivalent as homotopy Lie algebras. In order to solve this issue, we need to keep track of the datum of the \textit{filtration} of the Chevalley-Eilenberg complexes. This is the content of \cref{sec:homologicalCE} and \cref{sec:cohomCE}.
\end{warning}
\section{Mixed graded Chevalley-Eilenberg modules}
\label{chapter:liealgebras}
In this section, which is the core of the paper, we shall describe how Chevalley-Eilenberg chain complexes of differential graded Lie algebras (in the sense of \cite{dagx}) are actually the shadow of objects endowed with a much richer structure. Namely, they arise as the Tate realization of mixed graded $\Bbbk$-modules, which are moreover either cocommutative coalgebras (in the case of the homological complex) or commutative algebras (in the case of the cohomological complex) in mixed graded $\Bbbk$-modules. This fact is a folklore result expected, if not known, by many experts (the most explicit formulation appears in \cite[Appendix B]{CG}); however, the mixed graded structure on the Chevalley-Eilenberg algebra and coalgebra of a derived Lie algebra $\gfrak$ has until now appeared only in terms of explicit graded chain complexes equipped with mixed differentials, heavily relying on strict model choices. The main result of this section is the (completely \infinity-categorical) construction of \infinity-functors$$\CE_{\varepsilon}\colon\Lie_{\Bbbk}\longrightarrow\CoCommmixgraug,$$$$\CE^{\varepsilon}\colon\Lie^{\op}_{\Bbbk}\longrightarrow\CAlgmixgr_{\Bbbk//\Bbbk}$$and$$\CE^{\varepsilon}(\gfrak;\hsp-)\colon\LMod_{\Ug}\longrightarrow\ModCEmixgr,$$computing respectively the homological Chevalley-Eilenberg mixed graded cocommutative coalgebra (\cref{prop:promotionCEcocommutativealgebra}), the cohomological Chevalley-Eilenberg mixed graded commutative algebra (\cref{prop:CEcohomologicalalgebra}), and the Chevalley-Eilenberg cohomology mixed graded $\Bbbk$-module of a fixed Lie algebra $\gfrak$ with values in any representation (\cref{prop:CEmod}). The construction is quite technical, but we shall emphasize the main conceptual tools and motivations of our construction. Finally, employing the theory of mixed graded $\Bbbk$-modules developed in \cite{pavia1} (and briefly recalled in \cref{chapter:mixedgradedmodules}), we shall study some formal properties of the Chevalley-Eilenberg mixed graded \infinity-functors.
\subsection{The homological Chevalley-Eilenberg $\infty$-functor}
\label{sec:homologicalCE}
We shall now provide a mixed graded (hence, filtered) structure on the Chevalley-Eilenberg cocommutative coalgebra of a Lie algebra $\gfrak$ (\cref{def:CE}).  Our main objective in this subsection is to show that the left $\Ug$-module $\Ug\otimes_{\Bbbk}{\Sym_{\Bbbk}}{\lp\gfrak[-1]\rp}$, which is classically a projective resolution as a chain complex of left $\Ug$-modules of the discrete left $\Ug$-module $\Bbbk$, is naturally endowed with a mixed graded (hence, filtered) structure. What is more important, such mixed graded structure is naturally induced by the mixed graded structure on the \textit{cone} of $\gfrak$ (\cite[Construction $2.2.1$]{dagx}), which here naturally appears as some right adjoint to the \infinity-functor $(-)_0\colon\Liemixgr_{\Bbbk}\to\Lie_{\Bbbk}.$ All the Chevalley-Eilenberg $\Bbbk$-modules $\CE_{\bullet}(\gfrak;\hsp N)$ and $\CE^{\bullet}(\gfrak;\hsp M)$ for a right $\Ug$-module $N$ or a left $\Ug$-module $M$ (in particular, for $M=N=\Bbbk$) hence inherit a mixed graded / filtered structure as well, naturally induced by the one on $\Ug\otimes_{\Bbbk}{\Sym_{\Bbbk}}{\lp\gfrak[-1]\rp}$.
\begin{parag}
	\label{parag:const}
	Let ${\Fun}{\lp\Delta^1,\hsp\Lie_{\Bbbk}\rp}$ be the \infinity-category of morphisms in $\Lie_{\Bbbk}$, and consider the constant $\infinity$-functor\begin{align}
	\label{functor:constant}
	\operatorname{const}\colon\Lie_{\Bbbk}\longrightarrow{\Fun}{\lp\Delta^1,\hsp \Lie_{\Bbbk}\rp},
	\end{align} right adjoint to the evaluation on the codomain and left adjoint to the evaluation on the domain, informally described by the assignation $$\gfrak\mapsto\left\{\gfrak\xrightarrow{\operatorname{id}_{\gfrak}}\gfrak\right\}.$$Thanks to the discussion of \ref{parag:functorsthatpreservealgebras}, we know that the fully faithful embedding \ref{functor:insertioninweightq}, for $q=0$, induces an $\infinity$-functor at the level of the $\infinity$-categories of Lie algebra objects
	\begin{align}
	\label{functor:liealgebrainweight0}
	(-)(0)\colon\Lie_{\Bbbk}\longrightarrow {\Alg_{\Lie}}{\lp\Modmixgr_{\Bbbk}\rp}\eqqcolon\Liemixgr_{\Bbbk}.
	\end{align}
	The $\infinity$-functor \ref{functor:liealgebrainweight0} simply sends a Lie algebra $\gfrak$ to the graded Lie algebra $\gfrak(0)$ with same Lie bracket, and trivial mixed structure. In particular, since the \infinity-functor $(-)_0$ is an inverse for the \infinity-functor $(-)(0)$, and it is moreover both lax and oplax monoidal (as already stated in \ref{parag:functorsthatpreservealgebras}), we can think the \infinity-functor \ref{functor:constant} to land in the \infinity-category$${\Fun}{\lp\Delta^1,\hsp\Lie_{\Bbbk}\rp}_{(-)_0/}\simeq{\Fun}{\lp\Delta^1,\hsp\Lie_{\Bbbk}\rp}$$of morphisms in $\Lie_{\Bbbk}$ such that the domain lies in the essential image of the \infinity-functor $(-)_0\colon\Liemixgr_{\Bbbk}\to\Lie_{\Bbbk}$. 
\end{parag}
\begin{parag}
	\label{parag:cng}
	Recall now that the \infinity-functor $(-)_0\colon\Modmixgr_{\Bbbk}\to\Mod_{\Bbbk}$ admits both a left and right adjoint, and its right adjoint is given by  $$\operatorname{R}_{\varepsilon}\circ\hsp (-)(0)\colon\Mod_{\Bbbk}\longrightarrow\Modmixgrcn_{\Bbbk}\subseteq\Modmixgr_{\Bbbk},$$where $\operatorname{R}_{\varepsilon}\colon\Modgr_{\Bbbk}\to\Modmixgr_{\Bbbk}$ is the \infinity-functor described in \cref{porism:rightadjointforgetful}. Since it is right adjoint to an oplax monoidal \infinity-functor, $\operatorname{R}_{\varepsilon}\circ\hsp (-)(0)$ is lax monoidal itself. The lax monoidal transformation
	\begin{align}
	\label{map:laxmonoidal}
	\operatorname{R}_{\varepsilon}(M(0))\otimesmixgr_{\Bbbk}\operatorname{R}_{\varepsilon}(N(0))\longrightarrow\operatorname{R}_{\varepsilon}\lp M(0)\otimes^{\gr}_{\Bbbk}N(0)\rp,
	\end{align}
	is, by definition, the map adjoint to the map given by tensoring the unit for $M$ and $N$ $$\lp\operatorname{R}_{\varepsilon}(M(0))\otimesmixgr_{\Bbbk}\operatorname{R}_{\varepsilon}(N(0))\rp_0\simeq \operatorname{R}_{\varepsilon}(M(0))_0\otimes_{\Bbbk}\operatorname{R}_{\varepsilon}(N(0))_0\simeq M\otimes_{\Bbbk}N\overset{\simeq}{\longrightarrow}M\otimes_{\Bbbk}N.$$In particular, using the explicit model for $\operatorname{R}_{\varepsilon}$ provided in \cref{porism:rightadjointforgetful}, one can see that the map \ref{map:laxmonoidal} is given by the identity of $M\otimes_{\Bbbk}N$ in weight $0$, and by the codiagonal $M\otimes_{\Bbbk}N\oplus M\otimes_{\Bbbk}N\to M\otimes_{\Bbbk}N$ in weight $1$ (with an appropriate homological shift). In particular,  this right adjoint $\operatorname{R}_{\varepsilon}\circ \hsp(-)(0)$ induces an $\infinity$-functor at the level of Lie algebras
	\begin{align}
	\label{functor:Cng}
	\Cnmixgr\colon\operatorname{Lie}_{\Bbbk}\longrightarrow\Liemixgr_{\Bbbk}.
	\end{align}
	For formal reasons (described in \cite[Chapter $6$, Section $1.2$]{studyindag2}), the \infinity-functor $\Cnmixgr$ is again the right adjoint to the \infinity-functor $(-)_0\colon\Liemixgr_{\Bbbk}\to\Lie_{\Bbbk},$ at least after restricting to non-negatively graded mixed Lie algebras. Using explicit models, such \infinity-functor is given by sending a Lie algebra $\gfrak$ to the mixed graded Lie algebra described in the following way.
	\begin{itemize}
		\item As a mixed graded $\Bbbk$-module, $\Cnmixgr(\gfrak)$ is equivalent to the mixed graded $\Bbbk$-module $\operatorname{R}_{\varepsilon}(\oblv_{\Lie}\gfrak(0))$; in virtue of  \cref{porism:rightadjointforgetful}, this mixed graded $\Bbbk$-module is given by a copy of $\gfrak$ in weight $0$ and a copy of $\gfrak[-1]$ in weight $1$, together with mixed differential given by the identity of $\gfrak[-1]$. This is true because $\Cnmixgr$ is the \infinity-functor induced at the level of algebras for some operad by the lax monoidal \infinity-functor $\operatorname{R}_{\varepsilon}$.
		\item The Lie bracket of $\Cnmixgr(\gfrak)$ is given by the composition of the natural transformation \ref{map:laxmonoidal} with the image under $\operatorname{R}_{\varepsilon}$ of the bracket $[-,-]\colon\gfrak\otimes_{\Bbbk}\gfrak\to\gfrak$. In particular, consider an explicit differential graded Lie algebra $\gfrak$, assumed to be fibrant and cofibrant for the model structure on $\dgLie_{\Bbbk}$, with Lie bracket $[-,-]$: the Lie bracket of $\Cnmixgr(\gfrak)$ is then given by the composition$$\operatorname{R}_{\varepsilon}(\gfrak(0))\otimesmixgr_{\Bbbk}\operatorname{R}_{\varepsilon}(\gfrak(0))\xrightarrow{\ref{map:laxmonoidal}}\operatorname{R}_{\varepsilon}\lp\lp\gfrak\otimes_{\Bbbk}\gfrak\rp(0)\rp\xrightarrow{\operatorname{R}_{\varepsilon}([-,-])}\operatorname{R}_{\varepsilon}(\gfrak(0)).$$In particular, following this composition, we have that the Lie bracket on $\Cnmixgr(\gfrak)$ is given by taking two elements $x+\epsilon y$ and $x'+\epsilon y'$ in $\gfrak\oplus\gfrak[-1]$ and sending them to $[x,x']+\epsilon\lp[x,y']+[y,x']\rp$.
	\end{itemize}
	This has to be understood as a model independent rewriting of the mixed structure on the free loop space $\mathcal{L}(\gfrak)$ of a Lie algebroid $\gfrak$, in the case when $\gfrak$ is an ordinary Lie algebra, described in \cite[Section $6.3$]{nuiten19}. Indeed, in the case of an ordinary Lie algebras, the fiber $\mathfrak{n}$ of the anchor map $\gfrak\to 0$ is all $\gfrak$, hence we have an equivalence of mixed graded Lie algebras $\mathcal{L}(\gfrak)\simeq \Cnmixgr(\gfrak).$ The main ingredient in order to define a mixed structure on any Chevalley-Eilenberg \infinity-functor is already contained in this right adjoint: all mixed structures on Chevalley-Eilenberg algebras, coalgebras and modules are straightly derived from the mixed structure of $\Cnmixgr(\gfrak).$
\end{parag}
\begin{remark}
	Given a Lie algebra $\gfrak$ it is clear that the Tate realization  $\left|\Cnmixgr(\gfrak)\right|^{\operatorname{t}}$  of the mixed graded Lie algebra $\Cnmixgr(\gfrak)$ yields again a Lie algebra (in virtue of \ref{parag:functorsthatpreservealgebras}) which as a explicit differential graded Lie algebra is presented by the contractible differential graded Lie algebra $\operatorname{Cn}(\gfrak)$ of \cite[Construction $2.2.1$]{dagx}. This explains our choice of notation for the $\infinity$-functor \ref{functor:Cng}. 
\end{remark}
\begin{parag}
	Using \cref{remark:universalenvelopingalgebragaitsgory}, we can consider the universal enveloping mixed graded algebra $\infinity$-functor\begin{align}
	\label{functor:mixedgradedug}
	\Umixgr\colon\Liemixgr_{\Bbbk}\longrightarrow\Alg{\lp\Modmixgr_{\Bbbk}\rp}\eqqcolon\Algmixgr_{\Bbbk}.
	\end{align}
	Such $\infinity$-functor sits in a commutative diagram 
	$$\begin{tikzpicture}[scale=0.75]
	\node (a) at (-3,0){$\Liemixgr_{\Bbbk}$};
	\node(b) at (2,0){$\Algmixgr_{\Bbbk}$};
	\node (c) at (-3,-2.5){$\Lie^{\gr}_{\Bbbk}$};
	\node (d) at (2,-2.5){$\Alg^{\gr}_{\Bbbk}$};
	\draw[->,font=\scriptsize] (a) to node[above]{$\operatorname{U}^{\mixgr}$}(b);
	\draw[->,font=\scriptsize] (b) to node[right]{$\oblv_{\varepsilon}$} (d);
	\draw[->,font=\scriptsize] (a) to node[left]{$\oblv_{\varepsilon}$} (c);
	\draw[->,font=\scriptsize] (c) to node[below]{$\operatorname{U}^{\gr}$}(d);
	\end{tikzpicture}$$
	where the bottom arrow is analogously given by applying \cref{remark:universalenvelopingalgebragaitsgory} at the \infinity-category $\Modgr_{\Bbbk}$. The fact that this square is commutative follows from the fact that one can check that the Beck-Chevalley morphism is an equivalence at the level of the underlying graded modules by forgetting both the algebra structure and the mixed structure, since they are both conservative $\infinity$-functors. Then the commutativity becomes a corollary of the commutativity of the square of $\infinity$-functors in \ref{parag:gradedandmixedgradedcoalgebras}.
\end{parag}
As explained in \cite[Remark $2.2.4$]{dagx}, there is an obvious morphism $\gfrak(0)\hookrightarrow\Cnmixgr(\gfrak)$ which, after applying the universal enveloping algebra $\infinity$-functor, equips $\UCngmixgr$ with a left $\Ug(0)$-module structure. Applying the Tate realization $\infinity$-functor and dealing with explicit presentations via chain complexes, one recovers the usual left $\Ug$-module structure on $\left|\UCngmixgr\right|^{\operatorname{t}}$, which is a well known cofibrant replacement for the trivial left $\Ug$-module $\Bbbk$ in $\LMod_{\Ug}$ (see \cite[Theorem $7.7.2$]{weibel} and \cite[Remark $2.2.11$]{dagx}). Since both the $\Ug$ and $\UCngmixgr$ are assignations functorial in $\gfrak$, we need to find a way to keep track of the change of base associative ring and of the change of $\UCngmixgr$ \textit{simultaneously}. The tool we need in order to do so is the \infinity-functorial assignation of the adjoint morphism$$\left\{\gfrak\xrightarrow{\operatorname{id}}\gfrak\right\}\mapsto\left\{\gfrak(0)\xrightarrow{\operatorname{id}^*}\Cnmixgr(\gfrak)\right\}$$via the adjunction $(-)_0\dashv \operatorname{R}_{\varepsilon}(-)(0)$. This will provide an \infinity-functor$${\Fun}{\lp\Delta^1,\hsp\Lie_{\Bbbk}\rp}\simeq{\Fun}{\lp\Delta^1,\hsp\Lie_{\Bbbk}\rp}_{(-)_0/}\longrightarrow{\Fun}{\lp\Delta^1,\hsp\Liemixgr_{\Bbbk}\rp}_{/\Cnmixgr}\subseteq{\Fun}{\lp\Delta^1,\hsp\Liemixgr_{\Bbbk}\rp},$$where again ${\Fun}{\lp\Delta^1,\hsp\Liemixgr_{\Bbbk}\rp}_{/\Cnmixgr}$ is the \infinity-category of morphisms of mixed graded Lie algebras with codomain of the form $\Cnmixgr(\gfrak)$ for some ordinary Lie algebra $\gfrak$. We show how this can be dealt with using the following general machinery.
\begin{construction}
	\label{construction:adjointmorphism}
	Let $\operatorname{L}\colon \scrC\to\scrD$ be an $\infinity$-functor with right adjoint $\operatorname{R}\colon\scrD\to\scrC$. By \cite[Definition $5.2.2.7$ and Corollary $5.2.2.8$]{htt}, the datum of the adjunction $\operatorname{L}\dashv \operatorname{R}$ is encoded by a unit morphism $\eta\colon\operatorname{id}_{\scrC}\Rightarrow \operatorname{R}\circ \operatorname{L}$ in the $\infinity$-category ${\Fun}{\lp\scrC,\hsp\scrC\rp}$, i.e., by an object\begin{align*}
	\eta&\in{\Fun}{\lp\Delta^1,\hsp{\Fun}{\lp\scrC,\hsp\scrC\rp}\rp}\\&\simeq{\Fun}{\lp\Delta^1\times\scrC,\hsp\scrC\rp}\\&\simeq{\Fun}{\lp\scrC,\hsp{\Fun}{\lp\Delta^1,\hsp\scrC\rp}\rp}.
	\end{align*}
	So, let ${\Fun}{\lp\Delta^1,\hsp\scrD\rp}_{\operatorname{L}/}$ be the full sub-$\infinity$-category of ${\Fun}{\lp\Delta^1,\hsp\scrD\rp}$ with domain lying in the image of $\operatorname{L}$. Applying the $\infinity$-functor $\operatorname{R}$, we obtain an $\infinity$-functor $${\Fun}{\lp\Delta^1,\hsp\scrD\rp}_{\operatorname{L}/}\longrightarrow{\Fun}{\lp\Delta^1,\hsp\scrC\rp}_{\operatorname{RL}//\operatorname{R}}.$$The latter $\infinity$-category sits in a $(\infinity,2)$-limit diagram of $\infinity$-categories
	\begin{displaymath}
	\begin{tikzpicture}[scale=0.75]
	\node (a) at (-4,2){${\Fun}{\lp\Delta^1,\hsp\scrC\rp}_{\operatorname{RL}//\operatorname{R}}$};
	\node (b) at (-4,0){${\Fun}{\lp\Delta^1,\hsp\scrC\rp}_{/\operatorname{R}}$};
	\node(c) at (2,2){$\scrC$};
	\node(d) at (2,0){$\scrC$};
	\node at (-3.25,1.25){$\lrcorner$};
	\draw[->,font=\scriptsize] (a) to node [above]{} (b);
	\draw[->] (a) to node[above]{$\Phi$}(c);
	\draw[->,font=\scriptsize] (b) to node [below]{$\operatorname{ev}_0$} (d);
	\draw[->,font=\scriptsize] (c) to node [right]{$\operatorname{R}\circ \operatorname{L}$} (d);
	\end{tikzpicture}
	\end{displaymath}
	Consider the $\infinity$-functor $\Phi\colon{\Fun}{\lp\Delta^1,\hsp\scrC\rp}_{\operatorname{RL}//\operatorname{R}}\to\scrC$: it is the $\infinity$-functor which given a morphism $\operatorname{RL}(X)\to \operatorname{R}(Y)$ selects the object $X$ itself. So we get an $\infinity$-functor$${\Fun}{\lp\Delta^1,\hsp\scrC\rp}_{\operatorname{RL}//\operatorname{R}}\xrightarrow{\lp \Phi,\hsp \operatorname{id}\rp}\scrC\times{\Fun}{\lp\Delta^1,\hsp\scrC\rp}_{\operatorname{RL}//\operatorname{R}}\xrightarrow{\lp\eta,\hsp\operatorname{id}\rp}{\Fun}{\lp\Delta^1,\hsp\scrC\rp}_{/\operatorname{R}}\times{\Fun}{\lp\Delta^1,\hsp\scrC\rp}_{\operatorname{RL}//\operatorname{R}}.$$Informally, this composition is given by the assignation$$\{f\colon \operatorname{RL}(X)\to \operatorname{R}(Y)\}\mapsto\lp X,\hsp f\rp\mapsto \lp \{\eta_X\colon X\to \operatorname{RL}(X)\},\hsp \{f\colon \operatorname{RL}(X)\to \operatorname{R}(Y)\}\rp.$$It is clear that the image of such composition lies in the $\infinity$-category ${\Fun}{\lp\Lambda^2_1,\hsp\scrC\rp}$, and so by the inner horn filling property of any $\infinity$-category we have a well defined composition$${\Fun}{\lp\Lambda^2_1,\hsp\scrC\rp}\longrightarrow{\Fun}{\lp\Delta^2,\hsp\scrC\rp}\xrightarrow{\restr{}{\{0,2\}}}{\Fun}{\lp\Delta^1,\hsp\scrC\rp}.$$This composition determines the adjoint morphism $\infinity$-functor$${\Fun}{\lp\Delta^1,\hsp\scrD\rp}_{\operatorname{L}/}\longrightarrow{\Fun}{\lp\Delta^1,\hsp\scrC\rp}_{/\operatorname{R}}.$$
\end{construction}
By applying \cref{construction:adjointmorphism} to the case of the adjunction $(-)_0\dashv\Cnmixgr$, we get an $\infinity$-functor
\begin{align}
\label{functor:composition2}
\Lie_{\Bbbk}\xrightarrow{\ref{functor:constant}}{\Fun}{\lp\Delta^1,\hsp\Lie_{\Bbbk}\rp}\simeq{\Fun}{\lp\Delta^1,\hsp\Lie_{\Bbbk}\rp}_{(-)_0/}\longrightarrow{\Fun}{\lp\Delta^1,\hsp\Liemixgr_{\Bbbk}\rp}_{/\Cnmixgr}.
\end{align}
\begin{remark}
	\label{remark:CEdifferential}
	This is a slight rewriting of \cite[Remark $2.2.4$]{dagx} in the mixed graded setting. Recall that $\Cnmixgr(\gfrak)$, as a graded $\Bbbk$-module, comes equipped with a decomposition$$\oblv_{\Lie}\oblv_{\varepsilon}{\lp\Cnmixgr(\gfrak)\rp}\simeq\oblv_{\Lie}\gfrak(0)\oplus\oblv_{\Lie}\gfrak[-1](1).$$The natural inclusion of graded $\Bbbk$-modules $\oblv_{\Lie}\gfrak[-1](1)\hookrightarrow\oblv_{\Lie}\oblv_{\varepsilon}{\lp\Cnmixgr(\gfrak)\rp}$ agrees straight-forwardly with the Lie bracket of $\gfrak[-1](1)$ and $\Cnmixgr(\gfrak)$, hence can be promoted to a morphism of graded Lie algebras$$\gfrak[-1](1)\longhookrightarrow\oblv_{\varepsilon}{\lp\Cnmixgr(\gfrak)\rp}.$$Applying the graded universal enveloping algebra \infinity-functor, we have a morphism of graded associative algebras (and, in particular, of graded $\Bbbk$-modules)\begin{align}
	\label{map:ucng}
	\operatorname{U}^{\gr}(\gfrak[-1](1))\simeq\Symgr_{\Bbbk}{\lp\oblv_{\Lie}\gfrak[-1](1)\rp}\longrightarrow\operatorname{U}^{\gr}{\lp\oblv_{\varepsilon}\Cnmixgr(\gfrak)\rp}
	\end{align}which in turn induces by adjunction a map of graded left $\Ug$-modules$$\Ug(0)\otimes^{\gr}_{\Bbbk}\Symgr_{\Bbbk}{\lp\oblv_{\Lie}\gfrak[-1](1)\rp}\longrightarrow\operatorname{U}^{\gr}{\lp\oblv_{\varepsilon}\Cnmixgr(\gfrak)\rp}.$$This map is an equivalence: in order to prove this, we can check at the level of each weight$$\Ug\otimes_{\Bbbk}\Sym^p_{\Bbbk}{\lp\oblv_{\Lie}\gfrak[-1]\rp}\longrightarrow\operatorname{U}^p{\lp\oblv_{\varepsilon}\Cnmixgr(\gfrak)\rp}.$$The \infinity-functor which selects the weight $p$ part is obtained by first forgetting the associative algebra structure; in particular, by \cite[Corollary $6.1.7$]{studyindag2}, we have that\begin{align*}
	\operatorname{U}^p{\lp\oblv_{\varepsilon}\Cnmixgr(\gfrak)\rp}&\simeq \lp\oblv_{\Alg}\lp\operatorname{U}^{\gr}{\lp\oblv_{\varepsilon}\Cnmixgr(\gfrak)\rp}\rp\rp_p\\&\simeq \lp\Symgr_{\Bbbk}{\lp\oblv_{\Lie}\lp\oblv_{\varepsilon}\Cnmixgr(\gfrak)\rp\rp}\rp_p\\&\simeq \lp\Symgr_{\Bbbk}{\lp\oblv_{\Lie}\gfrak(0)\oplus\oblv_{\Lie}\gfrak[-1](1)\rp}\rp_p\\&\simeq \Sym_{\Bbbk}{\lp\gfrak\rp}(0)\otimes_{\Bbbk}\Sym^p_{\Bbbk}{\lp\oblv_{\Lie}\gfrak[-1]\rp}.
	\end{align*}
	In particular, the map \ref{map:ucng} is given by tensoring the equivalence of cocommutative coalgebras $\Ug\simeq\Sym_{\Bbbk}(\oblv_{\Lie}\gfrak)$ with the identity on the $p$-th symmetric power of $\oblv_{\Lie}\gfrak[-1].$ Unraveling all definitions, one can see that working with explicit models given by associative algebras in mixed graded complexes, the mixed differential of ${\operatorname{U}^{\mixgr}}{\lp\Cnmixgr(\gfrak)\rp}$ is given by\begin{align*}
	u\otimes \lp x_1,\ldots,x_p\rp\mapsto&\sum_{1\leqslant i\leqslant p}(-1)^{\left|x_i\right|\lp \left|x_{i+1}\right|+\ldots+\left|x_p\right|\rp}{\lp x_i\cdot u\rp}\otimes \lp x_1,\ldots,\hat{x_i},\ldots,x_p\rp+\\&\sum_{1\leqslant i<j\leqslant p}(-1)^{\left|x_i\right|\lp \left|x_{i+1}\right|+\ldots+\left|x_{j-1}\right|\rp}u\otimes \lp x_1,\ldots,\hat{x_i},\ldots,x_{j-1},[x_i,x_j],x_{j+1},\ldots,x_p\rp.
	\end{align*}
	The above description of the mixed differential employs the following computations and remarks.
	\begin{enumerate}
		\item The mixed graded universal enveloping algebra of a mixed graded Lie algebra $\gfrak_{\bullet}$ is explicitly computed by the (mixed graded) tensor algebra over $\gfrak_{\bullet}$ modded out by the usual relations between the multiplication and the Lie bracket.
		\item While the braiding on the category of chain complexes involves a change of signs, the braiding on the category of mixed graded complexes does not. In particular, the signs involved in the description of the mixed graded structure depend only on the homological (internal) degree of the elements $x_1,\ldots,x_p$ in $\gfrak$.
		\item Finally, the mixed differential on $\Umixgr(\Cnmixgr(\gfrak))$ is induced by the mixed differential on $\Cnmixgr(\gfrak)$, which is the identity. By the usual computations that classically show how the universal enveloping algebra on the cone of a differential graded Lie algebra $\gfrak$ (with its natural Lie bracket) is isomorphic as a differential graded $\Ug$-module to the Chevalley-Eilenberg complex $\Ug\otimes_{\Bbbk}\Sym_{\Bbbk}{\lp\gfrak[1]\rp}$, it follows that the  mixed differential on $\UCngmixgr$ agrees with the component of the differential on $\Ug\otimes_{\Bbbk}\Sym_{\Bbbk}{\lp\gfrak[1]\rp}$ which lowers the $\Sym$-degree.
	\end{enumerate}
\end{remark}
\begin{exmp}
\label{example:differential}
	Let $\gfrak$ be the essentially unique non abelian Lie algebra on $\Bbbk\oplus\Bbbk$ such that $[e_1,e_2]=e_1$. An explicit model for $\Cnmixgr(\gfrak)$ is the mixed graded Lie algebra consisting of $\gfrak$ in weight $0$, $\gfrak[-1]$ in weight $1$, with the canonical identity $\gfrak[-1]\cong\gfrak[-1]$ as mixed differential. Denoting by $\overline{e}_i$ the generator $e_i$ of $\gfrak[-1]$ sitting in weight $1$, the Lie bracket of $\Cnmixgr(\gfrak)$ is described by the formulas\begin{align*}
	[e_i,e_j]&=[e_i,e_j]\\
	[e_i,\overline{e}_j]&=\overline{[e_i,e_j]}\\
	[\overline{e}_i,\overline{e}_j]&=0.
	\end{align*}
	The mixed graded universal enveloping algebra on $\Cnmixgr(\gfrak)$ is then the mixed graded tensor algebra on $\Cnmixgr(\gfrak)$, modded out by the relations $x\otimes y-y\otimes x =[x,y].$ This implies that
	\begin{align*}
	\overline{e}_i\otimes\overline{e}_i&=0 && \text{in weight }2,\\
	\overline{e}_1\otimes\overline{e}_2&=-\overline{e}_2\otimes\overline{e}_1&&\text{in weight } 2,\\
	e_1\otimes\overline{e}_2-\overline{e}_2\otimes e_1&=\overline{e}_1&&\text{in weight }1,\\
	e_2\otimes\overline{e}_1-\overline{e}_1\otimes e_2&=-\overline{e}_1&&\text{in weight }1
	\end{align*}
	together with the usual relations for the universal enveloping algebra on $\gfrak$ in weight $0$. Now, consider the summand $$\Bbbk(0)\oplus\Cnmixgr(\gfrak)\oplus\lp\Cnmixgr(\gfrak)\otimesmixgr_{\Bbbk}\Cnmixgr(\gfrak)\rp$$of the mixed graded tensor algebra on $\Cnmixgr(\gfrak)$. This is a mixed graded complex concentrated in weights $[0,2]$, and its weight $2$ component is isomorphic to $$\lp\gfrak\otimes_{\Bbbk}\gfrak\rp[-2] \cong \bigoplus_{1\leqslant i,j\leqslant 2}  \langle  \overline{e}_i\otimes\overline{e}_j\rangle\Bbbk[-2].$$The mixed differential between weights $2$ and $1$ is then described as $$\langle\operatorname{id},-\operatorname{id}\rangle\colon\gfrak\otimes\gfrak[-2]\longrightarrow\lp\gfrak\otimes\gfrak[-1]\oplus\gfrak[-1]\otimes\gfrak\rp[-1].$$However, by the relations written above, it follows that the weight $2$ of this mixed graded complex is equivalent to the $\Bbbk$-module $\langle\overline{e}_1\otimes\overline{e}_2\rangle\Bbbk[-2]$, and under this identification the mixed differential above is defined by the assignation$$\overline{e}_1\otimes\overline{e}_2\mapsto \lp\overline{e}_1\otimes e_2, -e_2\otimes\overline{e}_1\rp.$$Again, the previous relations imply that $-e_2\otimes\overline{e}_1=-\overline{e}_1\otimes e_2+\overline{e}_1$, and thus the mixed differential is uniquely defined by the assignation $$\overline{e}_1\otimes\overline{e}_2\mapsto\overline{e}_1\otimes e_2-\overline{e}_1\otimes e_2+\overline{e}_1=\overline{e}_1.$$
	But now $\overline{e}_1$ is just $\overline{[e_1,e_2]}$, i.e., the Lie bracket $[e_1,e_2]$ sitting in weight $1$. On the other hand, the mixed differential between weights $1$ and $0$ is just the codiagonal$$\nabla\colon\lp\gfrak[-1]\otimes_{\Bbbk}\gfrak\rp\oplus\lp\gfrak\otimes_{\Bbbk}\gfrak[-1]\rp\longrightarrow\gfrak[-1].$$By the definition of the product on the universal enveloping algebra, this amounts to sending an element $u\otimes g$ to $u\cdot g$.\\
	Now, considering \textit{all} the summands of the universal enveloping algebra $\UCngmixgr$, we obtain $\Ug$ sitting in weight $0$, $\Ug\otimes_{\Bbbk}\gfrak[-1]$ sitting in weight $1$, and $\langle \overline{e}_1,\overline{e}_2\rangle\Ug[-2]$ sitting in weight $2$. The mixed differential of such mixed graded associative algebra is induced by tensoring the mixed differential of $\Cnmixgr(\gfrak)$ with itself a suitable number of times: in particular, taking carefully into account the signs and considering all the identifications implied by the relations for the universal enveloping algebra, we obtain that an element $u\otimes\overline{e}_1\wedge\overline{e}_2$ in weight $2$ is sent to $u\cdot \overline{e}_1\otimes \overline{e}_2-u\cdot\overline{e}_2\otimes\overline{e}_1-u\cdot\overline{[e_1,e_2]}$.
\end{exmp}
Applying the mixed graded version of the universal enveloping $\Bbbk$-algebra $\infinity$-functor \ref{functor:mixedgradedug}, we land in ${\Fun}{\lp\Delta^1,\hsp\Algmixgr_{\Bbbk}\rp}$ with an $\infinity$-functor\begin{align}\label{functor:composition3}
\Lie_{\Bbbk}\longrightarrow{\Fun}{\lp\Delta^1,\hsp\Algmixgr_{\Bbbk}\rp}.\end{align}
\begin{parag}
	Let us focus now on the $\infinity$-category ${\Fun}{\lp\Delta^1,\hsp\Alg_{\Ebb_1}{\lp\scrC\rp}\rp}\eqqcolon{\Fun}{\lp\Delta^1,\hsp\Alg(\scrC)\rp}$, where $\scrC$ is a stable symmetric monoidal $\infinity$-category which is $\Bbbk$-linear over a field of characteristic $0$ (in our case, $\scrC$ will be $\Modmixgr_{\Bbbk}$). By \cite[Corollary $4.2.3.2$]{ha}, we have a Cartesian fibration$$\theta\colon\LMod{\lp\scrC\rp}\longrightarrow\Alg(\scrC)$$which classifies the $\infinity$-functor$$\LMod_{(-)}(\scrC)\colon\lp\Alg(\scrC)\rp^{\op}\longrightarrow\Catinfty$$informally described by the assignation $\left\{f\colon A\to B\right\}$ to $\left\{f_*\colon\LMod_B(\scrC)\to\LMod_A(\scrC)\right\}$, where $f_*$ forgets the $B$-module structure to an $A$-module structure along $f$. 
	In particular, in view of \cite[Remark $2.4.2.9$]{htt}, we can promote such assignation to an $\infinity$-functor 
	\begin{align*}
	{\Fun}{\lp\Delta^1,\hsp\Alg(\scrC)\rp}\longrightarrow{\Fun}{\lp\Delta^1,\hsp\Catinfty^{\op}\rp}.
	\end{align*}
	Now, let $$s\colon\Alg(\scrC)\longrightarrow\LMod(\scrC)$$be the section of the Cartesian fibration sending an algebra object $A$ of $\scrC$ to the object $(A,\hsp A)$ of $\LMod(\scrC)$ where $A$ is seen as an $A$-module object (\cite[Example $4.2.1.17$]{ha}). Given an arrow $f\colon A\to B$, seen as an object of the \infinity-category ${\Fun}{\lp\Delta^1,\hsp\Alg(\scrC)\rp}$, we have a locally $\theta$-Cartesian morphism with codomain $s(B)$ which lifts $f$ (\cite[Proposition $2.4.2.8$]{htt}). Let $\overline{f}$ be such $\theta$-Cartesian lift. Then, the domain of $\overline{f}$ is described by the object $(A,\hsp B)$, where $B$ is seen as an $A$-module along $f$. Considering the codomain of an arrow $f\colon A \to B$ and taking its image under the section $s$, we obtain an \infinity-functor$${\Fun}{\lp\Delta^1,\hsp \Alg(\scrC)\rp}\longrightarrow\LMod(\scrC).$$
	Again, this can be promoted to an \infinity-functor\begin{align}
	\label{functor:algebramodule}{\Fun}{\lp\Delta^1,\hsp\Alg(\scrC)\rp}\longrightarrow{\Fun}{\lp\Delta^1,\hsp\LMod(\scrC)\rp}.\end{align}Indeed, for any square of arrows$$\begin{tikzpicture}[scale=0.75]
	\node (a) at (0,2){$A$};\node(b) at (2,2){$B$};
	\node (c) at (0,0){$C$};\node (d) at (2,0){$D$};
	\draw[->,font=\scriptsize] (a) to node[above]{$f$}(b);
	\draw[->,font=\scriptsize] (a) to node[left]{$\phi$}(c);
	\draw[->,font=\scriptsize] (b) to node[right]{$\psi$}(d);
	\draw[->,font=\scriptsize] (c) to node[below]{$g$}(d);
	\end{tikzpicture}$$
	the Cartesian lift $\overline{f}\colon (A,\hsp B)\to (B,\hsp B)$ produces an arrow in $\LMod(\scrC)$ given by $(A,\hsp B) \to (D,\hsp D)$. Considering now a $\theta$-cartesian lift $\overline{g}$ of $g$,  by the equivalence$$\LMod(\scrC)_{/\overline{g}}\simeq\LMod(\scrC)_{/s(D)}\times_{\Alg(\scrC)_{/D}}\Alg(\scrC)_{/f}$$(which is the definition of $\theta$-Cartesian morphism, \cite[Definition $2.4.1.1$]{htt}) we obtain an essentially unique arrow $(A,\hsp B)\to (C,\hsp D)$, where $(C,\hsp D)$ is the domain of the Cartesian lift $\overline{g}$. This arrow has to be understood as the map of associative algebras $\phi \colon A\to C$, and the map $\psi\colon B\to D$ that, being compatible with the associative algebra structure, is also $A$-linear with respect to the $A$-module structure of $B$ and $D$ inherited by the maps $f\colon A\to B$ and $\psi\circ f\colon A\to D$, respectively. The functoriality of this construction is then a consequence of \cite[Remark $2.4.1.4$]{htt}.
\end{parag}
\begin{parag}
	In the case $\scrC=\Modmixgr_{\Bbbk}$, composing \ref{functor:composition3} with the $\infinity$-functor \ref{functor:algebramodule}, we obtain the \infinity-functor\begin{align}
	\label{functor:composition4}
	\Lie_{\Bbbk}\longrightarrow\LMod{\lp\Modmixgr_{\Bbbk}\rp}
	\end{align}
	which can be interpreted as the $\infinity$-functor sending a Lie algebra $\gfrak$ to the object ${\Umixgr}{\lp\Cnmixgr(\gfrak)\rp}$, seen as a left $\Ug(0)$-module. 
\end{parag}
\begin{parag}
	\label{parag:relativetensorproduct}
	Let us consider the $\infinity$-category $\LMod{\lp\Modmixgr_{\Bbbk}\rp}$. Every mixed graded $\Bbbk$-module is a $\Bbbk(0)$-bimodule, since $\Bbbk(0)$ is the unit for the symmetric monoidal structure on $\Modmixgr_{\Bbbk}$, and in particular it is both left and right $\Bbbk(0)$-module. Moreover, these $\Bbbk(0)$-module structures are clearly compatible with any left or right $A_{\bullet}$-module structure over an associative algebra $A_{\bullet}$ in $\Modmixgr_{\Bbbk}$. So, by \cite[Theorem $4.3.2.7$]{ha}, the $\infinity$-categories $$\LMod{\lp\Modmixgr_{\Bbbk}\rp}\subseteq\LMod{\lp\RMod{\lp\Modmixgr_{\Bbbk}\rp}\rp}$$and$$\RMod{\lp\Modmixgr_{\Bbbk}\rp}\subseteq\RMod{\lp\LMod{\lp\Modmixgr_{\Bbbk}\rp}\rp}$$can be seen as a sub-$\infinity$-categories of the $\infinity$-category ${\operatorname{BMod}}{\lp\Modmixgr_{\Bbbk}\rp}$. The relative tensor product $\infinity$-functor$$T\colon{\operatorname{BMod}{\lp\Modmixgr_{\Bbbk}\rp}}\times_{\Alg{\lp\Modmixgr_{\Bbbk}\rp}}{\operatorname{BMod}{\lp\Modmixgr_{\Bbbk}\rp}}\longrightarrow{\operatorname{BMod}{\lp\Modmixgr_{\Bbbk}\rp}}$$(see \cite[Definition $4.4.2.10$]{ha}) sends an $(A_{\bullet},\hsp B_{\bullet})$-bimodule $M_{\bullet}$ and a $(B_{\bullet},\hsp C_{\bullet})$-bimodule $N_{\bullet}$ to the tensor product $M_{\bullet}\otimesmixgr_{B_{\bullet}}N_{\bullet}$, which is now an $(A_{\bullet},\hsp C_{\bullet})$-bimodule.\\
	
	With all these notations, we can consider the relative tensor product $\infinity$-functor $T$ restricted to the sub-$\infinity$-category$$\RMod\lp\Modmixgr_{\Bbbk}\rp\times_{\Alg{\lp\Modmixgr_{\Bbbk}\rp}}\LMod\lp\Modmixgr_{\Bbbk}\rp$$of$$ {\operatorname{BMod}{\lp\Modmixgr_{\Bbbk}\rp}}\times_{\Alg{\lp\Modmixgr_{\Bbbk}\rp}}{\operatorname{BMod}{\lp\Modmixgr_{\Bbbk}\rp}}$$spanned by those couples $\lp M_{\bullet},\hsp N_{\bullet}\rp$ comprising of right and left module objects $M_{\bullet}$ and $N_{\bullet}$ in $\Modmixgr_{\Bbbk}$ such that $M_{\bullet}$ is a right module on the very same algebra for which $N_{\bullet}$ is a left module. In this case, the relative tensor product $\infinity$-functor lands naturally in the \infinity-category of mixed graded $\Bbbk$-modules $\Modmixgr_{\Bbbk}$.
\end{parag}
\begin{parag}
	\label{parag:kasugmodule}
	By the discussion provided in \ref{parag:relativetensorproduct}, we have a family of $\infinity$-functors parametrized by Lie algebras in the following way:\begin{align}
	\label{functor:CEhomological}
	\{\gfrak\}\longhookrightarrow\Lie_{\Bbbk}\xrightarrow{\ref{functor:composition4}}\LModmixgr_{\Ug(0)}\overset{T}{\longrightarrow}{\Fun}{\lp\RModmixgr_{\Ug(0)},\hsp\Modmixgr_{\Bbbk}\rp}.
	\end{align}
	Here, by $\LModmixgr_{\Ug(0)}$ and $\RModmixgr_{\Ug(0)}$ we mean, respectively, the $\infinity$-categories $$\LMod_{\Ug(0)}{\lp\Modmixgr_{\Bbbk}\rp}$$and $$\RMod_{\Ug(0)}{\lp\Modmixgr_{\Bbbk}\rp}.$$ Let us remark that, working with explicit models given by mixed graded complexes of $\Bbbk$-modules, these $\infinity$-categories can be described as the $\infinity$-categories whose objects are countable collections of chain complexes of left (resp. right) $\Ug$-modules such that the mixed differential is a morphism of left (resp. right) $\Ug$-modules, and whose $1$-morphisms are countably many morphisms of left (resp. right) $\Ug$-modules commuting with the mixed differentials.\\
	Informally, the association \ref{functor:CEhomological} is defined by sending $\gfrak$ to the $\infinity$-functor $$\CE_{\varepsilon}(\gfrak;\hsp-)\coloneqq-\otimesmixgr_{\Ug(0)}\Umixgr{\lp\Cnmixgr(\gfrak)\rp}.$$ 
\end{parag}
We cannot in general promote the association \ref{functor:CEhomological} to an $\infinity$-functor $$\Lie_{\Bbbk}\longrightarrow{\Fun}{\lp\RMod\lp\Modmixgr_{\Bbbk}\rp,\hsp\Modmixgr_{\Bbbk}\rp}$$because, for any $\gfrak$ in $\Lie_{\Bbbk}$, we cannot define the action of an $\infinity$-functor over those right module objects in $\Modmixgr_{\Bbbk}$ which are right modules over algebras different from $\Ug(0)$. 
However, we can observe that $\Bbbk(0)$ is naturally a (both left and right) $\Ug(0)$-module for any $\gfrak$. This follows from the fact that $\Ug$ is an augmented $\Bbbk$-algebra, and so applying the $\infinity$-functor $(-)(0)\colon\Alg_{\Bbbk}\hookrightarrow\Algmixgr_{\Bbbk}$, $\Ug(0)$ becomes naturally an augmented $\Bbbk(0)$-algebra, and this provides $\Bbbk(0)$ with a (trivial) left and right $\Ug(0)$-module structure for any Lie algebra $\gfrak$. In particular, we have an $\infinity$-functor\begin{align}
\label{functor:CEhomologicalcoalgebra}
\Lie_{\Bbbk}\xrightarrow{\ref{functor:composition4}}\LMod\lp\Modmixgr_{\Bbbk}\rp\subseteq\{\Bbbk(0)\}\times{\LMod\lp\Modmixgr_{\Bbbk}\rp}\overset{T}\longrightarrow\Modmixgr_{\Bbbk}.
\end{align}
\begin{defn}
	\label{def:homologicalCE}
	The $\infinity$-functor \ref{functor:CEhomologicalcoalgebra} is the \textit{(mixed graded) homological Chevalley-Eilenberg $\infinity$-functor}, and shall be denoted by $\CE_{\varepsilon}$.
\end{defn}
\begin{remark}
	\label{remark:explicitgradedpiecesCEcoalg}
	The mixed graded homological Chevalley-Eilenberg $\infinity$-functor $\CE_{\varepsilon}$ sends a Lie algebra $\gfrak$ to a mixed graded $\Bbbk$-module given by the mixed graded tensor product $$\CE_{\varepsilon}(\gfrak)\coloneqq\Bbbk(0)\otimesmixgr_{\Ug(0)}{\UCngmixgr}.$$Since forgetting the mixed structure is a strongly monoidal $\infinity$-functor, the underlying graded $\Bbbk$-module is equivalent, in virtue of the discussion of \cref{remark:CEdifferential}, to\begin{align*}
	\oblv_{\varepsilon}{\lp\CE_{\varepsilon}(\gfrak)\rp}&\simeq\oblv_{\varepsilon}{\lp \Bbbk(0)\otimesmixgr_{\Ug(0)}{\UCngmixgr}\rp}\\
	&\simeq\oblv_{\varepsilon}{\lp \Bbbk(0)\rp}\otimes^{\gr}_{\Ug(0)}\oblv_{\varepsilon}{\lp{\UCngmixgr}\rp}\\
	&\simeq\Bbbk(0)\otimes^{\gr}_{\Ug(0)}\Ug(0)\otimes^{\gr}_{\Bbbk(0)}{\Symgr_{\Bbbk}}{\lp\gfrak[-1](1)\rp}\simeq\Sym^{\gr}_{\Bbbk}{\lp \oblv_{\Lie}{\lp\gfrak\rp}[-1](1)\rp}.
	\end{align*} In other words, for all integers $p$ we have a natural equivalence of $\Bbbk$-modules$$\CE_{\varepsilon}(\gfrak)_p\simeq\Sym^p_{\Bbbk}{\lp\oblv_{\Lie}(\gfrak)[-1]\rp}.$$Moreover, working with explicit models given by mixed graded complexes, the mixed differential of $\CE_{\varepsilon}(\gfrak)$ is given by tensoring the mixed differential of $\UCngmixgr$ exhibited in \cref{remark:CEdifferential} with the trivial differential of $\Bbbk(0)$ which kills the part of the mixed differential which depends on the left $\Ug$-module structure. Hence, it follows that the Tate realization of $\CE_{\varepsilon}(\gfrak)$ agrees with the usual Chevalley-Eilenberg complex $\CE_{\bullet}(\gfrak)$, again by directly inspecting the usual explicit model for $\left|\CE_{\varepsilon}(\gfrak)\right|^{\operatorname{t}}$.
\end{remark}
\begin{propositionn}
	\label{prop:promotionCEcocommutativealgebra}
	The mixed graded homological Chevalley-Eilenberg $\infinity$-functor can be promoted to an $\infinity$-functor$$\CE_{\varepsilon}\colon\Lie_{\Bbbk}\longrightarrow\CoCommmixgraug$$where $\CoCommmixgraug$ is the $\infinity$-category of coaugmented cocommutative coalgebra objects in $\Modmixgr_{\Bbbk}$.
\end{propositionn}
\begin{proof}
	First of all, let us remark that we have an equivalence of $\infinity$-categories$$\Lie_{\Bbbk}\simeq {\operatorname{cCcAlg}}{\lp\Lie_{\Bbbk}\rp}$$in virtue of \cite[Proposition $2.4.3.9$]{ha}: essentially, the diagonal morphism $\Delta\colon\gfrak\to\gfrak\times\gfrak$ turns every Lie algebra into a cocommutative coalgebra object in Lie algebras. So, our claim will follow by proving that $\CE_{\varepsilon}$ is a (actually, \textit{strongly}) monoidal $\infinity$-functor with respect to the Cartesian monoidal structure on $\Lie_{\Bbbk}$ and the monoidal structure on $\Modmixgr_{\Bbbk}$ described in \ref{parag:monoidalstructureonmodmixgr}.\\
	We have a natural map $$\CE_{\varepsilon}(\gfrak\times\hfrak)\longrightarrow\CE_{\varepsilon}(\gfrak)\times\CE_{\varepsilon}(\hfrak)\longrightarrow\CE_{\varepsilon}\otimesmixgr_{\Bbbk}\CE_{\varepsilon}(\hfrak)$$and, since forgetting the mixed structure is a strongly monoidal operation, we can safely check that this is an equivalence by inspecting the map on the underlying graded $\Bbbk$-modules$$\oblv_{\varepsilon}\lp\CE_{\varepsilon}(\gfrak\times\hfrak)\rp\longrightarrow\oblv_{\varepsilon}\lp\CE_{\varepsilon}(\gfrak)\rp\otimesmixgr_{\Bbbk}\oblv_{\varepsilon}\lp\CE_{\varepsilon}(\hfrak)\rp.$$In virtue of \cref{remark:explicitgradedpiecesCEcoalg}, we have an equivalence$$\oblv_{\varepsilon}{\lp\CE_{\varepsilon}{\lp\gfrak\times\hfrak\rp}\rp}\simeq{{\Symgr_{\Bbbk}}{\lp\oblv_{\Lie}\lp\gfrak\times\hfrak\rp[-1](1)\rp}},$$where $\Symgr_{\Bbbk}{\lp\oblv_{\Lie}(\gfrak\times\hfrak)[-1](1)\rp}$ is the free ind-nilpotent graded cocommutative coalgebra over $\gfrak\times\hfrak$ (which agrees on the usual symmetric coalgebra with its natural grading). On the other hand, the free ind-nilpotent cocommutative coalgebra \infinity-functor is always a right adjoint, and the same holds for the inclusion of ind-nilpotent coalgebras into all coalgebras (\cref{warning:coalgebras}). Moreover, products in cocommutative coalgebras are described by tensor products (\cite[Chapter $6$, Section $4.1.1$]{studyindag2} and \cite[Section $3.3$]{ellipticcohomology1}), therefore we have another equivalence$$\oblv_{\varepsilon}\lp\CE_{\varepsilon}(\gfrak)\rp\otimesmixgr_{\Bbbk}\oblv_{\varepsilon}\lp\CE_{\varepsilon}(\hfrak)\rp\simeq \Symgr_{\Bbbk}{\lp\oblv_{\Lie}(\gfrak)[-1](1)\rp}\otimesmixgr_{\Bbbk}\Symgr_{\Bbbk}{\lp\oblv_{\Lie}(\hfrak)[-1](1)\rp}.$$It follows that $\CE_{\varepsilon}$ is a strongly monoidal $\infinity$-functor between $\Lie_{\Bbbk}^{\times}$ and $\lp\Modmixgr_{\Bbbk}\rp^{\otimesmixgr}$,  and so we can promote $\CE_{\varepsilon}$ to an \infinity-functor\begin{align}
	\label{funct:tocoalgebras}
	{\operatorname{cCcAlg}}{\lp\Lie_{\Bbbk}\rp}\simeq\Lie_{\Bbbk}\longrightarrow{\operatorname{cCcAlg}}{\lp\Modmixgr_{\Bbbk}\rp}\eqqcolon\CoCommmixgr.
	\end{align}The fact that every $\CE_{\varepsilon}{\lp\gfrak\rp}$ is actually a \textit{coaugmented} mixed graded cocommutative coalgebra follows from the obvious fact that $\operatorname{U}(0)\simeq \Sym_{\Bbbk}(0)\simeq\Bbbk$, and since $0$ is a zero object in $\Lie_{\Bbbk}$ one has a natural coaugmentation, as desired.
\end{proof}
\begin{notation}
	In the following, we shall denote both the \infinity-functor \ref{functor:CEhomologicalcoalgebra} and the \infinity-functor \ref{funct:tocoalgebras} as $\CE_{\varepsilon}$.
\end{notation}
\begin{porism}
	\label{porism:cnmixgrcoalgebra}
	Unraveling all the constructions, it turns out that the mixed graded cocommutative coalgebra structure of $\CE_{\varepsilon}(\gfrak)$ is induced simply by the diagonal morphism $\gfrak\to\gfrak\times\gfrak$, which depends only on the $\Bbbk$-module structure of $\gfrak$ and not on its Lie bracket. In fact, the cocommutative coalgebra structure of $\CE_{\varepsilon}(\gfrak)$ is induced by the fact that $\UCngmixgr$ is itself a mixed graded cocommutative coalgebra, since it is the universal enveloping algebra of $\Cnmixgr(\gfrak)$ (and $\Cnmixgr$ preserves products of Lie algebras, trivially). \\ In other words, the proof of \cref{prop:promotionCEcocommutativealgebra} shows that the underlying graded cocommutative coalgebra of $\CE_{\varepsilon}(\gfrak)$ is \textit{precisely} $\Symgr_{\Bbbk}(\oblv_{\Lie}(\gfrak)[-1](1))$. The Lie bracket plays a role only in setting a mixed structure of $\CE_{\varepsilon}(\gfrak)$ which agrees with such cocommutative coalgebra structure.
\end{porism}
We conclude this subsection by noticing that, working with explicit models, this construction actually yields the usual homological Chevalley-Eilenberg chain complex (or, better, its homology).
\begin{propositionn}
	\label{prop:CEexplicitmodel}
	Given a Lie algebra $\gfrak$, the Tate realization $\left|\CE_{\varepsilon}(\gfrak)\right|^{\operatorname{t}}$ agrees with the homological Chevalley-Eilenberg complex of \cite[Construction $2.2.3$]{dagx}.
\end{propositionn}
\begin{proof}
	Our mixed graded Chevalley-Eilenberg $\infinity$-functor is given by the mixed graded relative tensor product$$\CE_{\varepsilon}(\gfrak)\coloneqq\Bbbk(0)\otimesmixgr_{\Ug(0)}\UCngmixgr,$$and $\Bbbk(0)$, $\Ug(0)$ and $\UCngmixgr$ are all mixed graded objects lying in the \infinity-category $\Modmixgrcn_{\Bbbk}$ of non-negatively graded mixed $\Bbbk$-modules. The restriction of the Tate realization \infinity-functor to non-negatively graded objects  is strongly monoidal in virtue of \cite[Porism $2.3.20$]{pavia1}, and it preserves colimits, hence geometric realizations, in virtue of \cite[Porism $2.3.21$]{pavia1}. Since the relative tensor product is calculated by a two-sided Bar construction (\cite[Section $4.4.2$]{ha}), it follows that\begin{align*}\left|\CE_{\varepsilon}(\gfrak)\right|^{\operatorname{t}}&\coloneqq\left|\Bbbk(0)\otimesmixgr_{\Ug(0)}\UCngmixgr\right|^{\operatorname{t}}\\&\simeq\left|\Bbbk(0)\right|^{\operatorname{t}}\otimes_{\left|\Ug(0)\right|^{\operatorname{t}}}\left|\UCngmixgr\right|^{\operatorname{t}}\\&\simeq \Bbbk\otimes_{\Ug}\left|\UCngmixgr\right|^{\operatorname{t}}.\end{align*}We are left to prove that the Tate realization of $\UCngmixgr$ is equivalent to $\Bbbk$, since the classical Chevalley-Eilenberg homology of a Lie algebra $\gfrak$ is given by $\Bbbk\otimes_{\Ug}\Bbbk$. This is a straightforward computation: the description of the mixed differential in \ref{remark:CEdifferential} shows that the Tate realization is the usual universal enveloping algebra of the cone of $\gfrak$ (in the sense of \cite[Construction $2.2.1$]{dagx}) which is a well known cofibrant replacement for the trivial $\Ug$-module $\Bbbk$.
\end{proof}
\begin{corollaryn}
	\label{corollary:trivliealgebrasaretrivcoalgebras}
	We have a natural equivalence$$\CE_{\varepsilon}\circ\triv_{\Lie}\overset{\simeq}{\longrightarrow}\triv_{\varepsilon}\circ\Symgr_{\Bbbk}{\lp(-)[-1](1)\rp}$$ of \infinity-functors from $\Bbbk$-modules to mixed graded $\Bbbk$-modules.
\end{corollaryn}
\begin{proof}
	If $\gfrak\simeq \triv_{\Lie}M$, since the Tate realization of $\CE_{\varepsilon}(\gfrak)$ agrees with the usual Chevalley-Eilenberg homological complex, it turns out that $$\left|\CE_{\varepsilon}(\gfrak)\right|^{\operatorname{t}}\simeq \CE_{\bullet}(\gfrak)\simeq\Sym_{\Bbbk}{\lp M[1]\rp}\coloneqq\bigoplus_{q\geqslant0}\Sym^q_{\Bbbk}{\lp M[1]\rp},$$where the latter equivalence is \cite[Chapter $6$, Section $4.2.3$]{studyindag2}. We prove the following useful remark.
	\begin{lemman}
		Let $M_{\bullet}$ be a non-negatively graded mixed $\Bbbk$-module. Then $$\left|M_{\bullet}\right|^{\operatorname{t}}\simeq \bigoplus_{q\geqslant0}M_q[2q]$$if and only if $M_{\bullet}\simeq \triv_{\varepsilon}\oblv_{\varepsilon}M_{\bullet}.$
	\end{lemman}
	Indeed, by the definition of Tate realization, it is clear that if $M_{\bullet}$ is trivial one has the desired description of its Tate realization. On the converse, suppose that $\left|M_{\bullet}\right|^{\operatorname{t}}\simeq\bigoplus_{q\geqslant0}M_q[2q]$. In virtue of the adjunction $\left|-\right|^{\operatorname{t}}\dashv (-)(0)$ described in \cite[Porism $2.3.21$]{pavia1}, for any $\Bbbk$-module $N$ we have a chain of equivalences of mapping spaces\begin{align*}
	\Map_{\Modmixgr_{\Bbbk}}{\lp M_{\bullet},\hsp N(0)\rp}&\simeq\Map_{\Mod_{\Bbbk}}{\lp\left|M_{\bullet}\right|^{\operatorname{t}},\hsp N\rp}\\&\simeq\Map_{\Mod_{\Bbbk}}{\lp\bigoplus_{q\geqslant0}M_q[2q],\hsp N\rp}\\&\simeq\prod_{q\geqslant 0}\Map_{\Mod_{\Bbbk}}{\lp M_q,\hsp N[-2q]\rp}.
	\end{align*}
	On the other hand, by the adjunction $\triv_{\varepsilon}\dashv \operatorname{NC}^{\operatorname{w}}$ of \cite[Remark $1.5$]{PTVV}, and recalling that $$\operatorname{NC}^{\operatorname{w}} {\lp N_{\bullet}\rp}_p\simeq\prod_{q\leqslant p} N_{q}[-2(p-q)],$$we have a chain of equivalences of mapping spaces\begin{align*}
	\Map_{\Modmixgr_{\Bbbk}}{\lp \triv_{\varepsilon}\oblv_{\varepsilon}M_{\bullet},\hsp N(0)\rp}&\simeq \Map_{\Modgr_{\Bbbk}}{\lp \oblv_{\varepsilon}M_{\bullet},\hsp \operatorname{NC}^{\operatorname{w}}(N)\rp}\\
	&\simeq \prod_{q\in\ZZ}\Map_{\Mod_{\Bbbk}}{\lp M_q,\hsp \operatorname{NC}^{\operatorname{w}}(N)_q\rp}\\
	&\simeq \prod_{q\geqslant 0 }\Map_{\Mod_{\Bbbk}}{\lp M_q,\hsp N[-2q]\rp}.
	\end{align*}
	Therefore the fully faithfulness of Yoneda embedding yields that $M_{\bullet}\simeq \triv_{\varepsilon}\oblv_{\varepsilon}M_{\bullet}$, i.e., that $M_{\bullet}$ has trivial mixed structure. 
\end{proof}
\begin{remark}
	\label{remark:trivliealgebrasaretrivcoalgebras}
	We believe that the statement of \cref{corollary:trivliealgebrasaretrivcoalgebras} can be refined in the following way: this is an equivalence of \infinity-functors with values in mixed graded \textit{cocommutative coalgebras}. We provide a sketch of the proof: since $\triv_{\varepsilon}$ is strongly monoidal and commutes also with limits, it becomes a strongly monoidal right adjoint also at the level of cocommutative coalgebras; hence, it admits a left adjoint which, after forgetting the cocommutative coalgebra structure, agrees with a left adjoint $\operatorname{triv}_{\varepsilon}^{\operatorname{L}}$ of $\operatorname{triv}_{\varepsilon}.$ One can see, using \cite[Porism $2.3.21$]{pavia1} and the naive truncation \infinity-functor of \cite[Definition $1.2.5$]{pavia1}, that a model for such left adjoint is given by the \infinity-functor sending a non-negatively mixed graded $\Bbbk$-module $M_{\bullet}$ to the non-negatively graded $\Bbbk$-module described in weight $p\geqslant0$ by$$\triv_{\varepsilon}^{\operatorname{L}}M_p\simeq \left|\sigma_{\geqslant 0}M_{\bullet}\lpd -p\rpd\right|^{\operatorname{t}}.$$
	In particular, since $\Symgr_{\Bbbk}$ is the ind-nilpotent cofree cocommutative coalgebra \infinity-functor, we have that providing a map of mixed graded cocommutative coalgebras $\CE_{\varepsilon}{\lp \triv_{\Lie}M\rp}\to\triv_{\varepsilon}\Symgr_{\Bbbk}{\lp M[-1](1)\rp}$ is equivalent to providing a map of $\Bbbk$-modules $$\operatorname{triv}_{\varepsilon}^{\operatorname{L}}{\lp \CE_{\varepsilon}{\lp \triv_{\Lie}M\rp}\rp}_1[1]\simeq \left|\sigma_{\geqslant 0}{\lp\CE_{\varepsilon}{\lp\triv_{\Lie}M_{\bullet}\rp}\lpd 1\rpd\rp}\right|^{\operatorname{t}}[1]\longrightarrow M.$$Since $\gfrak\simeq \triv_{\Lie}M$ has a trivial Lie bracket, such totalization contains $M$ as a direct summand. The map corresponding, by adjunction, to the projection on $M$ yields the map which after forgetting the cocommutative coalgebra structure realizes the equivalence of \cref{corollary:trivliealgebrasaretrivcoalgebras}.
\end{remark}
\begin{remark}
The associated filtered \infinity-functor $(-)^{\fil}\colon\Modmixgr_{\Bbbk}\to\Modfil_{\Bbbk}$ (see \cite{derivedfoliations2,calaque2021lie,pavia1}), being the right adjoint to the strongly monoidal \infinity-functor $(-)^{\mixgr}\colon\Modfil_{\Bbbk}\to\Modmixgr_{\Bbbk}$, is only lax monoidal in general. But it is strongly monoidal if one restricts it to the full sub-\infinity-category $\Modmixgrcn_{\Bbbk}$ of \textit{non-negatively} graded mixed $\Bbbk$-modules, as proved in \cite[Porism $2.3.20$]{pavia1}. In particular, $\lp\CE_{\varepsilon}(\gfrak)\rp^{\fil}$ is again a cocommutative coalgebra object in $\Modfil_{\Bbbk}$, whose underlying graded corresponds to$$\left\{\Gr_n\lp\CE_{\varepsilon}(\gfrak)\rp^{\fil}\right\}_{n\in\ZZ}\simeq \left\{\Sym^{-n}_{\Bbbk}{\lp\gfrak[-1]\rp}[-2n]\right\}_{n\in\ZZ}$$and whose underlying object $\lp\CE_{\varepsilon}(\gfrak)\rp^{\fil}_{-\infty}$ is the homological Chevalley-Eilenberg coalgebra $\CE_{\bullet}(\gfrak)$. This recovers the usual filtration on the homological Chevalley-Eilenberg coalgebra by increasing symmetric powers.
\end{remark}
\subsection{The cohomological Chevalley-Eilenberg $\infty$-functor}
\label{sec:cohomCE}
In this subsection, we want to construct a \textit{cohomological} variant of $\infinity$-functor \ref{functor:CEhomologicalcoalgebra}, i.e., an $\infinity$-functor$$\CE^{\varepsilon}\colon\Lie_{\Bbbk}^{\op}\longrightarrow\CAlgmixgr_{\Bbbk}.$$Luckily, for the most part the construction of our cohomological $\infinity$-functor relies on the technical machinery developed in order to define the \infinity-functor \ref{functor:CEhomologicalcoalgebra}, so we only need to apply some minor modifications to our previous construction.
\begin{parag}
	\label{parag:LModUGenriched}
	Let us recall the $\infinity$-functor \ref{functor:composition4}, sending a Lie algebra $\gfrak$ to $\UCngmixgr$ seen as a $\Ug(0)$-module in mixed graded $\Bbbk$-modules. We have the obvious "opposite" $\infinity$-functor\begin{align}
	\label{functor:composition4op}
	(\ref{functor:composition4})^{\op}\colon\Lie_{\Bbbk}^{\op}\longrightarrow\LMod{\lp\Modmixgr_{\Bbbk}\rp}^{\!\op}
	\end{align}which we can further compose with the inclusion $$\LMod{\lp\Modmixgr_{\Bbbk}\rp}^{\!\op}\subseteq\LMod{\lp\Modmixgr_{\Bbbk}\rp}^{\!\op}\times_{\Alg{\lp\Modmixgr_{\Bbbk}\rp}}\RMod{\lp\Modmixgr_{\Bbbk}\rp}.$$We can now observe that the $\infinity$-category $\LModmixgr_{\Ug(0)}$ is enriched over $\Modmixgr_{\Bbbk}$  (this is a mixed graded analogue of \cite[Variant $7.2.1.24$]{ha}). This can be shown in the following way: let $T$ denote the relative tensor product described in \ref{parag:relativetensorproduct}. Then, since $\LModmixgr_{\Ug(0)}$ is equivalent to the \infinity-category of $\lp\Ug(0),\hsp\Bbbk(0)\rp$-bimodules, the relative tensor product$$T\colon\LModmixgr_{\Ug(0)}\times\Modmixgr_{\Bbbk}\longrightarrow\LModmixgr_{\Ug(0)}$$provides a (right) tensor structure of $\LModmixgr_{\Ug(0)}$ over $\Modmixgr_{\Bbbk}$, in the sense of \cite[Definition $4.2.1.19$]{ha}. Such relative tensor product commutes with all colimits separably in each variable, since the forgetful \infinity-functor$$\oblv_{\Ug}\colon\LModmixgr_{\Ug(0)}\longrightarrow\Modmixgr_{\Bbbk}$$preserves all limits and colimits in virtue of \cite[Corollary $4.2.3.7$]{ha}, and the tensor product of mixed graded $\Bbbk$-modules commutes with all colimits separably in each variable as well. So, we can apply \cite[Proposition $4.2.1.33.(2)$]{ha} to get that $\LModmixgr_{\Ug(0)}$ is enriched over $\Modmixgr_{\Bbbk}$. In particular, we can consider the morphism object $\infinity$-functor $$\operatorname{Mor}^{\mixgr}_{\Ug(0)}\colon\lp\LModmixgr_{\Ug(0)}\rp^{\op}\times\LModmixgr_{\Ug(0)}\longrightarrow\Modmixgr_{\Bbbk}.$$Moreover, using \cite[Remark $4.2.1.31$]{ha}, we can promote the morphism object $\infinity$-functor to a functorial assignation\begin{align}
	\label{functor:morobj}
	\Mormixgr\colon\lp\LMod{\lp\Modmixgr_{\Bbbk}\rp}\rp^{\!\op}\times_{\Alg{\lp\Modmixgr_{\Bbbk}\rp}}\LMod{\lp\Modmixgr_{\Bbbk}\rp}\longrightarrow\Modmixgr_{\Bbbk}.
	\end{align}
\end{parag}
\begin{parag}
	Thanks to this formalism, we get a "cohomological" (contravariant) analogue of the $\infinity$-functor \ref{functor:CEhomological}\begin{align}
	\label{functor:CEcohomological}
	\{\gfrak\}\longhookrightarrow\Lie_{\Bbbk}^{\op}\xrightarrow{\ref{functor:composition4op}}\lp\LModmixgr_{\Ug(0)}\rp^{\!\op}\xrightarrow{\ref{functor:morobj}}{\Fun}{\lp\LModmixgr_{\Ug(0)},\hsp\Modmixgr_{\Bbbk}\rp}.
	\end{align}
	Informally, the assignation \ref{functor:CEcohomological} is defined by sending $\gfrak$ to the $\infinity$-functor $$\CE^{\varepsilon}(\gfrak;\hsp-)\coloneqq\Mormixgr_{\Ug(0)}{\lp\UCngmixgr,\hsp -\rp}.$$The above mixed graded $\Bbbk$-module must be understood as the mixed graded sub-$\Bbbk$-module of $\Mapmixgr_{\Bbbk}{\lp\UCngmixgr,\hsp-\rp}$ comprising of those morphisms which are also left $\Ug(0)$-linear, with mixed differential given by restricting the mixed differential described in \ref{parag:monoidalstructureonmodmixgr}.\\
	As in the homological case, this cannot be promoted to a contravariant $\infinity$-functor from Lie algebras to presheaves over ${\LMod}{\lp\Modmixgr_{\Bbbk}\rp}$ with values in mixed graded $\Bbbk$-modules; but since $\Bbbk(0)$ is obviously a left (trivial) $\Ug(0)$-module for any Lie algebra $\gfrak$ (as already observed in \ref{parag:kasugmodule}), we do have the contravariant $\infinity$-functor\begin{align}
	\label{functor:CEcohomologicalalgebra}
	\Lie_{\Bbbk}^{\op}\xrightarrow{\ref{functor:composition4op}}{\LMod{\lp\Modmixgr_{\Bbbk}\rp}}^{\!\op}\subseteq{\LMod{\lp\Modmixgr_{\Bbbk}\rp}}^{\!\op}\times\{\Bbbk(0)\}\xrightarrow{\Mormixgr}\Modmixgr_{\Bbbk}.
	\end{align}
\end{parag}
\begin{defn}
	\label{def:CEalgebra}
	The $\infinity$-functor \ref{functor:CEcohomologicalalgebra} is the \textit{(mixed graded) cohomological Chevalley-Eilenberg} $\infinity$-functor, and shall be denoted by $\CE^{\varepsilon}$.
\end{defn}
\begin{remark}
	\label{remark:weightCEcohomological}
	For the remainder of this subsection, it will be useful to show explicitly the underlying graded structure of the mixed graded cohomological Chevalley-Eilenberg $\CE^{\varepsilon}(\gfrak)$ for some Lie algebra $\gfrak$. For any integer $p$, we have\begin{align*}
	\CE^p{\lp\gfrak\rp}&\simeq{\lp\Mormixgr_{\operatorname{U}(\gfrak)(0)}{\lp\operatorname{U}^{\mixgr}(\Cnmixgr(\gfrak)),\hsp\Bbbk(0)\rp}\rp}_p\\
	&\simeq{\Map_{\Mod_{\operatorname{U}(\gfrak)}}{{\lp\operatorname{U}^{-p}(\Cnmixgr(\gfrak)),\hsp\Bbbk\rp}}}\\
	&\simeq {\Map_{\Mod_{\operatorname{U}(\gfrak)}}{\lp\operatorname{U}(\gfrak)(0)\otimes_{\Bbbk}\Sym^{-p}_{\Bbbk}{\lp\oblv_{\Lie}(\gfrak)[-1]\rp},\hsp\Bbbk\rp}}\\
	&\simeq {\Map_{\Mod_{\Bbbk}}{\lp{{\Sym^{-p}_{\Bbbk}}{\lp\oblv_{\Lie}(\gfrak)[-1]\rp}},\hsp\Map_{\Mod_{\operatorname{U}(\gfrak)}}{\lp\operatorname{U}(\gfrak),\hsp\Bbbk\rp}\rp}}\\
	&\simeq \Map_{\Mod_{\Bbbk}}{\lp{\Sym_{\Bbbk}^{-p}}{\lp\oblv_{\Lie}(\gfrak)[-1]\rp},\hsp\Bbbk\rp}.
	\end{align*}
	In particular, we have an equivalence of graded $\Bbbk$-modules$$\oblv_{\varepsilon}\lp \CE^{\varepsilon}(\gfrak)\rp\simeq\Mapin^{\gr}_{\Bbbk}{\lp\Symgr_{\Bbbk}{\lp\oblv_{\Lie}(\gfrak)[-1](1)\rp},\hsp\Bbbk(0)\rp}\eqqcolon\lp\Symgr_{\Bbbk}{\lp\oblv_{\Lie}(\gfrak)[-1](1)\rp}\rp^{\vee},$$where $\Mapin^{\gr}_{\Bbbk}$ is the internal graded mapping $\Bbbk$-module \infinity-functor for $\Modgr_{\Bbbk}$. 
\end{remark}
Just like in the homological case, the mixed graded cohomological Chevalley-Eilenberg $\Bbbk$-module is endowed with a much richer structure, namely it is a mixed graded commutative $\Bbbk$-algebra. 
\begin{propositionn}
	\label{prop:CEcohomologicalalgebra}
	The mixed graded Chevalley-Eilenberg $\infinity$-functor can be promoted to an $\infinity$-functor$$\CE^{\varepsilon}\colon\Lie_{\Bbbk}^{\op}\longrightarrow\CAlgmixgr_{\Bbbk//\Bbbk}$$where $\CAlgmixgr_{\Bbbk//\Bbbk}$ is the \infinity-category of augmented commutative algebra objects in $\Modmixgr_{\Bbbk}$.
\end{propositionn}
We shall deduce \cref{prop:CEcohomologicalalgebra} from the following results.
\begin{lemman}
	\label{lemma:dualcoalgebraisalgebra}
	Let $\scrC^{\otimes}$ be a closed symmetric monoidal \infinity-category with unit $\mathbbm{1}$, and denote with $(-)^{\vee}\colon\scrC^{\op}\to\scrC$ the \infinity-functor given by internally mapping objects onto $\mathbbm{1}$. Then, such $\infinity$-functor can be lifted to $(-)^{\vee}\colon\operatorname{cCcAlg}(\scrC)^{\op}\to\CAlg(\scrC)$.
\end{lemman}
\begin{proof}
	(See also \cite[Remark $5.2.5.10$]{ha}.) Given two objects $C$ and $D$ of $\scrC$, the natural map$$C^{\vee}\otimes D^{\vee}\otimes C\otimes D\simeq C^{\vee}\otimes C\otimes D^{\vee}\otimes D\overset{\operatorname{ev}_D}{\longrightarrow}C^{\vee}\otimes \mathbbm{1}\otimes C\simeq C^{\vee}\otimes C\overset{\operatorname{ev}_C}{\longrightarrow}\mathbbm{1}$$yields a lax symmetric monoidal structure on $(-)^{\vee}$ by adjunction. Hence, we have$$\CAlg(\scrC^{\op})\simeq\operatorname{cCcAlg}(\scrC)^{\op}\longrightarrow\CAlg(\scrC)$$as desired.
\end{proof}
\begin{propositionn}
	\label{prop:CEhomCEcohomdual}
	The \infinity-functor $\CE^{\varepsilon}\colon\Lie_{\Bbbk}^{\op}\to\Modmixgr_{\Bbbk}$ agrees with the composition$$\Lie_{\Bbbk}\overset{\CE_{\varepsilon}^{\op}}{\longrightarrow}\lp\CoCommmixgraug\rp^{\op}\xrightarrow{\oblv_{\operatorname{cCcAlg}}^{\op}}\lp\Modmixgr_{\Bbbk}\rp^{\op}\overset{(-)^{\vee}}{\longrightarrow}\Modmixgr_{\Bbbk}.$$
\end{propositionn}
\begin{proof}
	Given a Lie algebra $\gfrak$, we provide a morphism $$\CE^{\varepsilon}(\gfrak)\longrightarrow\CE_{\varepsilon}(\gfrak)^{\vee}\coloneqq\Mapmixgr_{{\Bbbk}}{\lp\CE_{\varepsilon}(\gfrak),\hsp\Bbbk(0)\rp}$$in the following way. By adjunction, we can restate our goal as providing a morphism$$\CE^{\varepsilon}(\gfrak)\otimesmixgr_{\Bbbk}\CE_{\varepsilon}(\gfrak)\longrightarrow\Bbbk(0).$$Recall that $\CE^{\varepsilon}(\gfrak)$ is the morphism object $\Mormixgr_{\Ug(0)}{\lp\UCngmixgr,\hsp\Bbbk(0)\rp}$; in particular, by the very same definition of \textit{morphism object} (\cite[Definition $4.2.1.28$]{ha}), there exists a map$$\CE^{\varepsilon}\otimesmixgr_{\Bbbk}\UCngmixgr\longrightarrow\Bbbk(0).$$Let us notice now that the associativity of the relative tensor product (\cite[Proposition $4.4.3.14$]{ha}) guarantees that we have equivalence of mixed graded $\Bbbk$-modules $$\CE^{\varepsilon}\otimesmixgr_{\Bbbk}\UCngmixgr\simeq\CE^{\varepsilon}\otimesmixgr_{\Bbbk}\Ug(0)\otimesmixgr_{\Ug(0)}\UCngmixgr,$$hence we can pre-compose the inverse to such equivalence with the natural morphism$$\CE^{\varepsilon}\otimesmixgr_{\Bbbk}\Bbbk(0)\otimesmixgr_{\Ug(0)}\UCngmixgr\longrightarrow\CE^{\varepsilon}\otimesmixgr_{\Bbbk}\Ug(0)\otimesmixgr_{\Ug(0)}\UCngmixgr.$$Noting that the source of such morphism is precisely $\CE^{\varepsilon}(\gfrak)\otimesmixgr_{\Bbbk}\CE_{\varepsilon}(\gfrak)$, we can compose everything accordingly and get a morphism$$\CE^{\varepsilon}(\gfrak)\otimesmixgr_{\Bbbk}\CE_{\varepsilon}(\gfrak)\longrightarrow\Bbbk(0)$$hence by adjunction$$\CE^{\varepsilon}(\gfrak)\longrightarrow\CE_{\varepsilon}(\gfrak)^{\vee}.$$Unraveling all the definition, using the description of $\CE_{\varepsilon}(\gfrak)$ (\cref{remark:explicitgradedpiecesCEcoalg}), the description of the mixed graded $\Bbbk$-linear dual (\ref{parag:monoidalstructureonmodmixgr}) and the description of $\CE^{\varepsilon}(\gfrak)$ (\cref{remark:weightCEcohomological}), we obtain that the map $\CE^{\varepsilon}(\gfrak)\longrightarrow\CE_{\varepsilon}(\gfrak)^{\vee}$, at the level of the underlying graded $\Bbbk$-modules, is precisely the identity of $\Symgr_{\Bbbk}{\lp\oblv_{\Lie}(\gfrak)[-1](1)\rp}^{\vee}$. 
\end{proof}
\cref{lemma:dualcoalgebraisalgebra} and \cref{prop:CEhomCEcohomdual} together obviously imply that, $\CE^{\varepsilon}(\gfrak)$ being the dual of the underlying $\Bbbk$-module of a cocommutative coalgebra, is endowed with a commutative algebra structure. Hence \cref{prop:CEcohomologicalalgebra} is proved.
\begin{warning}
	\label{warning:CEduals}
	In general, the $\Bbbk$-linear dual \textit{does not} exchange cocommutative coalgebras and commutative algebras: it simply sends the former to the latter. The reason lies in the lax monoidal structure of the $\infinity$-functor $(-)^\vee$, which is not oplax monoidal without some suitable finiteness assumptions on the algebra. For example, \cite[Corollary $3.2.5$]{ellipticcohomology1} states that in a symmetric monoidal \infinity-category $\scrC$ the linear dual exchanges commutative algebras and cocommutative coalgebras in the full sub-\infinity-category of $\scrC$ spanned by fully dualizable objects in the sense of \cite[Section $4.6.1$]{ha}. For our scope, this is a way too restrictive assumption: just consider the case of abelian Lie algebra $\triv_{\Lie}\Bbbk[-1]$ whose Chevalley-Eilenberg algebra is the graded $\Bbbk$-algebra $\Bbbk[t]$ where $t$ sits in homological degree $2$ and in weight $-1$. In particular $\CE^{\varepsilon}\lp \triv_{\Lie}\Bbbk[-1]\rp$ is not fully dualizable: indeed, the tensor product $\Bbbk[t]\otimesmixgr_{\Bbbk}\lp\Bbbk[t]\rp^{\vee}$ is equivalent in each weight $p$ to a direct sum of countably many copies of $\Bbbk$ sitting in homological degree $-p$, hence there is no reasonable candidate for the coevaluation morphism $\Bbbk\to\Bbbk[t]\otimesmixgr_{\Bbbk}\lp\Bbbk[t]\rp^{\vee}$. However, it is clear that $\CE^{\varepsilon}\lp \triv_{\Lie}\Bbbk[-1]\rp$ and $\CE_{\varepsilon}\lp\triv_{\Lie}\Bbbk[1]\rp$ should be one the dual of the other. So, for the mixed graded setting, we can refine \cite[Corollary $3.2.5$]{ellipticcohomology1} in the following way.
\end{warning}
\begin{propositionn}
	\label{prop:commalg&cocommcoalg}
	Let $\varepsilon\operatorname{-}\Perf^{{\operatorname{gr},-}}_{\Bbbk}$ be the full sub-\infinity-category of $\Modmixgr_{\Bbbk}$  and $\varepsilon\operatorname{-}\Perf^{{\operatorname{gr},+}}_{\Bbbk}$ be the full sub-\infinity-categories of $\Modmixgr_{\Bbbk}$ defined in \cite[Notation $1.1.7$]{pavia1}. Then the mixed graded linear dual \infinity-functor $$(-)^{\vee}\coloneqq\Mapmixgr_{\Bbbk}{\lp-,\hsp\Bbbk(0)\rp}\colon\lp\Modmixgr_{\Bbbk}\rp^{\op}\longrightarrow\Modmixgr_{\Bbbk}$$induces equivalences$$\operatorname{cCcAlg}{\lp\varepsilon\operatorname{-}\Perf^{{\operatorname{gr},-}}_{\Bbbk}\rp}^{\op}\overset{\simeq}{\longrightarrow}\CAlg{\lp\varepsilon\operatorname{-}\Perf^{{\operatorname{gr},+}}_{\Bbbk}\rp}$$and$$\operatorname{cCcAlg}{\lp\varepsilon\operatorname{-}\Perf^{{\operatorname{gr},+}}_{\Bbbk}\rp}^{\op}\overset{\simeq}{\longrightarrow}\CAlg{\lp\varepsilon\operatorname{-}\Perf^{{\operatorname{gr},-}}_{\Bbbk}\rp}.$$
\end{propositionn}
\begin{remark}
	In \cite[Proposition $1.1.8$]{pavia1}, we proved that fully dualizable objects in $\Modmixgr_{\Bbbk}$ are those mixed graded $\Bbbk$-modules which are perfect in each weight and \textit{and trivial outside at most finitely many weights}. However, the class of mixed graded (co)commutative (co)algebras described in \cref{prop:commalg&cocommcoalg} is way larger than that: we can consider even (co)commutative (co)algebras whose underlying objects, while consisting of perfect $\Bbbk$-modules in each weight, are \textit{unbounded} in weights $p\gg 0$ or $p\ll 0$, such as $\CE_{\varepsilon}(\gfrak)$ and $\CE^{\varepsilon}(\gfrak)$ for $\gfrak$ a Lie algebra which is perfect as a $\Bbbk$-module.
\end{remark}
\begin{proof}[Proof of \cref{prop:commalg&cocommcoalg}]
	Since the two statements are clearly dual one to the other, we can just prove the former. The main ingredient of the proof is the fact that the lax symmetric monoidal structure of $(-)^{\vee}$ is \textit{strict} when restricted to such \infinity-categories, hence the linear dual actually swaps cocommutative coalgebras and commutative algebras. \\
	So, let $M_{\bullet}$ and $N_{\bullet}$ be two mixed graded $\Bbbk$-modules, both perfect in each weight and both bounded above. Then, we can describe each weight of the tensor product of their dual as\begin{align*}
	\lp\Mapmixgr_{\Bbbk}{\lp M_{\bullet},\hsp\Bbbk(0)\rp}\otimesmixgr_{\Bbbk}\Mapmixgr_{\Bbbk}{\lp N_{\bullet},\hsp\Bbbk(0)\rp}\rp_p&\simeq\bigoplus_{i+j=p}\Map_{\Mod_{\Bbbk}}{\lp M_{\bullet},\hsp\Bbbk(0)\rp}_i\otimes_{\Bbbk}\Map_{\Mod_{\Bbbk}}{\lp N_{\bullet},\hsp\Bbbk(0)\rp}_j\\&\simeq\bigoplus_{i+j=p}M_{-i}^{\vee}\otimes_{\Bbbk} N_{-j}^{\vee}.
	\end{align*}On the other hand, the dual of their tensor product is described in each weight by the formula\begin{align*}
	\Mapmixgr_{\Bbbk}{\lp M_{\bullet}\otimesmixgr_{\Bbbk}N_{\bullet},\hsp\Bbbk(0)\rp}_p&\simeq \Map_{\Mod_{\Bbbk}}{\lp\bigoplus_{i+j=-p}M_i\otimes N_j,\hsp\Bbbk\rp}\\
	&\simeq\prod_{i+j=-p}\Map_{\Mod_{\Bbbk}}{\lp M_i\otimes_{\Bbbk}N_j,\hsp\Bbbk\rp}\\
	&\simeq \prod_{i+j=-p}\Map_{\Mod_{\Bbbk}}{\lp M_i,\hsp\Map_{\Mod_{\Bbbk}}{\lp N_j,\hsp\Bbbk\rp}\rp}\\&\simeq\prod_{i+j=-p}\Map_{\Mod_{\Bbbk}}{\lp M_i,\hsp\Bbbk\rp}\otimes_{\Bbbk}\Map_{{\Mod_{\Bbbk}}}{\lp N_j,\hsp\Bbbk\rp}\simeq \prod_{i+j=p}M_{-i}^{\vee}\otimes_{\Bbbk}N_{-j}^{\vee}.
	\end{align*}
	In these chains of equivalences, we only used the adjunction between tensor product and mapping spaces, and the fact that $M_{\bullet}$ is perfect in each weight. Since both $M_{\bullet}$ and $N_{\bullet}$ are bounded below, there are only finitely many admissible couples of integers $(i,j)$, hence this product is actually a sum. In particular, the lax symmetric monoidal structure provided in \cref{lemma:dualcoalgebraisalgebra} is strict, so $\Bbbk$-linear duality does swap commutative algebras and cocommutative coalgebras, and it is straightforward to show that if $A_{\bullet}$ is bounded below its dual is bounded above. (In general, if $M_{\bullet}$ is concentrated in weights $[m,n]$, for $m$ and $n$ ranging in $(-\infinity,+\infinity)$, its mixed graded $\Bbbk$-linear dual is concentrated in weights $[-n,-m]$.)\\
	The fact that $\Bbbk$-linear duality is actually an equivalence is easily checked. The mixed graded $\Bbbk$-linear duality is adjoint to itself: in particular, for any cocommutative coalgebra $C_{\bullet}$ as above we have a unit map$$C_{\bullet}\longrightarrow\Mapmixgr_{\Bbbk}{\lp\Mapmixgr_{\Bbbk}{\lp C_{\bullet},\hsp \Bbbk(0)\rp},\hsp\Bbbk(0)\rp}.$$The $\infinity$-functor is fully faithful if and only if this map is an equivalence, and since forgetting the cocommutative coalgebra structure is conservative, we can check this at the level of the underlying mixed graded $\Bbbk$-module. Using the usual formulas for the internal mixed graded mapping $\Bbbk$-module and the fact that $C_{\bullet}$ is perfect in each weight, one straightforwardly checks that this map is actually an equivalence on all weights, hence it is an equivalence.	
\end{proof}
\begin{remark}
	\label{remark:CEgperfect}
	Let $\gfrak$ be a Lie algebra which is perfect as a $\Bbbk$-module. Then its Chevalley-Eilenberg algebra $\CE^{\varepsilon}(\gfrak)$ is described, in each weight, by $$\CE^p(\gfrak)\simeq\Sym^p_{\Bbbk}{\lp\oblv_{\Lie}(\gfrak)[-1](1)\rp}^{\vee}.$$Being $\gfrak$ (and in particular $\gfrak[-1]$) perfect, every symmetric power $\Sym^p_{\Bbbk}(\oblv_{\Lie}(\gfrak)[-1])$ is perfect, and moreover one has an equivalence$$\Sym^p_{\Bbbk}(\oblv_{\Lie}(\gfrak)[-1])^{\vee}\simeq\Sym^p_{\Bbbk}{\lp\lp\oblv_{\Lie}(\gfrak)[-1]\rp^{\vee}\rp}\simeq \Sym^p_{\Bbbk}{\lp\oblv_{\Lie}(\gfrak)^{\vee}[1]\rp}.$$In particular, under these assumptions, the underlying mixed graded $\Bbbk$-module of $\CE^{\varepsilon}(\gfrak)$ lies in $\Perfmixgrccn\subseteq\varepsilon\text{-}\Perf^{\gr,-}_{\Bbbk}$, hence its dual is a cocommutative coalgebra which is precisely $\CE_{\varepsilon}(\gfrak)$. This fixes the drawback highlighted in \cref{warning:CEduals}.
\end{remark}
\cref{prop:CEhomCEcohomdual} straight-forwardly implies the following comparison result.
\begin{propositionn}
	\label{prop:CEcohomexplicitmodel}
	Given a Lie algebra $\gfrak$, the Tate realization $\left|\CE^{\varepsilon}(\gfrak)\right|^{\operatorname{t}}$ agrees with the cohomological Chevalley-Eilenberg complex of \cite[Construction $2.2.13$]{dagx}.
\end{propositionn} 
\begin{proof}
	In virtue of \cref{prop:commalg&cocommcoalg}, and using the fact that $\CE^{\varepsilon}(\gfrak)$ is non-negatively graded, we have the following chain of equivalences of $\Bbbk$-modules.\begin{align*}
	\left|\CE^{\varepsilon}(\gfrak)\right|^{\operatorname{t}}&\simeq\left|\Mapmixgr_{\Bbbk}{\lp\CE_{\varepsilon}(\gfrak),\hsp\Bbbk(0)\rp}\right|^{\operatorname{t}}\\&\simeq\left|\Mapmixgr_{\Bbbk}{\lp\CE_{\varepsilon}(\gfrak),\hsp\Bbbk(0)\rp}\right|\\
	&\simeq \Map_{\Modmixgr_{\Bbbk}}{\lp\Bbbk(0),\hsp\Mapmixgr_{{\Bbbk}}{\lp\CE_{\varepsilon}(\gfrak),\hsp\Bbbk(0)\rp}\rp}\\
	&\simeq \Map_{\Modmixgr_{\Bbbk}}{\lp\CE_{\varepsilon}(\gfrak),\hsp \Bbbk(0)\rp}
	\end{align*}
	where the latter $\Bbbk$-module provides the $\Mod_{\Bbbk}$-enrichment of $\Modmixgr_{\Bbbk}$. On the other hand, the usual cohomological Chevalley-Eilenberg complex is the mapping $\Bbbk$-module $\Map_{\Mod_{\Bbbk}}{\lp\CE_{\bullet},\hsp\Bbbk\rp}$ (\cite[Remark $2.2.14$]{dagx}), where $\CE_{\bullet}(\gfrak)$ is the usual Chevalley-Eilenberg homology of $\gfrak$. \cref{prop:CEexplicitmodel} yields an equivalence $\CE_{\bullet}(\gfrak)\simeq\left|\CE_{\varepsilon}(\gfrak)\right|^{\operatorname{t}}$, our claim would follow if we had an equivalence$$\Map_{\Modmixgr_{\Bbbk}}{\lp\CE^{\varepsilon}(\gfrak),\hsp\Bbbk(0)\rp}\simeq \Map_{\Mod_{\Bbbk}}{\lp\left|\CE_{\varepsilon}(\gfrak)\right|^{\operatorname{t}},\hsp\Bbbk\rp}.$$But this is true because of our computation of the explicit right adjoint to the Tate realization when restricted to non-negatively graded mixed $\Bbbk$-modules in \cite[Porism $1.3.21$]{pavia1}, which agrees with $\left|-\right|^{\operatorname{t}}$.
\end{proof}
\begin{remark}
	\cref{remark:weightCEcohomological} and \cref{prop:CEcohomexplicitmodel} suggest that our construction matches the explicit mixed graded structure on the cohomological Chevalley-Eilenberg algebra described in \cite[Proposition A.$3$]{CG}.
\end{remark}
\begin{remark}
Since the associated filtered \infinity-functor $(-)^{\fil}\colon\Modmixgr_{\Bbbk}\to\Modfil_{\Bbbk}$ is lax monoidal, $\lp\CE^{\varepsilon}(\gfrak)\rp^{\fil}$ is again a commutative algebra object in $\Modfil_{\Bbbk}$, whose underlying graded corresponds to$$\left\{\Gr_n\lp\CE^{\varepsilon}(\gfrak)\rp^{\fil}\right\}_{n\in\ZZ}\simeq \left\{\Sym^{n}_{\Bbbk}{\lp\gfrak^{\vee}[1]\rp}[-2n]\right\}_{n\in\ZZ}$$and whose underlying object $\lp\CE^{\varepsilon}(\gfrak)\rp^{\fil}_{-\infty}$ is the \textit{completed} cohomological Chevalley-Eilenberg algebra $\widehat{\CE}^{\bullet}(\gfrak)$, i.e., the completion of $\CE^{\bullet}(\gfrak)$ with respect to the ideal generated by $\gfrak^{\vee}[1]$. This recovers the usual filtration on the (completed) cohomological Chevalley-Eilenberg algebra by decreasing symmetric powers.
\end{remark}
\subsection{Representations and mixed graded modules}
\label{sec:CEmod}
Just like Lie algebras can be seen inside mixed graded commutative $\Bbbk$-algebras via the cohomological Chevalley-Eilenberg construction, the \infinity-categories of representations over Lie algebras live naturally inside the \infinity-category of mixed graded $\Bbbk$-modules, i.e., also the Chevalley-Eilenberg $\Bbbk$-modules with coefficients in some representation admit a mixed graded structure. Throughout all this subsction, we shall fix a Lie algebra $\gfrak$: while we expect our construction to be fully faithful only with some size constraints over $\gfrak$ (\cref{conj:modules}), it can be carried out in all generality for any Lie algebra.
\begin{parag}
	\label{parag:CEmod}
	Let $\LMod_{\Ug}\simeq\Rep_{\gfrak}$ be the $\infinity$-category of left $\Ug$-modules (hence, of representations of $\gfrak$), endowed with the Hopf symmetric monoidal structure of \cref{construction:monoidalstructureonLModUG}. Since the trivial representation $\Bbbk$ is the unit for such symmetric monoidal structure, we have a natural equivalence $$\LMod_{\Ug}\simeq \Mod_{\Bbbk}{\lp\LMod_{\Ug}^{\otimes_{\gfrak}}\rp}.$$Recall moreover the $\infinity$-functor \ref{functor:CEcohomological}$$\CE^{\varepsilon}(\gfrak;\hsp-)\colon\LMod_{\Ug}\longrightarrow\Modmixgr_{\Bbbk}$$which sends a left $\Ug$-module to the mixed graded $\Bbbk$-module$$\CE^{\varepsilon}(\gfrak;\hsp M)\coloneqq\Mormixgr_{\Ug(0)}{\lp\UCngmixgr,\hsp M(0)\rp}.$$Again, looking at the weight $p$ component, we have an equivalence of mixed graded $\Bbbk$-modules\begin{align*}
	\CE^p(\gfrak;\hsp M)&\coloneqq\Mormixgr_{\Ug(0)}{\lp\UCngmixgr,\hsp M(0)\rp}_p\\&\simeq \Mor_{\Ug}{\lp \UCngmixgr_{-p},\hsp M\rp}\\&\simeq \Mor_{\Ug}{\lp \Ug\otimes_{\Bbbk}\Sym_{\Bbbk}^{-p}{\lp \oblv_{\Lie}\gfrak[-1]\rp},\hsp M\rp}\simeq \Map_{\Mod_{\Bbbk}}{\lp \Sym_{\Bbbk}^{-p}{\lp \oblv_{\Lie}\gfrak[-1]\rp},\hsp M\rp}.
	\end{align*}
	Again, the mixed differential is completely inherited from the action of $\BGa$ on $\UCngmixgr$.
\end{parag}
\begin{remark}
	\label{remark:whengisperfect}
	When $\gfrak$ is perfect as a $\Bbbk$-module, we can write the weight $p$ component of $\CE^{\varepsilon}(\gfrak;\hsp M)$ in another way. The discussion of \ref{parag:CEmod} yields that $$\CE^p(\gfrak;\hsp M)\simeq \Map_{\Mod_{\Bbbk}}{\lp \Sym_{\Bbbk}^{-p}{\lp \oblv_{\Lie}\gfrak[-1]\rp},\hsp M\rp}.$$But if $\gfrak$ is perfect as a $\Bbbk$-module, then so are $\gfrak[-1]$ and each symmetric power of $\gfrak[-1].$ In particular, we can write\begin{align*}
	\Map_{\Mod_{\Bbbk}}{\lp \Sym_{\Bbbk}^{-p}{\lp \oblv_{\Lie}\gfrak[-1]\rp},\hsp M\rp}\simeq \Map_{\Mod_{\Bbbk}}{\lp \Sym_{\Bbbk}^{-p}{\lp \oblv_{\Lie}\gfrak[-1]\rp},\hsp\Bbbk\rp}\otimes_{\Bbbk}M
	\end{align*}and since we are in characteristic $0$, we can write $$\Map_{\Mod_{\Bbbk}}{\lp \Sym_{\Bbbk}^{-p}{\lp \oblv_{\Lie}\gfrak[-1]\rp},\hsp\Bbbk\rp}\simeq \Sym^{-p}_{\Bbbk}{\lp\oblv_{\Lie}\gfrak^{\vee}[1]\rp},$$where $(-)^{\vee}$ denotes the usual $\Bbbk$-linear dual of $\Bbbk$-modules. In particular, we have an equivalence of graded $\Bbbk$-modules$$\oblv_{\varepsilon}\CE^{\varepsilon}(\gfrak)\simeq\Symgr_{\Bbbk}{\lp\oblv_{\Lie}\gfrak^{\vee}[1](-1)\rp}\otimes^{\gr}_{\Bbbk}M(0).$$
\end{remark}
\begin{propositionn}
	\label{prop:CEmod}
	The $\infinity$-functor \ref{functor:CEcohomological} can be promoted to an $\infinity$-functor$$\CE^{\varepsilon}(\gfrak;\hsp-)
	\colon\LMod_{\Ug}\longrightarrow\Mod_{\CE^{\varepsilon}(\gfrak)}{\lp\Modmixgr_{\Bbbk}\rp}$$where $\ModCEmixgr$ is the \infinity-category of modules for the commutative algebra $\CE^{\varepsilon}(\gfrak)$ in $\Modmixgr_{\Bbbk}$.
\end{propositionn}
In order to prove \cref{prop:CEmod}, we shall need the following technical lemma.
\begin{lemman}
	\label{lemma:UGmodalgebras}
	Given two mixed graded left $\Ug$-modules $M_{\bullet}$ and $N_{\bullet}$, let $M_{\bullet}\otimes_{\gfrak}N_{\bullet}$ the mixed graded tensor product of \cref{construction:monoidalstructureonLModUG}, i.e., the mixed graded tensor product of left modules for the mixed graded Hopf algebra $\Ug(0)$. For all mixed graded $\Bbbk$-modules $P_{\bullet}$, there is an equivalence of mixed graded left $\Ug$-modules$${\lp M_{\bullet}\otimes_{\gfrak}N_{\bullet}\rp}	\otimesmixgr_{\Bbbk}P_{\bullet}\simeq  M_{\bullet}\otimes_{\gfrak}{\lp N_{\bullet}\otimesmixgr_{\Bbbk}P_{\bullet}\rp}$$where $\otimesmixgr_{\Bbbk}$ denotes the right action of $\Modmixgr_{\Bbbk}$ over $\LModmixgr_{\Ug(0)}$ given by the relative tensor product over $\Bbbk(0)$ described in \ref{parag:LModUGenriched}.
\end{lemman}
\begin{proof}
	Recall that, in virtue of \cite[Remark $4.2.1.21$]{ha}, the fact that $\LModmixgr_{\Ug(0)}$ is right tensored over $\Modmixgr_{\Bbbk}$ is equivalent to the saying that $\LModmixgr_{\Ug(0)}$ is a right $\Modmixgr_{\Bbbk}$-module in the \infinity-category of (not necessarily small) \infinity-categories $\Catinfty$. Since $\Modmixgr_{\Bbbk}$ is a symmetric monoidal \infinity-category, hence an $\Einf$-algebra in the \infinity-category $\Catinfty$, it follows that the $\LModmixgr_{\Ug(0)}$ is both a left and a right $\Modmixgr_{\Bbbk}$-module in \infinity-categories in a compatible way, i.e., it is actually a $\Modmixgr_{\Bbbk}$-bimodule (\cite[Section $4.5$]{ha}). This implies that our request is equivalent to saying that $\LModmixgr_{\Ug(0)}$ is a $\Modmixgr_{\Bbbk}$-algebra object in $\Catinfty$. Since both $\Modmixgr_{\Bbbk}$ and $\LModmixgr_{\Ug(0)}$ are $\Einf$-monoidal categories, a $\Modmixgr_{\Bbbk}$-algebra structure on $\LModmixgr_{\Ug(0)}$ is equivalent to a symmetric monoidal \infinity-functor  $\Modmixgr_{\Bbbk}\to\LModmixgr_{\Ug(0)}$. But now $\Ug(0)$ is an augmented associative $\Bbbk$-algebra, hence pullback along the augmentation provides an \infinity-functor$$F\colon\Modmixgr_{\Bbbk}\longrightarrow\LModmixgr_{\Ug(0)}.$$We are only left to prove that this \infinity-functor is strongly monoidal, but this is a straight-forward computation. Indeed, given two mixed graded $\Bbbk$-modules $M_{\bullet}$ and $N_{\bullet}$, the underlying mixed graded $\Bbbk$-module of both $FM_{\bullet}\otimes_{\gfrak}FN_{\bullet}$ and $F{\lp M_{\bullet}\otimesmixgr_{\Bbbk}N_{\bullet}\rp}$ is exactly $M_{\bullet}\otimesmixgr_{\Bbbk}N_{\bullet}$. Moreover, the action of $\Ug$ is given, in both cases, by the augmentation over $\Bbbk(0)$, hence the identity of $M_{\bullet}\otimesmixgr_{\Bbbk}N_{\bullet}$ is a morphism of mixed graded left $\Ug$-modules which yields our claim.
\end{proof}
\begin{proof}[Proof of \cref{prop:CEmod}]
	In order to prove our assertion, we only need to show that $\CE^{\varepsilon}(\gfrak;\hsp-)$ is lax monoidal with respect to the Hopf monoidal structure of \cref{construction:monoidalstructureonLModUG} and the mixed graded tensor product of $\Bbbk$-modules. Then our claim will follow from the fact that $\CE^{\varepsilon}(\gfrak;\hsp \Bbbk)$ is, by definition, the underlying mixed graded $\Bbbk$-module of $\CE^{\varepsilon}(\gfrak)$ (see \cref{def:CEalgebra} and \cref{prop:CEcohomologicalalgebra}). \\
	Let $M$ and $N$ be two representations of $\gfrak$ seen as left $\Ug$-modules. By the definition of the \infinity-functor \ref{functor:CEcohomological}, for any left $\Ug$-module $M$ we have a natural evaluation map\begin{align*}
	\UCngmixgr\otimesmixgr_{\Bbbk}\CE^{\varepsilon}(\gfrak;\hsp M)\longrightarrow M(0).
	\end{align*}
	Tensoring over $\gfrak$ with $\UCngmixgr$ on the left and tensoring over $\Bbbk$ with $\CE^{\varepsilon}(\gfrak;\hsp N)$ on the right, we obtain a map\begin{align}
	\label{eq:2}
	\begin{split}
	\UCngmixgr\otimes_{\gfrak}\UCngmixgr\otimesmixgr_{\Bbbk}\CE^{\varepsilon}(\gfrak;\hsp M)\otimesmixgr_{\Bbbk}\CE^ {\varepsilon}(\gfrak;\hsp N)\longrightarrow\\\longrightarrow \UCngmixgr\otimes_{\gfrak}M(0)\otimesmixgr_{\Bbbk}\CE^{\varepsilon}(\gfrak;\hsp N).
	\end{split}
	\end{align}Let us remark that in the above formula there is no associativity ambiguity, because $\LModmixgr_{\Ug(0)}$ is a $\Modmixgr_{\Bbbk}$-algebra in \infinity-categories in virtue of \cref{lemma:UGmodalgebras}. Since the tensor product of left $\Ug$-modules is symmetric monoidal, we can swap the first two factors in the target and then apply again the evaluation map:
	\begin{align}
	\label{equation:2}
	M(0)\otimes_{\gfrak}\UCngmixgr\otimesmixgr_{\Bbbk}\CE^{\varepsilon}(\gfrak;\hsp N)\longrightarrow M(0)\otimes_{\gfrak}N(0)
	\end{align}
	Recall now that $\UCngmixgr$ is a mixed graded Hopf algebra, since it is the universal enveloping mixed graded algebra of the mixed graded Lie algebra $\Cnmixgr(\gfrak)$. Moreover, the comultiplication map $$\mu\colon\UCngmixgr\longrightarrow\UCngmixgr\otimesmixgr_{\Bbbk}\UCngmixgr$$is a morphism of mixed graded left $\Ug$-modules, since it is a morphism of mixed graded left $\UCngmixgr$-modules (in virtue of \cref{remark:ugmapleftmodules}), and the action of $\Ug$ over $\UCngmixgr$ is given by the inclusion of $\Ug$ in weight $0$. In particular, tensoring the comultiplication for $\UCngmixgr$ with the identity on $\CE^{\varepsilon}(\gfrak;\hsp M)\otimesmixgr_{\Bbbk}\CE^{\varepsilon}(\gfrak;\hsp N)$ and post-composing with \ref{eq:2} and \ref{equation:2}, we obtain a map$$\UCngmixgr\otimesmixgr_{\Bbbk}\CE^{\varepsilon}(\gfrak;\hsp M)\otimesmixgr_{\Bbbk}\CE^{\varepsilon}(\gfrak;\hsp N)\longrightarrow M(0)\otimes_{\gfrak}N(0).$$By adjunction, this is the same as a map$$\CE^{\varepsilon}(\gfrak;\hsp M)\otimesmixgr_{\Bbbk}\CE^{\varepsilon}(\gfrak;\hsp N)\longrightarrow\Mormixgr_{\Ug(0)}{\lp\UCngmixgr,\hsp \lp M\otimes_{\gfrak} N\rp(0)\rp}\eqqcolon \CE^{\varepsilon}(\gfrak;\hsp M\otimes_{\gfrak}N),$$which testifies the lax monoidality of $\CE^{\varepsilon}(\gfrak;\hsp-)$.
\end{proof}
\cref{prop:CEmod} allows us to promote the \infinity-functor $\CE^{\varepsilon}(\gfrak;\hsp-)$ to a \infinity-functor\begin{align}
\label{functor:CEmod}
\CE^{\varepsilon}(\gfrak;\hsp-)\colon\LMod_{\Ug}\longrightarrow\Mod_{\CE^{\varepsilon}(\gfrak)}{\lp\Modmixgr_{\Bbbk}\rp}.
\end{align}
\begin{warning}
	The \infinity-functor $\CE^{\varepsilon}(\gfrak;\hsp-)$ is \textit{never} strongly monoidal, not even if $\gfrak$ is perfect. In the latter case, we know by the discussion in \cref{remark:whengisperfect} that the underlying graded of $\CE^{\varepsilon}(\gfrak;\hsp M)$ is $\Symgr_{\Bbbk}{\lp\oblv_{\Lie}\gfrak^{\vee}[1](-1)\rp}\otimes_{\Bbbk}M(0)$, so we can interpret the natural map$$\CE^{\varepsilon}(\gfrak;\hsp M)\otimesmixgr_{\Bbbk}\CE^{\varepsilon}(\gfrak;\hsp N)\longrightarrow\CE^{\varepsilon}(\gfrak;\hsp M\otimes_{\gfrak}N)$$simply as tensoring the natural multiplication of the commutative symmetric algebra over $\oblv_{\Lie}\gfrak^{\vee}[1]$, which is compatible with the Chevalley-Eilenberg mixed differential, with the identity of $M\otimes_{\Bbbk}N$. Again, such multiplication map is almost never an isomorphism.
\end{warning}
\begin{warning}
	Even if $\gfrak$ is perfect, hence there is an equivalence of graded $\Bbbk$-modules between $\oblv_{\varepsilon}\lp\CE^{\varepsilon}(\gfrak)\otimesmixgr_{\Bbbk}M(0)\rp\simeq\Symgr_{\Bbbk}{\lp\oblv_{\Lie}\gfrak^{\vee}[1](-1)\rp}\otimes^{\gr}_{\Bbbk}M(0)$ and $\oblv_{\varepsilon}\CE^{\varepsilon}(\gfrak;\hsp M)$, this equivalence is \textit{not} an equivalence of mixed graded $\Bbbk$-modules. Indeed, one can construct a map$$\UCngmixgr\otimesmixgr_{\Bbbk}\CE^{\varepsilon}(\gfrak)\otimesmixgr_{\Bbbk}M(0)\longrightarrow\Bbbk(0)\otimesmixgr_{\Bbbk}M\overset{\simeq}{\longrightarrow}M(0)$$but this is only a map of mixed graded $\Bbbk$-modules, and does not provide by adjunction a map $\CE^{\varepsilon}(\gfrak)\otimesmixgr_{\Bbbk}M(0)\to\CE^{\varepsilon}(\gfrak;\hsp M)$. This cannot be avoided: doing some checking with explicit mixed graded chain complexes of $\Bbbk$-modules, one can see that the mixed differential $\varepsilon_0$ on the domain is given by$$\Bbbk\otimes_{\Bbbk} M\simeq M \overset{0}{\longrightarrow}\gfrak^{\vee}[1]\otimes_{\Bbbk}M[-1]$$while the one on the codomain is given by$$\alpha\colon M\longrightarrow\gfrak^{\vee}[1]\otimes_{\Bbbk}M[-1]\simeq \Map_{\Mod_{\Bbbk}}{\lp\gfrak,\hsp M\rp}$$where $\alpha$ is the morphism that encodes the left action of $\gfrak$ on $M$. 
\end{warning}
We conclude this subsection with just a simple observation.
\begin{propositionn}
	\label{prop:CEalllimitscolimits}
	The Chevalley-Eilenberg \infinity-functor $$\CE^{\varepsilon}(\gfrak;\hsp-)\colon\LMod_{\Ug}\longrightarrow\ModCEmixgr$$is an accessible \infinity-functor which preserves all limits. If $\gfrak$ is perfect, it preserves also all colimits.
\end{propositionn}
\begin{proof}
	Recall that, if $\scrC$ is a monoidal \infinity-category such that the tensor product preserves colimits separably in each variable, then colimits in the \infinity-category of modules over some algebra of $\scrC$ are detected by the forgetful \infinity-functor towards $\scrC$ (\cite[Corollary $4.2.3.5$]{ha}). Moreover, both $\LMod_{\Ug}$ and $\ModCEmixgr$ are presentable \infinity-categories, being the \infinity-categories of modules over an algebra in some presentable stable \infinity-category. In particular, our statement reduces to show that $\CE^{\varepsilon}(\gfrak;\hsp-)$ sends limits of left $\Ug$-modules to limits of mixed graded $\Bbbk$-modules. The inclusion$$(-)(0)\colon\LMod_{\Ug}\longhookrightarrow\LModmixgr_{\Ug(0)}$$preserves all limits and colimits, since they are computed weight-wise in the \infinity-category of mixed graded $\Bbbk$-modules. So we are left to prove that $$\Mormixgr_{\Ug(0)}{\lp\UCngmixgr,\hsp-\rp}\colon\LModmixgr_{\Ug(0)}\longrightarrow\Modmixgr_{\Bbbk}$$commutes with limits. This is clear, since  the \infinity-functor is constructed as the right adjoint to the right tensor action of $\Modmixgr_{\Bbbk}$ over $\LModmixgr_{\Ug(0)}$.\\Now, suppose that $\gfrak$ is perfect as a $\Bbbk$-module. Then by the equivalence$$\oblv_{\varepsilon}\CE^{\varepsilon}(\gfrak;\hsp M)\simeq\Symgr_{\Bbbk}{\lp\oblv_{\Lie}\gfrak^{\vee}[1](-1)\rp}\otimes^{\gr}_{\Bbbk}M(0)$$provided in \cref{remark:whengisperfect}, we see that colimits of left $\Ug$-modules are preserved by $\CE^{\varepsilon}(\gfrak;\hsp-)$, since the tensor product commutes with them. 
\end{proof}
\section{Conjectures and open problems}
\label{chapter:conj}
As already mentioned in \cref{chapter:liealgebras}, we conjecture that mixed graded $\Bbbk$-modules offer a framework in which homotopy Lie algebras and their representations can be fully recovered. In this second section, we shall carefully formalize our expectations, providing some conjectures, showing how they are related one to the other, and hinting at possible strategies to adopt in order to prove our assertions. 
\subsection{Lie algebras as mixed graded coalgebras}
In this section, we shall describe how we expect that Lie algebras live fully faithfully inside the \infinity-category of mixed graded cocommutative coalgebras. Our standing assumptions and notations are the same as the ones of \cref{chapter:liealgebras}. 
\begin{defn}
	\label{def:CEcoalgebras}
	We say that an augmented mixed graded cocommutative coalgebra $A_{\bullet}$ is \textit{of Chevalley-Eilenberg type} (or \textit{semi-free}) if its underlying graded (augmented) cocommutative coalgebra is equivalent to a symmetric (free) cocommutative coalgebra $\Sym^{\gr}_{\Bbbk}{\lp M[-1](1)\rp}$, for some $\Bbbk$-module $M$.
\end{defn}
\begin{parag}
	Mixed graded cocommutative coalgebras of Chevalley-Eilenberg type are naturally gathered in a full sub-$\infinity$-category $\CEcAlgmixgr$ of $\CoCommmixgraug$, which we can express via a pullback of $\infinity$-categories in the following way.\begin{align}
	\begin{tikzpicture}[scale=0.75,baseline=(current  bounding  box.center)]
	\label{square:CEpullback}
	\node (a) at (-4,6){$\CEcAlgmixgr$};
	\node (c1) at (2,6){$\CoCommmixgraug$};
	\node (b) at (-4,3){$\Mod_{\Bbbk}$};
	\node (d) at (2,3){$\CoCommgraug$};
	\node at (-3.2,5.2){$\lrcorner$};
	\draw[->, font=\scriptsize] (a) to node[left]{} (b);
	\draw[->, font=\scriptsize] (c1) to node[right]{$\oblv_{\varepsilon}$} (d);
	\draw[->, font=\scriptsize] (a) to node[left]{} (c1);
	\draw[->, font=\scriptsize] (b) to node[below]{$\ref{functor:gradedsym}$} (d);
	\end{tikzpicture}\end{align}The above description exhibits $\CEcAlgmixgr$ as a pullback of presentable $\infinity$-categories: in fact, $\Mod_{\Bbbk}$ is obviously presentable (\cite[Corollary $4.2.3.7$]{ha}) while $\CoCommgraug$ and $\CoCommmixgraug$ are $\infinity$-categories of (augmented) cocommutative coalgebra objects in some symmetric monoidal $\infinity$-category whose tensor product preserves colimits separately in each variable. Hence, they are presentable in virtue of \cite[Proposition $3.1.3$]{ellipticcohomology1}. \\Moreover, it is clear from the description of $\CE_{\varepsilon}(\gfrak)$ yielded in \cref{remark:explicitgradedpiecesCEcoalg} that the $\infinity$-functor $$\CE_{\varepsilon}\colon\Lie_{\Bbbk}\longrightarrow\CoCommmixgraug$$factors through $\CEcAlgmixgr$. 
\end{parag}
We characterize some main properties of the $\infinity$-category $\CEcAlgmixgr$. 
\begin{lemman}
	\label{lemma:CElocalization}
	The $\infinity$-category $\CEcAlgmixgr$ is a localization (in the sense of \cite[$5.2.7.2$]{htt}) of the $\infinity$-category $\CoCommmixgraug$.
\end{lemman}
\begin{proof}
	Being a full presentable sub-$\infinity$-category of a presentable $\infinity$-category, it is enough to show that the inclusion preserves all limits and $\kappa$-filtered colimits for some regular cardinal $\kappa$. The discussion in \ref{parag:forgettingmixedstructureisok} allows us to reduce ourselves to check the claim at the level of the underlying graded coalgebras. But now the assertion is clear, since $\oblv_{\varepsilon}\circ\hsp\CE_{\varepsilon}$ is naturally equivalent to $\Symgr_{\Bbbk}{\lp(-)[-1](1)\rp}\circ\oblv_{\Lie}$, which is a composition of two right adjoints and as such preserves all limits and $\kappa$-filtered colimits for some regular cardinal $\kappa$.
\end{proof}
\begin{parag}
	\label{parag:CEpreserveslimits}
	\cref{lemma:CElocalization} implies that we have a left adjoint $$\operatorname{L}_{\CE}\colon\CoCommmixgraug\longrightarrow\CEcAlgmixgr$$which, by usual abstract nonsense, is the identity when restricted to mixed graded cocommutative coalgebra of Chevalley-Eilenberg type. Moreover, a slight reworking of the proof of \cref{prop:promotionCEcocommutativealgebra} shows that $\CE_{\varepsilon}$ actually preserves \textit{all} limits: we just need to observe that limits of Lie algebras are sent to limits on the underlying $\Bbbk$-modules by the forgetful $\infinity$-functor $\oblv_{\Lie}\colon\Lie_{\Bbbk}\to\Mod_{\Bbbk}$, and then use again that $\Sym^{\gr}_{\Bbbk}$ is a right adjoint to get
	\begin{align*}
	\Symgr_{\Bbbk}{\lp \oblv_{\Lie}{\lp\lim_{i\in I}\gfrak_i\rp}[-1](1)\rp} &\simeq \Symgr_{\Bbbk}{\lp \lp\lim_i{\lp \oblv_{\Lie}\lp\gfrak_i\rp\rp}\rp[-1](1)\rp}\\
	&\simeq \Symgr_{\Bbbk}{\lp \lim_i{\lp \oblv_{\Lie}\lp\gfrak_i\rp[-1](1)\rp}\rp}\\&\simeq \lim_i{\lp\Symgr_{\Bbbk}{\lp\oblv_{\Lie}\lp\gfrak_i\rp[-1](1)\rp}\rp}.
	\end{align*}
	In particular, $$\CE_{\varepsilon}\colon\Lie_{\Bbbk}\longrightarrow\CEcAlgmixgr$$admits a left adjoint as well, which we denote with \begin{align}
	\label{functor:CEleftadjoint}
	\Lfrak\colon\CEcAlgmixgr\longrightarrow\Lie_{\Bbbk}.
	\end{align}
\end{parag}
\begin{conjn}
	\label{conj:main}
	The \infinity-functor $\CE_{\varepsilon}\colon\Lie_{\Bbbk}\to\CoCommmixgraug$ is a fully faithful embedding. The essential image is the full sub-\infinity-category $\CEcAlgmixgr$ of cocommutative mixed graded coalgebras of Chevalley-Eilenberg type.
\end{conjn}
As a direct consequence of \cref{conj:main} and \cref{prop:commalg&cocommcoalg}, we have the following cohomological analogue.
\begin{conjn}
	\label{conj:main2}
	The \infinity-functor $\CE^{\varepsilon}\colon\Lie^{\op}_{\Bbbk}\to\CAlgmixgr_{\Bbbk}$ restricts to a fully faithful embedding$$\Lieperf_{\Bbbk}\longhookrightarrow\CAlgmixgr_{\Bbbk}$$where $\Lieperf_{\Bbbk}\coloneqq\Alg_{\Lie}{\lp\Perf_{\Bbbk}\rp}$ is the \infinity-category of Lie algebras whose underlying $\Bbbk$-module is perfect.
\end{conjn}
\cref{conj:main} and \cref{conj:main2} can be summarized as follows. Recall the definition of the \infinity-categories $\Perfmixgrcn_{\Bbbk}$ and $\Perfmixgrccn_{\Bbbk}$ of \cite[Notation $1.1.7$]{pavia1}: they are closed under the mixed graded tensor product, hence we can consider both commutative algebras and cocommutative coalgebras in these \infinity-categories. Let moreover $\operatorname{cLie}_{\Bbbk}$ be the \infinity-category of Lie coalgebras in $\Mod_{\Bbbk}$: one has an \infinity-functor$$\operatorname{coChev}\colon\operatorname{cLie}_{\Bbbk}\longrightarrow\operatorname{CAlg}_{\Bbbk//\Bbbk}$$which computes the Chevalley-Eilenberg cohomology of Lie coalgebras (\cite[Section $5$]{colie}). 
\begin{conjn}
	\label{conj:general}
	There exists an $\infinity$-functor $\operatorname{cC}^{\varepsilon}\colon\operatorname{cLie}_{\Bbbk}\to\CAlgmixgr_{\Bbbk}$, such that its Tate realization agrees with $\operatorname{coChev}$, which sits in a diagram of \infinity-categories$$\begin{tikzpicture}[scale=0.75]
	\node (a) at (-6.5,4){$\Lie_{\Bbbk}$};
	\node (b) at (6.5,4){$\CoCommmixgraug$};
	\node (a1) at (-4.5,2) {$\Lieperf_{\Bbbk}$};
	\node (b1) at (4.5,2){${\operatorname{cCcAlg}}{{\lp\Perfmixgrcn_{\Bbbk}\rp}}_{\Bbbk/}$};
	\node (c) at (-6.5,-2){$\lp\operatorname{cLie}_{\Bbbk}\rp^{\op}$};
	\node (d) at (6.5,-2){$\lp\CAlgmixgraug\rp^{\op}$};
	\node (c1) at (-4.5,0) {$\lp\operatorname{cLie}^{\operatorname{perf}}_{\Bbbk}\rp^{\op}$};
	\node (d1) at (4.5,0){$\CAlg{\lp\Perfmixgrccn_{\Bbbk}\rp}_{/\Bbbk}^{\op}$};
	\draw[right hook->,font=\scriptsize] (a) to node[above]{$\CE_{\varepsilon}$}(b);
	\draw[right hook->,font=\scriptsize] (a1) to node[above]{$\CE_{\varepsilon}$}(b1);
	\draw[right hook->,font=\scriptsize] (c) to node[below]{$\lp\operatorname{cC}^{\varepsilon}\rp^{\op}$}(d);
	\draw[right hook->,font=\scriptsize] (c1) to node[below]{$\lp\operatorname{cC}^{\varepsilon}\rp^{\op}$}(d1);
	\draw[right hook->,font=\scriptsize] (a1) to (a);
	\draw[right hook->,font=\scriptsize] (b1) to (b);
	\draw[right hook->,font=\scriptsize] (c1) to (c);
	\draw[right hook->,font=\scriptsize] (d1) to (d);
	\draw[->,font=\scriptsize] (b) to[bend left] node[right]{$(-)^{\vee}$} (d);
	\draw[<->,font=\scriptsize] (a1) to node[left]{$(-)^{\vee}$} (c1);
	\draw[<->,font=\scriptsize] (a1) to node[above,rotate=-90]{$\simeq$} (c1);
	\draw[<->,font=\scriptsize] (b1) to node[left]{$(-)^{\vee}$} (d1);
	\draw[<->,font=\scriptsize] (b1) to node[above,rotate=-90]{$\simeq$} (d1);
	\end{tikzpicture}$$
	which commutes in every direction.
\end{conjn}
\begin{remark}
	While the dual of a Lie coalgebra is always a Lie algebra, the converse is true only for perfect Lie algebras; hence, we cannot hope to have a well defined $\Bbbk$-linear dual \infinity-functor $\Lie_{\Bbbk}\to\operatorname{cLie}_{\Bbbk}$. Moreover, the diagram in \cref{conj:general} cannot commute in every direction if we add the $\Bbbk$-linear dual $\operatorname{cLie}_{\Bbbk}\to\operatorname{Lie}_{\Bbbk}$: if $\mathfrak{a}$ is an abelian (i.e., trivial) Lie coalgebra, its Chevalley-Eilenberg algebra is equivalent as a commutative algebra to $\Sym_{\Bbbk}{\lp\mathfrak{a}[-1]\rp}$. In particular, if $\mathfrak{a}\coloneqq \Bbbk[t]^{\operatorname{ab}}$, we have that its Chevalley-Eilenberg algebra is the trivial square-zero extension $\Bbbk\oplus\Bbbk[t][-1]$, while the Chevalley-Eilenberg algebra of its dual Lie algebra is $\Bbbk\oplus\Bbbk\lqd t\rqd^{\vee}.$ 
\end{remark}
\begin{parag}
	Conjectures \ref{conj:main} and \ref{conj:main2} provide a derived analogue of \cref{prop:vogliamogeneralizzare}. Indeed, if $\gfrak$ is a discrete Lie algebra, then the homological mixed graded Chevalley-Eilenberg construction can be identified with the usual Grassmann graded coalgebra $\Sym^{\bullet}_{\Bbbk}(\gfrak[-1])\cong\bigwedge^{\bullet}\gfrak$, endowed with the usual Chevalley-Eilenberg differential, considered as a honest chain complex in homological non-negative degrees in the heart of the mixed graded Postnikov $t$-structure on $\Modmixgr_{\Bbbk}$ (\cite[Theorem $1.3.1$]{pavia1}).  If moreover $\gfrak$ is perfect (hence, finite and projective), then the cohomological mixed graded Chevalley-Eilenberg construction sends $\gfrak$ to the Grassmann graded algebra $\Sym_{\Bbbk}^{\bullet}{\lp\gfrak[1]\rp}^{\vee}\cong\bigwedge^{\bullet}\gfrak^{\vee}$, again with a Chevalley-Eilenberg differential, considered as a honest chain complex in homological non-positive degrees in the heart of the mixed graded Postnikov $t$-structure. 
\end{parag}
\begin{remark}
	In \cite[Proposition $3.6.2$]{CPTVV}, the authors proved \cref{conj:main2}, in terms of formal mixed graded commutative $\Bbbk$-algebras, for the case of discrete Lie algebras whose underlying $\Bbbk$-module is finite and projective.
\end{remark}
The main obstruction to proving \cref{conj:main} lies in the description of the left adjoint $\operatorname{L}_{\CE}$ to the inclusion of mixed graded cocommutative coalgebras of Chevalley-Eilenberg type inside all coaugmented mixed graded cocommutative coalgebras. In the following, we provide some evidence that strongly suggests our conjectures to hold.
\begin{lemman}
	\label{lemma:siftedcolimitsCE}
	The natural \infinity-functor $\Prim\circ \oblv_{\varepsilon}\colon\CEcAlgmixgr\to\Mod_{\Bbbk}$ in the diagram \ref{square:CEpullback}  preserves sifted colimits.
\end{lemman}
\begin{proof}
	The forgetful \infinity-functor $\oblv_{\varepsilon}\colon\CEcAlgmixgr\to\CoCommgraug$ factors through the essential image of the cofree ind-nilpotent cocommutative coalgebra on mixed graded $\Bbbk$-modules of the form $M[-1](1).$ Since such \infinity-functor is fully faithful , the \infinity-functor $\Prim$ can be identified with the composition of the forgetful \infinity-functor $\oblv_{\varepsilon}$ and the inverse equivalence that takes a cofree ind-nilpotent cocommutative graded coalgebra to its primitive elements (i.e., the $\Bbbk$-module of elements in weight $1$). But the composition$$\oblv_{\operatorname{cCcAlg}}\circ \operatorname{coFree}_{\operatorname{CoComm^{aug}}}\circ (-)(-1)\circ [-1]\colon\Mod_{\Bbbk}\longrightarrow\Modgr_{\Bbbk}$$agrees with$$\oblv_{\operatorname{CAlg}}\circ \operatorname{Free}_{\operatorname{Comm^{aug}}}\circ (-)(-1)\circ [-1]\colon\Mod_{\Bbbk}\longrightarrow\Modgr_{\Bbbk},$$because in characteristic $0$ the graded cofree ind-nilpotent cocommutative coalgebra and the graded free commutative algebra on a $\Bbbk$-module are equivalent (\cite[Chapter $6$, Section $4.2$]{studyindag2}). Now, $\oblv_{\operatorname{cCcAlg}}$ is conservative and commutes with all colimits, being a left adjoint. Moreover, $$\oblv_{\operatorname{CAlg}}\circ \operatorname{Free}_{\operatorname{Comm^{aug}}}\circ (-)(-1)\circ [-1]\colon\Mod_{\Bbbk}\to\Modgr_{\Bbbk}$$commutes with sifted colimits, because it is a composition of \infinity-functors which commute with sifted colimits (they are all left adjoint except for $\oblv_{\operatorname{CAlg}}$, which however commutes with sifted colimits in virtue of \cite[Chapter $6$, Section $1.1.3$]{studyindag2}). Finally, the forgetful \infinity-functor $\oblv_{\varepsilon}\colon\CoCommmixgraug\to\CoCommgraug$ commutes with colimits, since the usual forgetful \infinity-functor $\oblv_{\varepsilon}\colon\Modmixgr_{\Bbbk}\to\Mod_{\Bbbk}$ commutes with colimits and colimits of coalgebras are computed in the underlying category. All these ingredients together imply that sifted colimits of mixed graded cocommutative coalgebras of Chevalley-Eilenberg type are again of Chevalley-Eilenberg type, hence our claim follows.
\end{proof}
\begin{porism}
	\label{porism:monad}
	\cref{lemma:siftedcolimitsCE} implies that the \infinity-functor $\Prim\colon\CEcAlgmixgr\to\Mod_{\Bbbk}$ is a monadic \infinity-functor: it commutes with all limits, it commutes with sifted colimits, and it is conservative, hence the conditions of Barr-Beck-Lurie's monadicity theorem (\cite[Theorem $4.7.0.3$]{ha}) are satisfied. Arguing analogously to what we did in the proof of \cref{lemma:CElocalization}, we obtain then that the \infinity-functor $\CE_{\varepsilon}\colon\Lie_{\Bbbk}\to\CoCommgraug$ is monadic as well.\\In particular, \cref{conj:main} can be re-formulated as follows: the two monads over $\Mod_{\Bbbk}$ defined by $\oblv_{\Lie}\colon\Lie_{\Bbbk}\to\Mod_{\Bbbk}$ and $\Prim\circ\oblv_{\varepsilon}\colon\CEcAlgmixgr\to\Mod_{\Bbbk}$ are equivalent.
\end{porism}
\begin{parag}
	\label{parag:trivialgroup}
	Recall (\cite[Chapter $6$, Proposition $1.6.4$]{GR}) that the loop $\infinity$-functor $$\Omega_{\Lie}\colon \Lie\to\GrpLie$$is an equivalence of $\infinity$-categories. Moreover, the diagram
	$$	\begin{tikzpicture}[scale=0.75,baseline=0.5ex]
	\node (a) at (-3,3){$\Lie_{\Bbbk}$};
	\node (b) at (3,3){$\GrpLie$};
	\node (a1) at (-3,0){$\Mod_{\Bbbk}$};
	\node (b1) at (0,0){$\Mod_{\Bbbk}$};
	\node (c1) at (3,0){$\Lie_{\Bbbk}$};
	\draw[->,font=\scriptsize] (a) to node[above]{$\simeq$}(b);
	\draw[->,font=\scriptsize] (a) to node[below]{$\Omega_{\Lie}$}(b);
	\draw[->,font=\scriptsize] (a) to node[left]{$\oblv_{\Lie}$}(a1);
	\draw[->,font=\scriptsize] (b) to node[right]{$\oblv_{\operatorname{Grp}}$} (c1);
	\draw[->,font=\scriptsize] (a1) to node[above]{$\simeq$} (b1);
	\draw[->,font=\scriptsize] (a1) to node[below]{$[-1]$} (b1);
	\draw[->,font=\scriptsize] (b1) to node[below]{$\triv_{\Lie}$} (c1);
	\end{tikzpicture}$$
	commutes (\cite[Chapter $6$, Proposition $1.7.2$]{GR}).
\end{parag}
A similar result applies also to mixed graded cocommutative coalgebras of Chevalley-Eilenberg type.
\begin{lemman}
	\label{lemma:CEcalgebrasaregroupobjects}
	The loop \infinity-functor$$\Omega_{\varepsilon}\colon\CEcAlgmixgr\overset{\simeq}{\longrightarrow}\operatorname{Grp}{\lp\CEcAlgmixgr\rp}$$provides an equivalence of \infinity-categories, with inverse given by delooping. 
\end{lemman}
\begin{proof}
	This is, essentially, a re-writing of the proof of \cite[Proposition $1.6.2$]{studyindag2}. Since the delooping $\mathsf{B}_{\varepsilon}\colon\GrpChev\to\CoCommmixgraug$ is a left adjoint to the loop \infinity-functor $\Omega_{\varepsilon}$, we have a unit morphism$$\operatorname{id}_{\operatorname{Grp}{\lp\CoCommmixgraug\rp}}\longrightarrow\Omega_{\varepsilon}\circ\mathsf{B}_{\varepsilon}$$and a counit morphism$$\mathsf{B}_{\varepsilon}\circ\Omega_{\varepsilon}\longrightarrow\operatorname{id}_{\CoCommmixgraug}.$$In order to prove the equivalence, it suffice to show that these natural transformations are equivalences of \infinity-functors. Since both forgetting the group structure and the mixed graded structure are conservative operations, it suffices to prove that$$\oblv_{\varepsilon}\circ\oblv_{\operatorname{Grp}}\longrightarrow \oblv_{\varepsilon}\circ\oblv_{\operatorname{Grp}}\circ\hsp\Omega_{\varepsilon}\circ\mathsf{B}_{\varepsilon}$$and$$\oblv_{\varepsilon}\circ\hsp \mathsf{B}_{\varepsilon}\circ\Omega_{\varepsilon}\longrightarrow\oblv_{\varepsilon}$$are natural equivalences of graded \textit{cofree} cocommutative coalgebras. Now the \infinity-functor $$\Prim\colon\CoCommgraug\longrightarrow\Mod_{\Bbbk},$$given by considering the $\Bbbk$-module $A_1[1]$ of a graded cocommutative coalgebra $A_{\bullet}$, reflects equivalences of graded cofree cocommutative coalgebras, since $$\Symgr_{\Bbbk}{\lp(-)[-1](1)\rp}\colon \Mod_{\Bbbk}\longrightarrow\CoCommgraug$$is fully faithful in virtue of \cref{prop:gradedsymfullyfaithful}. In particular, we are left to show that the natural transformations$$\Prim\circ\oblv_{\varepsilon}\circ\oblv_{\operatorname{Grp}}\longrightarrow\Prim\circ \oblv_{\varepsilon}\circ\oblv_{\operatorname{Grp}}\circ\hsp\Omega_{\varepsilon}\circ\mathsf{B}_{\varepsilon}$$and$$\Prim\circ\oblv_{\varepsilon}\circ\hsp \mathsf{B}_{\varepsilon}\circ\Omega_{\varepsilon}\longrightarrow\Prim\circ\oblv_{\varepsilon}$$are equivalences. Let us remark that $\mathsf{B}_{\varepsilon}$ is, by definition, a sifted colimit of mixed graded cocommutative Chevalley-Eilenberg coalgebras. Thanks to \cref{lemma:siftedcolimitsCE} we have that the following square$$\begin{tikzpicture}[scale=0.75]
	\node (b) at (3,3){$\CEcAlgmixgr$};\node (a) at (-3,3){$\operatorname{Grp}{\lp\CEcAlgmixgr\rp}$};
	\node (c) at (-3,0){$\Mod_{\Bbbk}$};\node (d) at (3,0){$\Mod_{\Bbbk}$};
	\draw[->,font=\scriptsize] (a) to node[above]{$\mathsf{B}_{\varepsilon}$}(b);
	\draw[->,font=\scriptsize] (a) to node[left]{$\Prim\circ\oblv_{\varepsilon}\circ\oblv_{\operatorname{Grp}}$}(c);
	\draw[->,font=\scriptsize] (b) to node[right]{$\Prim\circ\oblv_{\varepsilon}$}(d);
	\draw[->,font=\scriptsize] (c) to node[below]{$[1]$}(d);
	\draw[->,font=\scriptsize] (c) to node[above]{$\simeq$}(d);
	\end{tikzpicture}$$
	commutes, since $\Prim$ commutes with sifted colimits.\\On the other hand, since $\oblv_{\varepsilon}\colon\Modmixgr_{\Bbbk}\to\Modgr_{\Bbbk}$ is strongly monoidal and preserves all colimits, hence geometric realizations, it preserves also relative tensor products. In particular,$$\oblv_{\varepsilon}\colon\CoCommmixgraug\longrightarrow\CoCommgraug$$preserves pullbacks, since they are given by the relative tensor product: this follows from the fact that $\operatorname{cCcAlg}{\lp\scrC\rp}\simeq\CAlg(\scrC^{\op})^{\op}$ and pushouts of commutative algebras are given by the relative tensor product, in virtue of \cite[Proposition $3.2.4.7$]{ha}. Therefore, we have another commutative square$$\begin{tikzpicture}[scale=0.75]
	\node (a) at (-3,3){$\CEcAlgmixgr$};\node (b) at (3,3){$\operatorname{Grp}{\lp\CEcAlgmixgr\rp}$};
	\node (c) at (-3,0){$\Mod_{\Bbbk}$};\node (d) at (3,0){$\Mod_{\Bbbk}.$};
	\draw[->,font=\scriptsize] (a) to node[above]{$\Omega_{\varepsilon}$}(b);
	\draw[->,font=\scriptsize] (a) to node[left]{$\Prim\circ\oblv_{\varepsilon}$}(c);
	\draw[->,font=\scriptsize] (b) to node[right]{$\Prim\circ\oblv_{\varepsilon}\circ\oblv_{\operatorname{Grp}}$}(d);
	\draw[->,font=\scriptsize] (c) to node[below]{$[-1]$}(d);
	\draw[->,font=\scriptsize] (c) to node[above]{$\simeq$}(d);
	\end{tikzpicture}$$Then our claim follows from the fact that looping and delooping are one inverse to the other in $\Mod_{\Bbbk}$.
\end{proof}
\cref{lemma:CEcalgebrasaregroupobjects} states that the \infinity-category $\CEcAlgmixgr$, which is the \infinity-category of algebras for some monad in virtue of \cref{porism:monad}, shares another important feature with the \infinity-categories of algebras for the monad defined by an operad $\mathscr{O}.$ Namely, every group object admits an essentially unique delooping. 
\begin{parag}
	As already observed (\ref{parag:CEpreserveslimits}), $\CE_{\varepsilon}\colon\Lie_{\Bbbk}\longrightarrow\CEcAlgmixgr$ preserves limits and in particular products.  Hence, it preserves group objects, and so lifts to an $\infinity$-functor\begin{align*}
	\CE_{\varepsilon}\colon\GrpLie\longrightarrow{\operatorname{Grp}}{\lp\CEcAlgmixgr\rp}.
	\end{align*}Moreover, being the Tate realization $\infinity$-functor strongly monoidal when restricted to non-negatively graded $\Bbbk$-modules, it lifts to the \infinity-category of cocommutative coalgebras and there it respects products, so we have also an \infinity-functor landing in the \infinity-category of cocommutative bialgebras\begin{align}
	\label{functor:GrpChev}
	\left|-\right|^{\operatorname{t}}\colon{\operatorname{Grp}}{\lp\CEcAlgmixgr\rp}\longrightarrow{\operatorname{Grp}}{\lp\operatorname{cCcAlg}_{\Bbbk//\Bbbk}.\rp}\subseteq{\operatorname{Alg}}{\lp\operatorname{cCcAlg}_{\Bbbk//\Bbbk}\rp}\eqqcolon\operatorname{cCBAlg}_{\Bbbk//\Bbbk}.
	\end{align}It is clear that \cref{prop:CEexplicitmodel} implies that the composition$$\left|-\right|^{\operatorname{t}}\circ\CE_{\varepsilon}\colon{\operatorname{Grp}}{\lp\Lie_{\Bbbk}\rp}\longrightarrow{\operatorname{Grp}}{\lp\operatorname{cCcAlg}_{\Bbbk//\Bbbk}\rp}$$is equivalent to the \infinity-functor $\GrpChev$ of \cite[Chapter $6$, Section $4.3$]{studyindag2}, which in turn is equivalent to the fully faithful \infinity-functor of the universal enveloping algebra with its Hopf structure (\cite[Chapter $6$, Theorem $6.1.2$]{studyindag2}). We can restate \cref{conj:main} in the following way.
\end{parag}
\begin{conjn}
	The \infinity-functor \ref{functor:GrpChev} is fully faithful.
\end{conjn}
The advantage of this perspective to the problem is that the Lie structure of any group Lie algebra is canonically trivial (\cite[Corollary $1.7.3$]{studyindag2}). This allows to reduce the proof of the fully faithfulness to trivial Lie algebras, which are sent to mixed graded modules with trivial mixed structure in virtue of \cref{corollary:trivliealgebrasaretrivcoalgebras}. However, it is not guaranteed that $\mathfrak{L}$ sends mixed graded Chevalley-Eilenberg coalgebras with trivial mixed structure to abelian Lie algebras.
\subsection{Representations as mixed graded modules}
\label{sec:conjrep}
Analogously to the claim of \cref{conj:main}, we can expect representations of Lie algebras to fully faithfully embed into some precise \infinity-category of mixed graded $\Bbbk$-modules, assuming a sufficiently well behaviour of our Lie algebra..\\For this section, we shall fix a Lie algebra $\gfrak$, perfect as a $\Bbbk$-module. As already stated in \cref{prop:CEmod}, the \infinity-functor $\CE^{\varepsilon}(\gfrak;\hsp-)\colon\LMod_{\Ug}\to\Modmixgr_{\Bbbk}$ factors through the \infinity-category of mixed graded $\CE^{\varepsilon}$-modules $\ModCEmixgr$. As a purely graded $\Bbbk$-module, $\CE^{\varepsilon}(\gfrak;\hsp M)$ is equivalent to the graded $\Bbbk$-linear dual $$\oblv_{\varepsilon}\CE^{\varepsilon}(\gfrak;\hsp M)\simeq\Mapin^{\gr}_{\Bbbk}{\lp \oblv_{\varepsilon}\CE_{\varepsilon}(\gfrak),\hsp M(0)\rp}$$ in virtue of the discussion of \cref{parag:CEmod}. Since $\gfrak$ is perfect, and $\oblv_{\varepsilon}\CE_{\varepsilon}(\gfrak)\simeq\Symgr_{\Bbbk}{\lp\oblv_{\Lie}\gfrak[-1](1)\rp}$, we can equivalently write$$\CE^{\varepsilon}(\gfrak;\hsp M)\simeq\Symgr_{\Bbbk}{\lp\gfrak^{\vee}[1](-1)\rp}\otimes^{\gr}_{\Bbbk}M(0).$$This motivates the following definition.
\begin{defn}
	\label{def:constantmodules}
	We say that a mixed graded $\CE^{\varepsilon}(\gfrak)$-module $M_{\bullet}$ is \textit{constant} if, as a graded $\Bbbk$-module, is equivalent to $$\Symgr_{\Bbbk}{\lp\gfrak^{\vee}[1](1)\rp}\otimes^{\gr}_{\Bbbk}M(0)$$for some $\Bbbk$-module $M$.
\end{defn}
Constant mixed graded $\CE^{\varepsilon}(\gfrak)$-modules are naturally gathered in a full sub-\infinity-category of $\ModCEmixgr$, which we denote by $\ModCEmixgrconst$. 
\begin{remark}
	It is clear, in virtue of the discussion above, that the Chevalley-Eilenberg $\infinity$-functor $\CE^{\varepsilon}\colon\LMod_{\Ug}\to\ModCEmixgr$ factors through $\ModCEmixgrconst.$
\end{remark}
The main conjecture concerning left $\Ug$-modules, for $\gfrak$ perfect as a $\Bbbk$-module, is the following.
\begin{conjn}
	\label{conj:modules}
	Under our assumptions, the Chevalley-Eilenberg \infinity-functor $$\CE^{\varepsilon}(\gfrak;\hsp-)\colon\LMod_{\Ug}\longrightarrow\ModCEmixgrconst$$is fully faithful. 
\end{conjn}
\begin{parag}
	Similarly to the case of Lie algebras and mixed graded cocommutative coalgebras, \cref{prop:CEalllimitscolimits} implies that the Chevalley-Eilenberg \infinity-functor $$\CE_{\varepsilon}\colon\LMod_{\Ug}\longrightarrow\ModCEmixgr$$admits both a left and right adjoint if $\gfrak$ is perfect. We suspect that it is the \textit{right} adjoint, that usually does not exist without the finitenss assumption on $\gfrak$, that realizes the quasi-inverse of the equivalence. Indeed, if $A_{\bullet}$ is a mixed graded commutative $\Bbbk$-algebra with some size constraint (e.g., it is bounded in non-positive weights), we can always consider both the \infinity-category of mixed graded $A_{\bullet}$-modules and the full sub-\infinity-category $\Mod^{\operatorname{const}}_{A_{\bullet}}{\lp\Modmixgr_{\Bbbk}\rp}$ of \textit{constant} mixed graded $A_{\bullet}$-modules, that is the \infinity-category spanned by mixed graded $A_{\bullet}$-modules such that the map$$A_q\otimes_{A_0}M_0 \longrightarrow M_q$$is an equivalence in each weight. This sub-\infinity-category is always a colocalization of $\Modmixgr_{\Bbbk}$: it is a presentable \infinity-category and, since the tensor product commutes with colimits separably in each variable (and they are computed as in $\Modmixgr_{\Bbbk}$, i.e., weight-wise), it is also closed under colimits. In particular, in our case, the left adjoint should correspond to the colocalization $\infinity$-functor$$\ModCEmixgr\longrightarrow\ModCEmixgrconst$$followed by the inverse equivalence $\ModCEmixgrconst\simeq\LMod_{\Ug}$.
\end{parag}
\begin{remark}
	\label{remark:tstructure}
	Given a mixed graded $\Bbbk$-algebra $A_{\bullet}$, the $\infinity$-category of mixed graded $A_{\bullet}$-modules $\Mod_{A_{\bullet}}{\lp\Modmixgr_{\Bbbk}\rp}$ is endowed itself with a left complete Postnikov $t$-structure, where the connective part is spanned by those mixed graded $A_{\bullet}$-modules which are connective for the mixed graded Postnikov $t$-structure of \cite[Theorem $1.3.1$]{pavia1}. Similarly, one can define a $t$-structure on $\Mod^{\operatorname{const}}_{A_{\bullet}}{\lp\Modmixgr_{\Bbbk}\rp}$ whose connective part is spanned by those mixed graded $A_{\bullet}$-modules which are connective for the mixed graded Postnikov $t$-structure. If $\gfrak$ is perfect \textit{and coconnective as a $\Bbbk$-module}, then it is clear that$$\CE^{\varepsilon}(\gfrak;\hsp-)\colon\LMod_{\Ug}\longrightarrow\ModCEmixgrconst$$is a $t$-exact \infinity-functor, where on the source we consider the $t$-structure described in \cite[Warning $3.5.9$]{dagx}. Indeed, if $\gfrak$ is coconnective, then its dual $\gfrak^{\vee}$ is connective and each symmetric power $\Sym^p_{\Bbbk}{\lp\gfrak^{\vee}[1]\rp}$ is $p$-connective. Therefore, given a left $\Ug$-module $M$ which is connective as a $\Bbbk$-module, we have$$\CE^{-p}{\lp\gfrak;\hsp M\rp}\simeq \Sym^p_{\Bbbk}{\lp\oblv_{\Lie}\gfrak^{\vee}[1]\rp}\otimes_{\Bbbk}M$$which is again $p$-connective, since the $t$-structure on $\Bbbk$-modules is compatible with the monoidal structure (\cite[Lemma $7.1.3.10$]{ha}).\end{remark}
\begin{remark}
	It is clear that these $t$-structures on both $\LMod_{\Ug}$ and $\ModCEmixgrconst$ described in \cref{remark:tstructure} are left complete, because limits and connective objects are detected by the forgetful \infinity-functors $\LMod_{\Ug}\to\Mod_{\Bbbk}$ and $\ModCEmixgrconst\to\Modmixgr_{\Bbbk}$, and so one can reduce to check the assertion of \cite[Proposition $1.2.1.19$]{ha} at the level of the underlying $\Bbbk$-modules and mixed graded $\Bbbk$-modules, respectively. In particular, if \cref{conj:modules} is true, the $t$-structures on $\LMod_{\Ug}$ and $\ModCEmixgrconst$ should correspond. This is linked to deformation theory of formal moduli problems. Namely, let $X$ be the formal moduli problem that corresponds to $\gfrak$ under the equivalence between homotopy Lie algebras and formal moduli problems over fields of characteristic $0$ (\cite[Theorem $2.0.2$]{dagx}). In virtue of \cite[Theorem $2.4.1$]{dagx}, there exists a monoidal fully faithful embedding $$\Qcoh(X)\longhookrightarrow\Rep_{\gfrak}\simeq\LMod_{\Ug}$$of the \infinity-category of quasi-coherent sheaves over $X$ in the \infinity-category of representations of $\gfrak$, which preserves connective objects. Moreover, $\Rep_{\gfrak}$ can be described geometrically as the $\infinity$-category of Ind-coherent sheaves $\IndCoh(X)$ (\cite[Theorem $3.5.1$]{dagx}): under this equivalence the $t$-structure of $\Rep_{\gfrak}$ corresponds to the canonical $t$-structure of $\IndCoh(X)$ characterized by the property that the $t$-structure is compatible with filtered colimits, and the canonical inclusion $\Coh(X)\hookrightarrow\IndCoh(X)$ is $t$-exact (\cite[Section $1.2.1$]{indcoh}). By Ind-extending the inclusion $\Coh(X)\hookrightarrow\Qcoh(X)$, one has an $\infinity$-functor \begin{align}
	\label{functor:indcohleftcompletion}
	\Phi_X\colon\IndCoh(X)\longrightarrow \Qcoh(X)
	\end{align}
	which exhibits $\Qcoh(X)$ with its $t$-structure described above as the left completion of the $t$-structure of $\IndCoh(X)$ (\cite[Proposition $1.3.4$]{indcoh}). In particular, \cref{conj:modules} would allow to recover quasi-coherent modules over the formal moduli problem $X$ associated to a perfect coconnective Lie algebra $\gfrak$ in the \infinity-category $\ModCEmixgrconst$.
\end{remark}
    \printbibliography
\end{document}